\documentclass[a4paper,12pt]{article}
\usepackage{amsmath,amsthm,amsfonts,amssymb,bbm}
\usepackage{graphicx,psfrag,subfigure,color,cite}
\usepackage[UKenglish]{babel}
\usepackage{textcmds}
\usepackage{mathtools}
\usepackage{float}
\usepackage{changes}
\usepackage[colorlinks]{hyperref}

\numberwithin{equation}{section}

\usepackage[margin=1.2in]{geometry}

\newcommand{\e}{\varepsilon}
\newcommand{\Pb}{\mathbb{P}}

\newcommand{\I}{\mathrm{i}}

\newcommand{\R}{\mathbb{R}}
\newcommand{\N}{\mathbb{N}}
\newcommand{\Z}{\mathbb{Z}}
\newcommand{\C}{\mathbb{C}}
\newcommand{\Id}{\mathbbm{1}}
\newcommand{\Ai}{\mathrm{Ai}}
\renewcommand{\R}{\mathbb{R}}
\renewcommand{\Im}{\operatorname{Im}}
\newcommand{\Tr}{\operatorname{Tr}}

\newcommand{\Or}{{\cal O}}
\newcommand{\sgn}[2]{{\cal S}^{#1}_{#2}}

\newtheorem{prop}{Proposition}[section]

\newtheorem{thm}[prop]{Theorem}
\newtheorem{lem}[prop]{Lemma}

\newtheorem{cor}[prop]{Corollary}

\newtheorem{cla}[prop]{Claim}

\newtheorem{rem}[prop]{Remark}

\title{Exact decay of the persistence\\ probability in the Airy$_1$ process}

\author{Patrik L.\ Ferrari\thanks{Institute for Applied Mathematics, Bonn University, Endenicher Allee 60, 53115 Bonn, Germany. E-mail: {\tt ferrari@uni-bonn.de}} \and
Min Liu\thanks{Institute for Applied Mathematics, Bonn University, Endenicher Allee 60, 53115 Bonn, Germany. E-mail: {\tt s6miliuu@uni-bonn.de}}
}

\date{14th September 2024}

\begin{document}
\maketitle
\sloppy

\begin{abstract}
We consider the Airy$_1$ process, which is the limit process in KPZ growth models with flat and non-random initial conditions. We study the persistence probability, namely the probability that the process stays below a given threshold $c$ for a time span of length $L$. This is expected to decay as $e^{-\kappa(c) L}$. We determine an analytic expression for $\kappa(c)$ for all $c\geq 3/2$ starting with the continuum statistics formula for the persistence probability. As the formula is analytic only for $c>0$, we determine an analytic continuation of $\kappa(c)$ and numerically verify the validity for $c<0$ as well.
\end{abstract}

\section{Introduction and main results}
For stochastic growth models in the Kardar-Parisi-Zhang (KPZ) universality class, the large time limit process of the interface depends on the initial and boundary conditions. In the one-dimensional case, several processes are known. When the limit shape is curved, the limit process is the Airy$_2$ process~\cite{PS02,Jo03b} (see also~\cite{Dim20,QS20} for non-determinantal models). On the other hand, when the limit shape is flat and the initial condition is non-random one observes the Airy$_1$ process~\cite{Sas05,BFPS06,BFP06} (see also~\cite{QS20,Vir20} for convergence to the KPZ fixed point~\cite{MQR17} for non-determinantal models). This is still the case for random initial conditions, provided that the initial height function under diffusive scaling goes to zero as first discussed by Quastel and Remenik in~\cite{QR16}.

In this paper we focus on the Airy$_1$ process, ${\cal A}_1$, discovered by Sasamoto in~\cite{Sas05}. The one-point distribution is given by~\cite{FS05b,BR99b}
\begin{equation}
    \Pb({\cal A}_1(t)\leq s)=F_1(2s),
\end{equation}
where $F_1$ is GOE Tracy-Widom distribution \cite{TW96}, while the $m$-point joint distribution is given by a Fredholm determinant (see~\cite{Sas05,BFPS06} for explicit expressions). In this paper we study another observable of the Airy$_1$ process, namely the persistence probability, which is the probability that the Airy$_1$ process stays below a given threshold over a time span $[0,L]$. One expects that
\begin{equation}\label{eq1.2}
P(c,L)=\Pb({\cal A}_1(s)\leq s\textrm{ for all }s\in [0,L])\sim C e^{-\kappa(c) L}
\end{equation}
for large $L$. Numerical computations using the method in~\cite{Born08} for small values of $L$ indicates that the exponential form is quite accurate already for small values of $L$~\cite{FF12a}. This is probably due to the fast (super-exponential) decay of the correlation of the Airy$_1$ process as noticed first numerically in~\cite{BFP08} and recently proven in~\cite{BBF22}.

The starting point is of our analysis is the continuum statistics formula of the probability in \eqref{eq1.2} obtained by Quastel and Remenik~\cite{QR12} (see Theorem~\ref{qr12} below). For the Airy$_2$ process such a formulation was obtained by Corwin, Quastel and Remenik in~\cite{CQR11}, where they started by the expression of the joint distribution as a Fredholm determinant on a fixed space as in original paper of Pr\"ahofer and Spohn~\cite{PS02} (this is referred as path integral formula), see also~\cite{BCR13,MQR17} for a general scheme to connect the two representations for other limit processes in the KPZ class. The formula for the joint distribution in terms of an extended kernel follows from a biorthogonalization procedure~\cite{BFPS06}, which could be made explicit in~\cite{MQR17}, see also~\cite{MR23} for extensions. The continuum statistics occurred to be very useful to determine properties of the Airy processes~\cite{QR12,QR12b,QR13b,BKS12}, but also in discrete analogues~\cite{MR23,FV19}.

In 2010, Takeuchi and Sano were able to verify experimentally the KPZ predictions in an experiment with turbulent nematic liquid crystals~\cite{TS10,TSSS11}, in particular for the distribution functions and covariances. In~\cite{TS12}, they also measure the $\kappa(c)$ with respect to the threshold given by the average of the process.  Later, applying the numerical method in~\cite{Born08} on continuum statistics of Airy$_1$, Ferrari and Frings~\cite{FF12a}  numerically computed $\kappa(c)$ for more general $c$, but they could not provide any analytic results for $\kappa(c)$. The main result of this work is an analytic formula for $\kappa(c)$ which is the following theorem.
\begin{thm}\label{main}
    For any $c\geq\frac32$ and $L$ large enough, it holds
    \begin{equation}\label{kappac}
        \Pb\left(\mathcal A_1(s)\leq c,\ s\in[0,L]\right)= Ce^{-\kappa(c)L+\mathcal O(e^{-L})},
    \end{equation}
    where $C$ does not depend on $L$ and
    \begin{equation}\label{k}
        \kappa(c)=-2\sum_{n=1}^\infty n^{-5/3}\Ai'(2n^{2/3}c),
    \end{equation}
where $\Ai'$ is the derivative of the Airy function $\Ai$.
\end{thm}

Using the fact that the Airy$_1$ process is a limit of the last passage percolation, where a FKG inequality can be applied, we can also show that $\kappa(c)$ exists for all values of $c$.
\begin{prop}[Existence of $\kappa(c)$]\label{p12}
For any $c\in\R$,
\begin{equation}\label{hatkappa}
    \kappa(c)=-\lim_{L\to\infty}\frac{\ln\left(\Pb\left(\mathcal A_1(t) \leq c, \forall t \in[0, L]\right)\right)}{L}
\end{equation}
exists.
\end{prop}

The lower bound on $c$ in Theorem~\ref{main} is purely technical and it could potentially be slightly improved with the approach of this paper, however not below $c=0$. Thus we did not pursue this aspect. The formula we obtain is analytic for all $c>0$, but not at $0$ or below. Denoting by $\tilde\kappa$ the analytic continuation of \eqref{k}, we have the following result.
\begin{prop}[Analytic continuation of $\kappa(c)$]\label{p13}
The analytic continuation is given by
    \begin{equation}\label{tildekappas}
        \tilde{\kappa}(c)= \begin{cases}\kappa(c), & \text { if } c \geq 0, \\ \kappa(0)-\int_c^0 d x f(x)-6 c-\frac{48}{7} \sum_{n \geq 1}(c-c(n)) \mathbf{1}_{c<c(n)}, & \text { if } c<0,\end{cases}
    \end{equation}
where $c(n)=-(2n\pi/3)^{2/3}$ for any $n\in\Z_{\geq 1}$, $\kappa(c)$ is define in \eqref{k} and
    \begin{equation}
        f(x)=\frac{2}{\pi \I}\int_{\Gamma}dw\frac{w^2e^{\frac{w^3}{3}-2wx}}{1-e^{\frac{w^3}{3}-2wx}}
    \end{equation}
with $\Gamma=\{|r|e^{{\rm sgn}(r)\pi \I/3}| r\in\R\}$ oriented with increasing imaginary part\footnote{In principle one can integrate explicitly $\int_c^0 dx f(x)$ and get a logarithm, see also \eqref{eq4.2}. However, if one is not careful with the branch-cut of the logarithm, a numerical evaluation with Mathematica or similar programs leads to a non-smooth plot, unlike \eqref{tildekappas}.}.
\end{prop}

Since we do not have an analytic proof that $\kappa(c)$ is analytic, we test whether the analytic continuation $\tilde\kappa(c)$ fits with the data obtained by numerical computations. The persistence probability is given by a Fredholm determinant with a kernel $K_{L,c}$, see Proposition~\ref{ff13} below. As mentioned in~\cite{FF12a} the kernel $K_{L,c}$  does not behaves well for large $L$: there are some off-diagonal entries which diverges super-exponentially in $L$. On top of it, some parts of the entries are highly oscillating. These two effects restrict very much the numerical implementation of the Fredholm determinant computation of~\cite{Born08}, namely the values of $L$ which can be simulated is (depending on the values of $c$) usually not more than $L=3$. On the other hand, probably due to the fast decorrelation decay of the Airy$_1$ process~\cite{BBF22}, already for small values of $L$ the logarithm of the persistence probability is already almost a perfect straight line, see for example Figure~\ref{numeric}.
\begin{figure}[h!]
    \centering
    \includegraphics[width=0.5\textwidth]{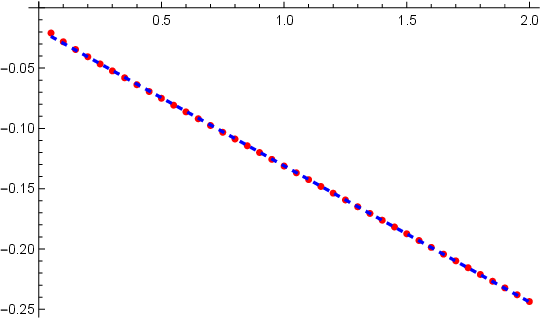}
    \caption{The red points are $(L_n,\log P(c,L_n))$ with $c=1$ and $L_n=0.05n$ with $n\in\{1,2,\ldots,40\}$, where $P(c,L_n)$ is calculated numerically. The dashed blue line is the reference line with slope -0.112, we refer the slope obtained in this way as $\hat\kappa(c)$.}
    \label{numeric}
\end{figure}

In~\cite{FF12a} it was derived that, for $c\in\R$,
\begin{equation}
  P(c,L)=\det\left(\Id-B_0+\Lambda_{L, c} e^{-L \Delta} B_0\right)_{L^2(\R)},
\end{equation}
see Proposition~\ref{ff13} below for details. We numerically compute $P(c,L_n)$ for $L_n=0.05n$ with $n\in\{1,2,\ldots,40\}$. Interpolating the obtained data $(L_n,\log P(c,L_n))$, we get a numerical estimate $\hat\kappa(c)$ for the persistence exponent $\kappa(c)$, see Figure~\ref{numeric} for $c=1$. For more values of $\hat\kappa(c)$ and details on the numerical issues regarding to the calculation of persistence probability, we refer the reader to Section~4 in~\cite{FF12a}.

In Figure~\ref{minus2} we compare $\hat\kappa(c)$ and analytic continuation of the persistence exponent $\tilde\kappa(c)$. The result indicates that the analytic continuation is likely to be the correct function. The numerical data are available as the \emph{bonndata} repository~\cite{FK2/ANX3PQ_2024}.
\begin{figure}[h!]
    \centering
    \includegraphics[width=0.5\textwidth]{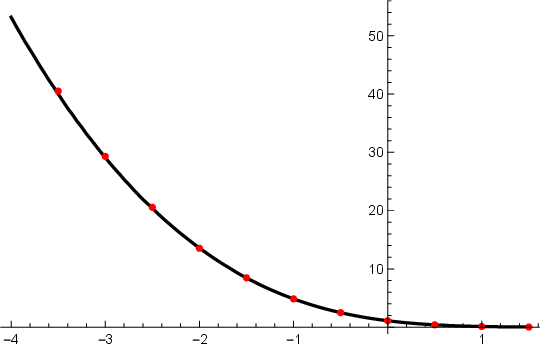}
    \caption{Comparison between numerical and theoretical exponent. The black line is the graph of $(c,\tilde\kappa(c))$ with $c\in(-4,1.5)$ and the red points are the $(c,\hat\kappa(c))$ with $c\in\{-3.5,-3,\ldots,1.5\}$ obtained by numerical calculation. We can see that the experimental data fits the theoretic data very well.}
    \label{minus2}
\end{figure}

\paragraph{Acknowledgments:} We are grateful to A. Borodin for discussions on the analytic continuation. We would like to thank the anonymous referee for the careful reading and the valuable feedback. The work was partly funded by the Deutsche Forschungsgemeinschaft (DFG, German Research Foundation) under Germany’s Excellence Strategy - GZ 2047/1, projekt-id 390685813 and by the Deutsche Forschungsgemeinschaft (DFG, German Research Foundation) - Projektnummer 211504053 - SFB 1060.

\section{Strategy and proof of Theorem~\ref{main}}
The starting point of the analysis is the result on the continuum statistics by Quastel and Remenik~\cite{QR12}.
\begin{thm}[Theorem~4 of~\cite{QR12}]\label{qr12}
It holds
\begin{equation}
\Pb\left(\mathcal{A}_1(t) \leq g(t), 0 \leq t \leq L\right)=\det\left(\Id-B_0+\Lambda_{L, g} e^{-L \Delta} B_0\right)_{L^2(\R)}
\end{equation}
where $g$ is a function in $H^1([0, L])$ (that is, both $g$ and its derivative are in $L^2([0,L])$), $\Delta$ is the Laplacian, $B_0(x, y)=\mathrm{Ai}(x+y)$, and
\begin{equation}
\Lambda_{L, g}(x, y)=\frac{e^{-(x-y)^2 /(4 L)}}{\sqrt{4 \pi L}} \Pb_{b(0)=x, b(L)=y}(b(s) \leq g(s), 0 \leq s \leq L)
\end{equation}
with $b$ a Brownian Bridge from $x$ at time 0 to $y$ at time $L$ and with diffusion coefficient 2.
\end{thm}
In order to state the persistence probability of a constant threshold $c$, we denote by $\bar P_0$ the projection onto the interval $(-\infty,0]$ and $P_0=\Id-\bar P_0$, furthermore, we define a new operator $e^{L\tilde\Delta}$ with kernel
\begin{equation}\label{eldeltat}
	e^{L\tilde\Delta}(x,y)=e^{L\Delta}(-x,y).
\end{equation}
 For a constant threshold $c$, this was computed in~\cite{FF12a} with the following result.
\begin{prop}[Proposition 2.1 of~\cite{FF12a}]\label{ff13}
   For $c\in\R$ and $L>0$, it holds
   \begin{equation}\label{fredholm}
       \Pb(\mathcal A_1(s)\leq c,\ s\in[0,L])=\det\left(\Id-K_{L,c}\right)_{L^2(\R)},
   \end{equation}
where
\begin{equation}\label{eq2.5}
	K_{L,c}=P_0B_{0,c}+\tilde K_{L,c}+\hat K_{L,c}
\end{equation}
 with
 \begin{equation}
 	B_{0,c}(x,y)=\Ai(x+y+2c)
 \end{equation}
 and
\begin{equation}\label{klcdef}
    \begin{aligned}
    &\tilde K_{L,c}=\bar P_0e^{L\Delta}P_0e^{-L\Delta}B_{0,c},\quad\hat K_{L,c}=\bar P_0e^{L\tilde\Delta}\bar P_0e^{-L\Delta}B_{0,c}.
\end{aligned}
\end{equation}
\end{prop}
Note that although the coefficient in front of Laplacian is negative, the operator $e^{-L\Delta}B_{0,c}$ is well-defined with kernel given by (see for instance \cite{Sas07,QR12,BFPS06,BFP06})
\begin{gather}\label{eldelta}
e^{-L\Delta}B_{0,c}(x,y)=e^{-2L^3/3}e^{-L(x+y+2c)}\Ai(L^2+x+y+2c).
\end{gather}
For later use, we also set
\begin{equation}
\tilde B_{0,c}(x,y)=\Ai(y-x+2c),\quad \hat B_{0,c}(x,y)=\Ai(x-y+2c).
\end{equation}

Since the probability we are interested in goes to $0$ as $L\to\infty$, the Fredholm determinant goes to $0$ and as usual in these cases is the Fredholm series expansion not a good representation for the analysis. Instead, we are considering directly the logarithm of the persistence probability and use the trace expansion, namely
\begin{equation}\label{logdet}
    \ln(\det(\Id-K_{L,c}))=-\sum_{n=1}^\infty\frac{\Tr(K_{L,c}^n)}{n}.
\end{equation}

The strategy is to single out the terms in \eqref{logdet} which are linear in $L$ (their sum will give $-\kappa(c)$) and control all other terms. These terms will be bounded by using different norms after having multiplied by appropriate conjugations. These will be given by multiplication operators $U_r: L^2(\R)\to L^2(\R)$ with $U_r f(x)=e^{rx}f(x)$ for $r>0$. The following standard results, see e.g.~\cite{Sim00}, will be constantly used throughout this work.
\begin{thm}\label{simon}\label{operatorinequality}
    Let $A,B$ be two operators. Then it holds
    \begin{enumerate}
    \item\label{ts1} $\Tr(U_r^{-1}AU_r)=\Tr(A)$,
    \item\label{ts2} $|\Tr(A)|\leq\|A\|_1$,
    \item\label{ts3}$ \max\{\|AB\|_1,\|AB\|_{\rm HS}\}\leq\|A\|_{\rm HS}\|B\|_{\rm HS}$,
    \item $\|AB\|_{\rm HS}\leq \|A\|_{\rm HS}\|B\|_{\rm op}$,
    \item \label{ts4} $\|A\|_{\rm op}\leq\|A\|_{\rm HS}\leq\|A\|_{1}$,
    \item $\Tr(AB)=\Tr(BA)$,
    \end{enumerate}
    whenever the r.h.s.\ are well-defined.
\end{thm}
\paragraph{Strategy of the proof.} The main idea of the proof of Theorem~\ref{main} is as follows. For each $n\geq 1$, we calculate $\Tr(K_{L,c}^n)$ with $K_{L,c}$ given in Proposition~\ref{ff13}. The resulting terms can be divided into three categories: the ones independent of $L$ (for instance $\Tr((P_0B_{0,c})^n$), the ones providing $\mathcal O(L)$ term (which comes from $\Tr(\hat K_{L,c}^n)$) and the rest error terms.

To illustrate this idea, we will first consider the $n=1$ case, see Section~\ref{n=1}. In this case, the trace is given as a sum of three terms according to the decomposition \eqref{eq2.5}. The contribution of the first two terms is easy to control and they are not growing in $L$. The last term, $\Tr(\hat K_{L,c})$ can be written as a double integral in the complex plane, see~\eqref{eq2.18}. To analyze the large $L$ behavior, we would like to do steep descent on this double integral. This is possibly only after exchanging the position of the contours, which generates a residue term. This is a single integral which can be explicitly computed and it is proportional to $L$. The double integral from the steep descent analysis is very small. This exchange of contours is equivalent to replace a $\bar P_0$ with a $1-P_0$.

For the general $n\geq 2$, we will mostly do manipulations directly on the operators. We will constantly replace $\bar P_0$ appearing in $K_{L,c}^n$ by $1-P_0$ so that we can use the following ingredients:
	\begin{enumerate}
		\item various identities involving Airy function, for instance $e^{L\Delta}e^{-L\Delta}B_{0,c}=B_{0,c}$ for any $L>0$ (see Lemma~\ref{identity} for more) and invariance of trace under circular shifts to simplify the kernels;
		\item Cauchy's residue theorem to deduce the $\mathcal O(L)$ term from the simplified kernels, see Lemma~\ref{wkey3};
		\item\label{int2} super-exponential decay of Airy function on positive line (which gives us Lemma~\ref{bound},~\ref{cobound}) and inequalities of Theorem~\ref{operatorinequality} to control the $L-$independent terms and rest error terms (the main part is given in Appendix Section~\ref{upperbounds}).
	\end{enumerate}
In Section~\ref{sec23} we consider the general case for $n\geq 2$. The kernel $K_{L,c}$ is further decomposed, see \eqref{e226}, as
\begin{equation}
	K_{L,c}=K_u+K_v+K_w+K_d-K_e.
\end{equation}
Thus $K_{L,c}^n$ can be written as a sum of products of those five kernels. Essentially, there are two types of those products: single terms, that is, those of the form $K_i^n$ with $i\in\{u,v,w,d,e\}$ and mixed terms. Using Cauchy's residue theorem, we will see that only $\Tr(K_w^n)$ provides $\Or(L)$ term and all other terms can be controlled by  ingredient~\ref{int2} mentioned above.

Combining all the results, we prove Theorem~\ref{main} in Section~\ref{sec24}. The analytic continuation of $\kappa(c)$ is obtained in Section~\ref{pp13}.

Throughout this work we will consider $1\leq r^2\leq 2c$ with $r>0$ and also define
\begin{equation}\label{beta}
	\beta= \max\{2e^{r^3/3-2rc},e^{(r-1/7)^3/3-2(r-1/7)c}\}.
\end{equation}
\begin{rem}\label{obstica}
	We will see that the absolute value of sum of all error terms is bounded by $\sum_{n=1}^\infty\frac{7^n\beta^n}{n}$. If we set $r=\sqrt{2c}$, then $\beta\leq 2e^{-\frac{4\sqrt{2}c^{3/2}}{3}}<\frac{1}{7}$ for $c\geq \frac32$, which explains why we choose $c\geq\frac32$ in Theorem~\ref{main}.
	
	The term $e^{(r-1/7)^3/3-2(r-1/7)c}$ and the prefactor 2 are purely technical. With a more delicate method, one can improve slightly this term, but it will not reduce the lower bound $c\geq \frac32$ significantly, so we will not pursue in this aspect. For the estimate in the proof, the term $e^{r^3/3-2rc}$ is and the restriction $0<r^2\leq 2c$ are essentially. They are the main obstacle to generalizing our method to negative real number.
\end{rem}

\subsection{Case $n=1$}\label{n=1}
In order to illustrate the idea, let's first consider the $\Tr(K_{L,c}^n)$ with $n=1$ and $L\geq 1$. We have
\begin{equation}\label{eq2.13}
\Tr(K_{L,c})= \Tr(P_0B_{0,c})+\Tr(\tilde K_{L,c})+\Tr(\hat K_{L,c})
\end{equation}
with $\tilde K_{L,c}$ and $\hat K_{L,c}$ given in \eqref{klcdef}. Clearly, $\Tr(P_0B_{0,c})$ does not depend on $L$ and is finite by $0\leq\Ai(x)\leq e^{-\frac{2x^{3/2}}{3}}$ for $x\geq 0$. Now let's consider $\Tr(\tilde K_{L,c})$, by definition \eqref{klcdef}, we have
\begin{equation}\label{tildeklc}
	\begin{aligned}
		&\Tr(\tilde K_{L,c})=\Tr(\bar P_0e^{L\Delta}P_0e^{-L\Delta}B_{0,c})\\
		=&\Tr(e^{L\Delta}P_0e^{-L\Delta}B_{0,c})-\Tr(P_0e^{L\Delta}P_0e^{-L\Delta}B_{0,c})\\
		=&\Tr(P_0B_{0,c})-\Tr( P_0e^{L\Delta}P_0e^{-L\Delta}B_{0,c}),
	\end{aligned}
\end{equation}
where in the last step we use cyclic property and the fact $e^{-L\Delta}B_{0,c}e^{L\Delta}=B_{0,c}$. The error term is then given by the second term on the right hand side of \eqref{tildeklc}. In order to bound the error term, we define $A_1=P_0e^{L\Delta}P_0$ and $A_2=P_0e^{-L\Delta}B_{0,c}P_0$. Applying now $\|U_r^{-1}A_1U_r^{-1}\|_{\rm HS}\leq \tfrac{1}{\sqrt{ L}}\leq 1$ (by~\eqref{4b3}), $\|U_rA_2U_r\|_{\rm HS}\leq \beta e^{-\frac{4L^3}{3}}e^{-2Lc}$ (by~\eqref{1b4}) and ~\ref{ts3} of Theorem~\ref{simon}, we have
    \begin{equation}\label{s212}
    \begin{aligned}
     &|\Tr( P_0e^{L\Delta}P_0e^{-L\Delta}B_{0,c})|=|\Tr(U_r^{-1}A_1U_r^{-1}U_rA_2U_r)|\\
     \leq&\|U_r^{-1}A_1U_r^{-1}\|_{\rm HS}\|U_rA_2U_r\|_{\rm HS}\leq \beta e^{-\frac{4L^3}{3}}.
    \end{aligned}
    \end{equation}
    It remains to deal with $\Tr(\hat K_{L,c})$. Recall that heat kernel has the following integral representation: for any $L>0$,
    \begin{equation}\label{s215}
        e^{L\Delta}(x,y)=\frac{1}{2\pi i}\int_{i\R+\sigma}dv e^{Lv^2+v(x-y)}
    \end{equation}
    where $\sigma\in\R$ can arbitrarily be chosen. Furthermore, using the integral representation of Airy function on \eqref{eldelta}, we obtain
\begin{equation}\label{111}
\begin{aligned}
&e^{-L\Delta}B_{0,c}(x,y)=\frac{1}{2\pi i}\int_{i\R+\mu_1}dw e^{\frac{w^3}{3}-Lw^2-w(x+y+2c)}
\end{aligned}
\end{equation}
under the condition $\mu_1>L$. Using \eqref{eldeltat}, \eqref{s215} and \eqref{111}, we have
\begin{equation}\label{eq2.18}
    \begin{aligned}
    	&\Tr(\bar P_0e^{L\tilde\Delta}\bar P_0e^{-L\Delta}B_{0,c})\\
    	=&\int_{-\infty}^0dx\int_{-\infty}^0dye^{L\tilde\Delta}(x,y)e^{-L\Delta}B_{0,c}(y,x)\\
    	=&\frac{1}{(2\pi i)^2}\int_{i\R+\mu_1}dw\int_{i\R+\mu_2}dv e^{\frac{w^3}{3}-Lw^2+Lv^2-2wc}\int_{-\infty}^0dx\int_{-\infty}^0dy e^{(v-w)x}e^{(v-w)y}\\
    	=&\frac{1}{(2\pi i)^2}\int_{i\R+\mu_1}dw\int_{i\R+\mu_2}dv\frac{e^{\frac{w^3}{3}-Lw^2+Lv^2-2wc}}{(v-w)^2}
    \end{aligned}
\end{equation}
provided $\mu_2>\mu_1>L$. We then deform the contour $v$ to $\I\R$ and taking care of the pole at $v=w$ by Cauchy's residue theorem, we obtain
\begin{equation}
\begin{aligned}
      \Tr(\bar P_0e^{L\tilde\Delta}\bar P_0e^{-L\Delta}B_{0,c})&=\frac{1}{2\pi i}\int_{i\R+2L}dw{\rm Res}\bigg(\frac{e^{\frac{w^3}{3}-Lw^2+Lv^2-2wc}}{(v-w)^2}\Big|_{v=w}\bigg)  \\
      & \quad+\frac{1}{(2\pi i)^2}\int_{i\R+2L}dw\int_{i\R}dv\frac{e^{\frac{w^3}{3}-Lw^2+Lv^2-2wc}}{(v-w)^2}\\
      &=\frac{2L}{2\pi i}\int_{i\R+2L}dw e^{w^3/3-2 c w} w-\Tr(P_0e^{L\tilde\Delta}P_0e^{-L\Delta}B_{0,c})\\
      &=-2L\Ai'(2c)-\Tr(P_0e^{L\tilde\Delta}P_0e^{-L\Delta}B_{0,c}).
\end{aligned}
\end{equation}
For the error term $\Tr(P_0e^{L\tilde\Delta}P_0e^{-L\Delta}B_{0,c})$, similarly as~\eqref{s212}, we get
\begin{equation}
\begin{aligned}
    &
\end{aligned}
    |\Tr(P_0e^{L\tilde\Delta}P_0e^{-L\Delta}B_{0,c})|\leq e^{-\frac{4L^3}{3}}.
\end{equation}
Summarizing, for $n=1$ we have obtained the following result.
\begin{prop}\label{klc1}
    For any $c,r>0$ with $1\leq r^2\leq2c$, we have
        \begin{equation}
            |\Tr(K_{L,c})-2\Tr(P_0B_{0,c})+2L\Ai'(2c)|\leq e^{-\frac{4L^3}{3}},\quad\forall L\geq 1.
        \end{equation}
\end{prop}

\subsection{Leading term for $n\geq 2$ case}
In the decomposition that we will do below of $\Tr(K_{L,c}^n)$ with general $n$, there will be one term given by $\Tr(\bar P_0e^{L\tilde\Delta}\hat B_{0,c}^n\bar P_0e^{-L\Delta}B_{0,c})$, which is up to error terms the term appearing in $\kappa(c)$. Since the decomposition is a bit lengthly, we first get a control on this term.
\begin{lem}\label{wkey3}
Let $n\in\mathbb Z_{\geq 1},L,r\geq 1$ and $r^2\leq 2c$, it holds
    \begin{equation}
    \big|\Tr(\bar P_0e^{L\tilde\Delta}\hat B_{0,c}^n\bar P_0e^{-L\Delta}B_{0,c})+2(n+1)^{-2/3}L\Ai'(2(n+1)^{2/3}c)\big|\leq\beta^{n+1}e^{-\frac{4L^3}{3}}.
    \end{equation}
\end{lem}
\begin{proof}
Using the integral representation of Airy function, we have (see \eqref{equal})
    \begin{equation}
        e^{L\tilde\Delta}\hat B_{0,c}^n(x,y)=\frac{1}{2\pi i}\int_{i\R+\mu_2}dwe^{\frac{w^3}{3}+wn^{-1/3}(x+y)+n^{-2/3}Lw^2-2n^{2/3}c}
    \end{equation}
with arbitrary $\mu_2>0$. Together with \eqref{111}, we obtain
\begin{equation}\label{eq2.22}
    \begin{aligned}
        &\Tr(\bar P_0e^{L\tilde\Delta}\hat B_{0,c}^n\bar P_0e^{-L\Delta}B_{0,c})\\
        =&\frac{1}{(2\pi i)^2}\int_{i\R+\mu_1} dw_1\int _{i\R+\mu_2}dw_2\frac{e^{\frac{nw_1^3+w_2^3}{3}+L(w_1^2-w_2^2)-2w_1(n+1)c-2w_2c}}{(w_1-w_2)^2},
    \end{aligned}
\end{equation}
    with $\mu_1>\mu_2>0$. Deforming the contours to satisfy $\mu_2=2L$ and $\mu_1=0$ we get
\begin{multline}
    \Tr(\bar P_0e^{L\tilde\Delta}(\hat B_{0,c})^n\bar P_0e^{-L\Delta}B_{0,c})\\
       = -2(n+1)^{-2/3}\Ai'(2(n+1)^{2/3}c) L+\Tr(P_0e^{L\tilde\Delta}\hat B_{0,c}^nP_0e^{-L\Delta}B_{0,c}),
\end{multline}
where the first term is coming from the residue at $w_1=w_2$ (in form of an integral representation of the derivative of the Airy function). Similarly as~\eqref{s212}, we define $A_1=P_0e^{L\tilde\Delta}\hat B_{0,c}^nP_0$ and $A_2=P_0e^{-L\Delta}B_{0,c}P_0$. Applying now $\|U_r^{-1}A_1U_r^{-1}\|_{\rm HS}\leq\beta^ne^{Lr^2}$ (by~\eqref{4bb4}), $\|U_rA_2U_r\|_{\rm HS}\leq\beta e^{-\frac{4L^3}{3}}$ (by~\eqref{1b4}) and $r^2\leq 2c$, we have
\begin{equation}
\begin{aligned}
   & |\Tr(P_0e^{L\tilde\Delta}\hat B_{0,c}^nP_0e^{-L\Delta}B_{0,c})|
   \leq \|U_r^{-1}A_1U_r^{-1}\|_{\rm HS} \|U_rA_2U_r\|_{\rm HS} \leq\beta^{n+1}e^{-\frac{4L^3}{3}}.
\end{aligned}
\end{equation}
\end{proof}

\subsection{Full expansion of $\Tr(K_{L,c}^n)$ for $n\geq 2$}\label{sec23}
Using the same idea, we can also deduce the asymptotic behavior of $\Tr(K_{L,c}^n)$ with $n\geq 2$. Using $e^{L\tilde\Delta}e^{-L\Delta}B_{0,c}=\tilde B_{0,c}$ (by~\eqref{a112}), we can decompose
\begin{equation}\label{e226}
    K_{L,c}=K_u+K_v+K_w+K_d-K_e
\end{equation}
where
\begin{equation}\label{allnotation}
    \begin{aligned}
        &K_u=P_0B_{0,c},\quad K_v=e^{L\Delta}P_0e^{-L\Delta}B_{0,c},\\
        &K_w=\tilde B_{0,c}-P_0\tilde B_{0,c}-e^{L\tilde\Delta}P_0e^{-L\Delta}B_{0,c},\\
        &K_d=P_0e^{L\tilde\Delta}P_0e^{-L\Delta}B_{0,c},\ K_e=P_0e^{L\Delta}P_0e^{-L\Delta}B_{0,c}.
    \end{aligned}
\end{equation}
For a word $\sigma_n$ of length $n$, we say that $\alpha\in\sigma_n$ if it exists $i\in\{1,\ldots,n\}$ such that $\sigma_n(i)=\alpha$. Also we introduce the notation
\begin{equation}
\sgn{A}{\sigma_n}=(-1)^{\#\{i|\sigma_n(i)\in A\}},
\end{equation}
where $A$ is a subset of the letters of $\sigma_n$. With the above definitions we can rewrite
\begin{equation}\label{kln}
    \begin{aligned}
        &\Tr(K_{L,c}^n)=\sum_{\sigma_n\in\{u,v,w,d,e\}^n}\sgn{e}{\sigma_n}\Tr\bigg(\prod_{i=1}^nK_{\sigma_n(i)}\bigg)\\
        =&\sum_{\sigma_n\in\{u,v,w\}^n}\Tr\bigg(\prod_{i=1}^nK_{\sigma_n(i)}\bigg)+\sum_{\substack{\sigma_n\in\{u,v,w,d,e\}^n\\d\ or\ e\in\sigma_n}
        }\sgn{e}{\sigma_n}\Tr\bigg(\prod_{i=1}^nK_{\sigma_n(i)}\bigg).
    \end{aligned}
\end{equation}
Besides $K_u,K_v,K_w,K_d$ and $K_e$, we introduce further the following operators:
\begin{equation}
K_a=\tilde B_{0,c},\ K_b=P_0\tilde B_{0,c},\ K_c=e^{L\tilde\Delta}P_0e^{-L\Delta}B_{0,c}
\end{equation}
so that $K_w=K_a-K_b-K_c$. First we control the last term in \eqref{kln} as follows.
\begin{lem}\label{forsimple}
Let $n\geq 2$, $L\geq 1, r^2\leq 2c$ with $r\geq 1$ and $\beta$ given in~\eqref{beta}, it holds
\begin{equation}\label{simpl}
    \sum_{\substack{\sigma_n\in\{u,v,w,d,e\}^n\\d\ or\ e\in\sigma_n}
        }\bigg|\Tr\bigg(\prod_{i=1}^nK_{\sigma_n(i)}\bigg)\bigg|\leq 7^ne^{-\frac{4L^3}{3n^2}}\beta^n.
\end{equation}
\end{lem}
\begin{proof}
Since $K_w=K_a-K_b-K_c$, we have then
\begin{gather}
   \sum_{\substack{\sigma_n\in\{u,v,w,d,e\}^n\\d\ or\ e\in\sigma_n}
        }\Tr\bigg(\prod_{i=1}^nK_{\sigma_n(i)}\bigg)=\sum_{\substack{\sigma_n\in\{u,v,a,b,c,d,e\}^n\\d\ or\ e\in\sigma_n}
        }\sgn{b,c}{\sigma_n}\Tr\bigg(\prod_{i=1}^nK_{\sigma_n(i)}\bigg).
\end{gather}
Note that there are in total $7^n-5^n$ many summations appearing on the right hand side, that is, the cardinality of the set $\{a,b,c,d,e,u,v\}^n\setminus\{a,b,c,u,v\}^n$. Hence, in order to prove the claim, we only need to bound the summation on the right hand side, to this end, we choose arbitrary $\sigma_n\in\{u,v,a,b,c,d,e\}^n$ with $e\in\sigma_n$. Using the cyclic property of trace, we can assume $\sigma_n(1)=e$. Define now
\begin{equation}    \Phi=\underbrace{P_0e^{L\Delta}P_0}_{=:\varphi_1}\cdot\underbrace{ P_0e^{-L\Delta}B_{0,c}\bigg[\prod_{i=2}^nK_{\sigma_n(i)}\bigg]P_0}_{=:\varphi_2}.
\end{equation}
Applying $\|U_r^{-1}\varphi_1 U_r^{-1}\|_{\rm HS}\leq\tfrac{1}{\sqrt{L}}\leq1$ (by~\eqref{4b3}) and $\|U_r \varphi_2 U_r\|_{\rm HS}\leq \beta^{n}e^{-\frac{4L^3}{3n^2}}e^{-2Lc}$ (by Corollary~\ref{ca17}), we have
\begin{equation}
   |\Tr(\Phi)|=|\Tr(U_r^{-1}\Phi U_r)|\leq\|U_r^{-1}\varphi_1 U_r^{-1}\|_{\rm HS}\|U_r \varphi_2 U_r\|_{\rm HS}\leq\beta^{n}e^{-\frac{4L^3}{3n^2}} .
\end{equation}
Similarly, we can also show the result for $d\in\sigma_n$, we only need to apply the transformation $P_0e^{L\tilde \Delta}P_0\mapsto U_r^{-1}P_0e^{L\tilde \Delta}P_0U_r^{-1}$ and use~\eqref{4b3}.
\end{proof}

Next we consider the first term of \eqref{kln}, namely
\begin{equation}\label{127}
    \begin{aligned}
        &\sum_{\sigma_n\in\{u,v,w\}^n}\Tr\bigg(\prod_{i=1}^nK_{\sigma_n(i)}\bigg)=\Tr(K_u^n)+\Tr(K_v^n)+\Tr(K_w^n)\\
        &+\sum_{\substack{\sigma_n\in\{u,v\}^n\\
        u,v\in\sigma_n
        }}\Tr\bigg(\prod_{i=1}^nK_{\sigma_n(i)}\bigg)+\sum_{\substack{\sigma_n\in\{u,w\}^n\\
        u,w\in\sigma_n
        }}\bigg(\prod_{i=1}^nK_{\sigma_n(i)}\bigg)+\sum_{\substack{\sigma_n\in\{v,w\}^n\\
        v,w\in\sigma_n
        }}\Tr\bigg(\prod_{i=1}^nK_{\sigma_n(i)}\bigg)\\
        &+\Id_{n\geq 3}\sum_{\substack{\sigma_n\in\{u,v,w\}^n\\
        u,v,w\in\sigma_n
        }}\Tr\bigg(\prod_{i=1}^nK_{\sigma_n(i)}\bigg).
    \end{aligned}
\end{equation}
In the next sections we analyze the terms in \eqref{127} one after the other.

\subsubsection{Single terms}
The first two terms in \eqref{127} do not depend on $L$ and are easy to bound.
\begin{lem}\label{lemEst1}
For any $n\in\mathbb Z_{\geq 2}$ we have
$\Tr(K_u^n)=\Tr(K_v^n)=\Tr((P_0B_{0,c})^n)$
and
\begin{equation}\label{242}
|\Tr((P_0B_{0,c})^n)|\leq \beta^n.
\end{equation}
\end{lem}
\begin{proof}
By definition of $K_u$, we have $\Tr(K_u^n)=\Tr((P_0B_{0,c})^n)$. As for $\Tr(K_v^n)$, using the definition of $K_v=e^{L\Delta}P_0e^{-L\Delta}B_{0,c}$ and the fact that $B_{0,c}$ commute with $e^{L\Delta}$, we obtain $\Tr(K_v^n)=\Tr((P_0B_{0,c})^n)$. It remains to prove the upper bound. By~\eqref{1ab3}, we have $\|U_rP_0B_{0,c}P_0U_r^{-1}\|_{\rm HS}\leq\beta$. Since $n\geq 2,$ we can apply Theorem~\ref{simon} to deduce
\begin{equation}
|\Tr((P_0B_{0,c})^n)|=|\Tr((U_rP_0B_{0,c}P_0U_r^{-1})^n)|\leq\|U_rP_0B_{0,c}P_0U_r^{-1}\|_{\rm HS}^n\leq\beta^n.
\end{equation}
\end{proof}

Now we need to consider $\Tr(K_w^n)$ with $n\geq 2$, which gives some terms of order $1$ and some terms linear in $L$ plus error terms as we will show in Proposition~\ref{kwkey}. Recall that
\begin{equation}\label{kw}
    K_w=\tilde B_{0,c}-P_0\tilde B_{0,c}-e^{L\tilde\Delta}P_0e^{-L\Delta}B_{0,c}.
\end{equation}
Using the cyclic property of the trace we get, for $n\geq 2$
\begin{equation}\label{tkwn}
\Tr(K_w^{n})=\Tr(\bar P_0\tilde B_{0,c}K_w^{n-1})-\Tr(P_0e^{-L\Delta}B_{0,c}K_w^{n-1}e^{L\tilde\Delta}).
\end{equation}
We will start with the easy term, that is, the second term on the right hand side:
\begin{lem}\label{kw2p}
  Let $n\in\mathbb Z_{\geq 2},L,r\geq 1$ and $r^2\leq 2c$, it holds
    \begin{equation}\label{kw2pin}
        |\Tr(P_0e^{-L\Delta}B_{0,c}K_w^{n-1}e^{L\tilde\Delta})-\Tr(P_0\hat B_{0,c}(\bar P_0\hat B_{0,c})^{n-1})|\leq 3^{n-1}\beta^ne^{-\frac{4L^3}{3n^2}}
    \end{equation}
and
\begin{equation}\label{214}
    |\Tr(P_0\hat B_{0,c}(\bar P_0\hat B_{0,c})^{n-1})|\leq\beta^n.
\end{equation}
\end{lem}
\begin{proof}
Let us start by deriving the bound \eqref{214}. Applying~\eqref{3ab3} and Theorem~\eqref{simon}, we have
\begin{equation}
    \begin{aligned}
        &|\Tr(P_0\hat B_{0,c}(\bar P_0\hat B_{0,c})^{n-1})|\\
        \leq&\|U_rP_0\hat B_{0,c}\bar P_0U_r^{-1}\|_{\rm HS}\|U_r\bar P_0\hat B_{0,c}U_r^{-1}\|_{\rm op}^{n-2}\|U_r\bar P_0\hat B_{0,c}P_0U_r^{-1}\|_{\rm HS}\leq\beta^n.
    \end{aligned}
\end{equation}
Next we show \eqref{kw2pin}. Recall that $K_w=K_a-K_b-K_c$ and $K_a-K_c=e^{L\tilde\Delta}\bar P_0e^{-L\Delta}B_{0,c}$. Hence,
\begin{equation}\label{22backA}
        \begin{aligned}
        &\Tr(P_0e^{-L\Delta}B_{0,c}K_w^{n-1}e^{L\tilde\Delta}P_0) = \Tr(P_0e^{-L\Delta}B_{0,c}(e^{L\tilde\Delta}\bar P_0e^{-L\Delta}B_{0,c}-P_0\tilde B_{0,c})^{n-1}e^{L\tilde\Delta}P_0)\\
         =&\Tr(P_0e^{-L\Delta}B_{0,c}(e^{L\tilde\Delta}\bar P_0e^{-L\Delta} B_{0,c})^{n-1}e^{L\tilde\Delta}P_0)
        \\&+\sum_{\substack{\sigma_{n-1}\in\{a,b,c\}^{n-1}\\
        b\in\sigma_{n-1}}}\sgn{b,c}{\sigma_{n-1}}\Tr\bigg(P_0e^{-L\Delta}B_{0,c}\bigg[\prod_{i=1}^{n-1}K_{\sigma_{n-1}(i)}\bigg]e^{L\tilde\Delta}P_0\bigg).
    \end{aligned}
    \end{equation}
Applying $
e^{-L\Delta}B_{0,c}e^{L\tilde\Delta}=e^{-L\Delta}e^{L\Delta}\hat B_{0,c}=\hat B_{0,c}$ (by~\eqref{a102}), we have
\begin{equation}\label{22back}
\Tr(P_0e^{-L\Delta}B_{0,c}(e^{L\tilde\Delta}\bar P_0e^{-L\Delta} B_{0,c})^{n-1}e^{L\tilde\Delta}P_0)=\Tr(P_0\hat B_{0,c}(\bar P_0\hat B_{0,c})^{n-1}).
    \end{equation}
Now it remains to bound the sum in \eqref{22backA}. Let now $\sigma_{n-1}\in\{a,b,c\}^n$ with $b\in\sigma_{n-1}$ and define
\begin{equation}
\begin{aligned}
      \Phi=&P_0e^{-L\Delta}B_{0,c}\bigg[\prod_{i=1}^{n-1}K_{\sigma_{n-1}(i)}\bigg]e^{L\tilde\Delta}P_0\\
     =&\underbrace{P_0e^{-L\Delta}B_{0,c}\bigg[\prod_{i=1}^{\ell_b-1}K_{\sigma_{n-1}(i)}\bigg]P_0}_{=:\varphi_1}\cdot\underbrace{ P_0\tilde B_{0,c}\bigg[\prod_{i=\ell_b+1}^{n-1}K_{\sigma_{n-1}(i)}\bigg]e^{L\tilde\Delta}P_0}_{=:\varphi_2},
\end{aligned}
\end{equation}
where $\ell_b=\max\{i|\sigma_{n-1}(i)=b\}\geq 1$. Applying $\|U_r\varphi_1U_r\|_{\rm HS}\leq \beta^{\ell_b}e^{-\frac{4L^3}{3\ell_b^2}}e^{-2Lc}$ (by Corollary~\ref{ca17}), $\|U_r^{-1}\varphi_2U_r^{-1}\|_{\rm HS}\leq \beta^{n-\ell_b}e^{Lr^2}$ (by Lemma~\ref{wkey}) and $r^2\leq 2c$, we have
\begin{equation}
    |\Tr(\Phi)|\leq\|U_r\varphi_1U_r\|_{\rm HS}\|U_r^{-1}\varphi_2U_r^{-1}\|_{\rm HS}\leq  \beta^ne^{-\frac{4L^3}{3n^2}}.
\end{equation}
Applying this and triangle inequality on \eqref{22backA}, the result follows from the fact that there are in total $3^{n-1}-2^{n-1}$ many summations.
\end{proof}

It remains to consider $\Tr(\bar P_0\tilde B_{0,c}K_w^{n-1})$ in \eqref{tkwn}. Similarly as  \eqref{22backA}, we deduce
\begin{equation}\label{3rdterm}
    \begin{aligned}
        \Tr(\bar P_0\tilde B_{0,c}K_w^{n-1})&=\Tr(\bar P_0\tilde B_{0,c}e^{L\tilde\Delta}(\bar P_0\hat B_{0,c})^{n-2}\bar P_0e^{-L\Delta}B_{0,c})\\
        &+\sum_{\substack{\sigma_{n-1}\in\{a,b\}^{n-1}\\
        b\in\sigma_{n-1}
        }}\sgn{b}{\sigma_{n-1}}\Tr\bigg(\bar P_0\tilde B_{0,c}\prod_{i=1}^{n-1}K_{\sigma_{n-1}(i)}\bigg)\\
        &+\sum_{\substack{\sigma_{n-1}\in\{a,b,c\}^{n-1}\\
        b,c\in\sigma_{n-1}
        }}\sgn{b,c}{\sigma_{n-1}}\Tr\bigg(\bar P_0\tilde B_{0,c}\prod_{i=1}^{n-1}K_{\sigma_{n-1}(i)}\bigg).
    \end{aligned}
\end{equation}
By definition, $K_a=\tilde B_{0,c}$ and $K_b=P_0\tilde B_{0,c}$, the term on the second line does not depend on $L$, hence for this term, it is enough to get an upper bound which is summable for $n\geq 1$. In Lemma~\ref{lem2.12} we get a bound for the terms in the second line, in Lemma~\ref{kwiii} we bound the terms in the third line. Finally, in Lemma~\ref{kwi}, we will show that the first term on the right hand side will provide $\mathcal O(L)$ term.

\begin{lem}\label{lem2.12}
    Let $n\in\mathbb Z_{\geq 2}$, $L,r\geq 1, r^2\leq 2c$ and $\sigma_{n-1}\in\{a,b\}^{n-1}$ with $b\in\sigma_{n-1}$, then
    \begin{equation}\label{234}
        \bigg|\Tr\bigg(\bar P_0\tilde B_{0,c}\bigg[\prod_{i=1}^{n-1}K_{\sigma_{n-1}(i)}\bigg]\bigg)\bigg|\leq \beta^n.
    \end{equation}
\end{lem}
\begin{proof}
    Since $b\in\sigma_{n-1}$ we have $1\leq \ell_b=\max\{i|\sigma_{n-1}=b\}\leq n-1$. This implies
    \begin{equation}
        \Phi=\bar P_0\tilde B_{0,c}\bigg[\prod_{i=1}^{n-1}K_{\sigma_{n-1}(i)}\bigg] \bar P_0=\underbrace{\bar P_0\tilde B_{0,c}\bigg[\prod_{i=1}^{\ell_b-1}K_{\sigma_{n-1}(i)}\bigg]P_0}_{=:\varphi_1}\cdot\underbrace{ P_0\tilde B_{0,c}^{n-\ell_b} \bar P_0}_{=:\varphi_2}.
    \end{equation}
Applying $\|U_r^{-1}\varphi_1U_r\|_{\rm HS}\leq\beta^{\ell_b}$ (by Lemma~\ref{l25}) and $\|U_r^{-1}\varphi_2U_r\|_{\rm HS}\leq\beta^{n-\ell_b}$ (by~\eqref{2ab3}), we have
\begin{equation}
    |\Tr(\Phi)|\leq\|U_r^{-1}\varphi_1U_r\|_{\rm HS}\|U_r^{-1}\varphi_2U_r\|_{\rm HS}\leq \beta^{n}.
\end{equation}
\end{proof}

Next we bound the traces of the terms on the last line of \eqref{3rdterm}.
\begin{lem}\label{kwiii}
    Let $n\in\mathbb Z_{\geq 2}$, $L,r\geq 1,r^2\leq 2c$ and $\sigma_{n-1}\in\{a,b,c\}^{n-1}$ with $b,c\in\sigma_{n-1}$, then
    \begin{equation}
        \bigg|\Tr\bigg(\bar P_0\tilde B_{0,c}\bigg[\prod_{i=1}^{n-1}K_{\sigma_{n-1}(i)}\bigg]\bigg)\bigg|\leq2\beta^ne^{-\frac{4L^3}{3n^2}}.
    \end{equation}
\end{lem}
\begin{proof}
    Replacing $\bar P_0=\Id-P_0$ we have the decomposition
\begin{gather}\label{wer}
     \Tr\bigg(\bar P_0\tilde B_{0,c}\prod_{i=1}^{n-1}K_{\sigma_{n-1}(i)}\bigg)=\Tr\bigg(\tilde B_{0,c}\prod_{i=1}^{n-1}K_{\sigma_{n-1}(i)}\bigg)
     -\Tr\bigg( P_0\tilde B_{0,c}\prod_{i=1}^{n-1}K_{\sigma_{n-1}(i)}\bigg).
\end{gather}
By assumption, there exists $i\in\{1,\ldots,n-1\}$ such that $K_{\sigma_{n-1}(i)}=P_0\tilde B_{0,c}$, hence, due to the cyclic property of trace, it is enough to show that
\begin{gather}
    \bigg|\Tr\bigg(P_0\tilde B_{0,c}\bigg[\prod_{i=1}^{n-1}K_{\sigma_{n-1}(i)}\bigg]P_0\bigg)\bigg|\leq\beta^ne^{-\frac{4L^3}{3n^2}}
\end{gather}
for any $\sigma_{n-1}\in\{a,b,c\}^{n-1}$ with $c\in\sigma_{n-1}$. Define $\ell_c=\min\{i|\sigma_{n-1}=c\}\leq n-1$ and
\begin{equation}
    \begin{aligned}
        \Phi=& P_0\tilde B_{0,c}\bigg[\prod_{i=1}^{n-1}K_{\sigma_{n-1}(i)}\bigg]P_0\\
        =& \underbrace{P_0\tilde B_{0,c}\bigg[\prod_{i=1}^{\ell_c-1}K_{\sigma_{n-1}(i)}\bigg]e^{L\tilde\Delta}P_0}_{=:\varphi_1}\cdot \underbrace{P_0e^{-L\Delta}B_{0,c}\bigg[\prod_{i=\ell_c+1}^{n-1}K_{\sigma_{n-1}(i)}\bigg]P_0}_{=:\varphi_2}.
    \end{aligned}
\end{equation}
Applying $\|U_r^{-1}\varphi_1U_r^{-1}\|_{\rm HS}\leq\beta^{\ell_c}e^{Lr^2}$ (by Lemma~\ref{wkey}), $\|U_r\varphi_2U_r\|_{\rm HS}\leq\beta^{n-\ell_c}e^{-\frac{4L^3}{3(n-\ell_c)^2}}e^{-2Lc}$ (by Corollary~\ref{ca17}) and $r^2\leq 2c$, we have
\begin{equation}
    |\Tr\left(\Phi\right)|\leq\|U_r^{-1}\varphi_1U_r^{-1}\|_{\rm HS}\|U_r\varphi_2U_r\|_{\rm HS}\leq\beta^ne^{-\frac{4L^3}{3n^2}}.
\end{equation}
Applying triangle inequality on \eqref{wer}, we then obtain the claimed result.
\end{proof}
It remains to consider the first term appearing on the right hand side of \eqref{3rdterm}.
\begin{lem}\label{kwi}
Let $n\in\mathbb Z_{\geq 2},L,r\geq 1$ and $r^2\leq 2c$, denote \begin{equation}
\tilde \Phi=\bar P_0\tilde B_{0,c}e^{L\tilde\Delta}(\bar P_0\hat  B_{0,c})^{n-2}\bar P_0e^{-L\Delta}B_{0,c}.
\end{equation}
Then it holds
\begin{equation}\label{eq2.73}
|\Tr(\tilde \Phi)-\kappa_n(c)L+\Psi^2_n|\leq n^2\beta^{n}e^{-\frac{4L^3}{3n^2}},
\end{equation}
  where  $\kappa_n(c)=-2n^{-2/3}\Ai'(2n^{2/3}c)$ and
    \begin{equation}\label{psi2n}
        \Psi^2_n=\Id_{n\geq 3}\sum_{j=2}^{n-1}\Tr(P_0\hat B_{0,c}(\bar P_0\hat B_{0,c})^{n-1-j}\bar P_0\hat B_{0,c}^j)\quad\text{with}\quad |\Psi^2_n|\leq n\beta^n.
    \end{equation}
\end{lem}
\begin{proof}
We first claim that for any two operators $A,B$ and $n\geq 0$, we have
\begin{equation}\label{obtainthen}
    \bar P_0A(\bar P_0B)^n=\bar P_0AB^n-\sum_{j=0}^{n-1}AB^jP_0B(\bar P_0B)^{n-(j+1)}+\sum_{j=0}^{n-1}P_0AB^jP_0B(\bar P_0B)^{n-(j+1)}.
\end{equation}
We show this via induction on $n$. The case $n=0$ is trivial. For $n=1$, we have
 \begin{equation}\label{barbarminus}
    \bar P_0A\bar P_0B=P_0AP_0B-AP_0B+\bar P_0AB.
\end{equation}
For induction step $n-1\mapsto n$, we have then
\begin{equation}\label{backtoas}
    \begin{aligned}
         \bar P_0A(\bar P_0B)^n=& P_0AP_0B(\bar P_0B)^{n-1}-AP_0B(\bar P_0B)^{n-1}+\bar P_0AB(\bar P_0B)^{n-1}.
    \end{aligned}
\end{equation}
Applying induction assumption on the last term, we obtain
\begin{equation}
    \begin{aligned}
        &\bar P_0AB(\bar P_0B)^{n-1}\\
        =&\bar P_0ABB^{n-1}-\sum_{j=0}^{n-2}ABB^jP_0B(\bar P_0B)^{n-1-(j+1)}+\sum_{j=0}^{n-2}P_0ABB^jP_0B(\bar P_0B)^{n-1-(j+1)}\\
        =&\bar P_0AB^{n}-\sum_{j=1}^{n-1}AB^jP_0B(\bar P_0B)^{n-1-j}+\sum_{j=1}^{n-1}P_0AB^jP_0B(\bar P_0B)^{n-1-j}.
    \end{aligned}
\end{equation}
Plugging this back to \eqref{backtoas}, we obtain then \eqref{obtainthen}. Applying now \eqref{obtainthen} with $A=\tilde B_{0,c}e^{L\tilde\Delta}$ and $B=\hat B_{0,c}$, we then obtain
\begin{equation}
    \begin{aligned}
        &\Tr(\tilde\Phi)=\Tr(\bar P_0\tilde B_{0,c}e^{L\tilde\Delta}(\bar P_0\hat  B_{0,c})^{n-2}\bar P_0e^{-L\Delta}B_{0,c})\\
        =&\Tr(\bar P_0\tilde B_{0,c}e^{L\tilde\Delta}\hat  B_{0,c}^{n-2}\bar P_0e^{-L\Delta}B_{0,c})-\sum_{j=0}^{n-3}\Tr(\tilde B_{0,c}e^{L\tilde\Delta}\hat B_{0,c}^jP_0\hat B_{0,c}(\bar P_0\hat B_{0,c})^{n-2-(j+1)}\bar P_0e^{-L\Delta}B_{0,c})\\
        +&\sum_{j=0}^{n-3}\Tr(P_0\tilde B_{0,c}e^{L\tilde\Delta}\hat B_{0,c}^jP_0\hat B_{0,c}(\bar P_0\hat B_{0,c})^{n-2-(j+1)}\bar P_0e^{-L\Delta}B_{0,c})
    \end{aligned}
\end{equation}
Using the identity $e^{L\tilde\Delta}\hat B_{0,c}=\tilde B_{0,c}e^{L\tilde\Delta}$ (by~\eqref{a101}), $e^{-L\Delta}B_{0,c}\tilde B_{0,c}e^{L\tilde\Delta}=\hat B_{0,c}^2$ (by~\eqref{a103}) and cyclic property of trace we obtain
    \begin{equation}\label{trphin}
        \begin{aligned}
            \Tr(\tilde \Phi)&=\Tr(\bar P_0\tilde B_{0,c}^{n-1}e^{L\tilde\Delta}\bar P_0e^{-L\Delta}B_{0,c}) -\sum_{j=2}^{n-1}\Tr(P_0\hat B_{0,c}(\bar P_0\hat B_{0,c})^{n-1-j}\bar P_0\hat B_{0,c}^j)\\
            &+\sum_{j=1}^{n-2}\Tr(P_0\tilde B_{0,c}^{j}e^{L\tilde\Delta}P_0\hat B_{0,c}(\bar P_0\hat B_{0,c})^{n-2-j}\bar P_0e^{-L\Delta}B_{0,c}).
        \end{aligned}
    \end{equation}
By Lemma~\ref{wkey3}, we have
\begin{equation}\label{eq2.79}
    |\Tr(\bar P_0\tilde B_{0,c}^{n-1}e^{L\tilde\Delta}\bar P_0e^{-L\Delta}B_{0,c})-2n^{-2/3}\Ai'(2n^{2/3}c) L|\leq \beta^ne^{-\frac{4L^3}{3}}.
\end{equation}
The first sum in \eqref{trphin} is just $\Psi_n^2$ in \eqref{psi2n}, which is independent of $L$. In order to prove $|\Psi_n^2|\leq n\beta^n$, we apply bounds~\eqref{3ab3} for each $j\in\{2,3,\ldots,n-1\}$ to deduce
\begin{equation}\label{253}
\begin{aligned}
    &|\Tr(P_0\hat B_{0,c}(\bar P_0\hat B_{0,c})^{n-1-j}\bar P_0\hat B_{0,c}^j)| \\
    &\leq\|U_rP_0\hat B_{0,c}\bar P_0U_r^{-1}\|_{\rm HS}\|U_r\bar P_0\hat B_{0,c}\bar P_0U_r^{-1}\|_{\rm op}^{n-1-j}\| U_r\bar P_0\hat B_{0,c}^{j} P_0U_r^{-1}\|_{\rm HS}\leq\beta^n.
\end{aligned}
\end{equation}
The bound on $|\Psi_n^2|$ follows directly from \eqref{253} and triangle inequality. It remains to bound the last sum in \eqref{trphin}. Let $j\in\{1,\ldots,n-2\}$ and define
\begin{equation}\label{secondline}
\begin{aligned}
       \Phi=&P_0\tilde B_{0,c}^{j}e^{L\tilde\Delta}P_0\hat B_{0,c}(\bar P_0\hat B_{0,c})^{n-2-j}\bar P_0e^{-L\Delta}B_{0,c}\\
       =&P_0\tilde B_{0,c}^{j}e^{L\tilde\Delta}P_0\hat B_{0,c}(\bar P_0\hat B_{0,c})^{n-2-j}e^{-L\Delta}B_{0,c}\\
       -&P_0\tilde B_{0,c}^{j}e^{L\tilde\Delta}P_0\hat B_{0,c}(\bar P_0\hat B_{0,c})^{n-2-j}P_0e^{-L\Delta}B_{0,c}.
\end{aligned}
\end{equation}
Now we claim that for any $n\geq 1$, we have
\begin{equation}\label{barp0}
    (\bar P_0\hat B_{0,c})^n=\hat B_{0,c}^n-\sum_{i=1}^{n}(\bar P_0\hat B_{0,c})^{n-i}P_0\hat B_{0,c}^i.
\end{equation}
We show this via induction on $n$. For $n=1$, we have $\bar P_0\hat B_{0,c}=\hat B_{0,c}-P_0\hat B_{0,c}$. For induction step $n-1\mapsto n$, we have then
\begin{equation}
    \begin{aligned}
        (\bar P_0\hat B_{0,c})^n=&(\bar P_0\hat B_{0,c})^{n-1}\hat B_{0,c}-(\bar P_0\hat B_{0,c})^{n-1}P_0\hat B_{0,c}\\
        =&\left(\hat B_{0,c}^{n-1}-\sum_{i=1}^{n-1}(\bar P_0\hat B_{0,c})^{n-1-i}P_0\hat B_{0,c}^{i}\right)\hat B_{0,c}-(\bar P_0\hat B_{0,c})^{n-1}P_0\hat B_{0,c}\\
        =&\hat B_{0,c}^n-\sum_{i=1}^{n}(\bar P_0\hat B_{0,c})^{n-i}P_0\hat B_{0,c}^i.
    \end{aligned}
\end{equation}
Applying now \eqref{barp0} on the second line of \eqref{secondline}, we have then
\begin{equation}\label{2711}
    \begin{aligned}
       \Phi=& P_0\tilde B_{0,c}^{j}e^{L\tilde\Delta}P_0\hat B_{0,c}^{n-1-j}e^{-L\Delta}B_{0,c}\\
       -&\sum_{i=1}^{n-2-j} P_0\tilde B_{0,c}^{j}e^{L\tilde\Delta}P_0\hat B_{0,c}(\bar P_0\hat B_{0,c})^{n-2-j-i}P_0\hat B_{0,c}^ie^{-L\Delta}B_{0,c}\\
       -&P_0\tilde B_{0,c}^{j}e^{L\tilde\Delta}P_0\hat B_{0,c}(\bar P_0\hat B_{0,c})^{n-2-j}P_0e^{-L\Delta}B_{0,c}\\
       =&P_0\tilde B_{0,c}^{j}e^{L\tilde\Delta}P_0\hat B_{0,c}^{n-1-j}e^{-L\Delta}B_{0,c}\\
       -&\sum_{i=0}^{n-2-j} P_0\tilde B_{0,c}^{j}e^{L\tilde\Delta}P_0\hat B_{0,c}(\bar P_0\hat B_{0,c})^{n-2-j-i}P_0\hat B_{0,c}^ie^{-L\Delta}B_{0,c}.
    \end{aligned}
\end{equation}
    Note that for any $p\geq 1,q\geq 0$ we have $\|U_r^{-1}P_0\tilde B_{0,c}^{p}e^{L\tilde\Delta}P_0U_r^{-1}\|_{\rm HS}\leq \beta^pe^{Lr^2}$ by~\eqref{4bb4}, $\|U_rP_0\hat B_{0,c}^{q}e^{-L\Delta}B_{0,c}U_r\|_{\rm HS}\leq \beta^{q+1}e^{-\frac{4L^3}{3(q+1)^2}}e^{-2Lc}$ by~\eqref{1b4} and $\|\hat U_r B_{0,c}U_r^{-1}\|_{\rm op}\leq\beta$ by~\eqref{3ab3}. Applying those upper bounds, Theorem~\ref{simon} and assumption $r^2\leq 2c$ on \eqref{2711}, we have then
    \begin{equation}
        \begin{aligned}
        |\Tr(\Phi)|=|\Tr(U_r^{-1}\Phi U_r)|\leq \beta^{n}e^{-\frac{4L^3}{3n^2}}+\sum_{i=0}^{n-2-j}\beta^{n}e^{-\frac{4L^3}{3n^2}}\leq n\beta^ne^{-\frac{4L^3}{3n^2}}
        \end{aligned}
    \end{equation}
Plugging this back to \eqref{trphin}, we obtain the claimed results.
\end{proof}

Hence, we have the following result for $\Tr(K_w^n)$:
\begin{prop}\label{kwkey}
    Let $n\in\mathbb Z_{\geq 2},L,r\geq 1$ and $r^2\leq 2c$, it holds
    \begin{equation}
        |\Tr(K_w^n)-(\kappa_n(c)L+\Psi_n^1-\Psi_n^2-\Psi_n^3)|\leq 3^{n+1}\beta^ne^{-\frac{4L^3}{3n^2}},
    \end{equation}
    where  $\kappa_n(c)=-2n^{-2/3}\Ai'(2n^{2/3}c)$ and
\begin{equation}
    \begin{aligned}
            &\Psi_n^1=\sum_{\substack{\sigma_{n-1}\in\{a,b\}^{n-1}\\
        b\in\sigma_{n-1}
        }}\sgn{b}{\sigma_{n-1}}\Tr\bigg(\bar P_0\tilde B_{0,c}\bigg[\prod_{i=1}^{n-1}K_{\sigma_{n-1}(i)}\bigg]\bigg),\\
        &\Psi_n^2=\Id_{n\geq 3}\sum_{j=2}^{n-1}\Tr(P_0\hat B_{0,c}(\bar P_0\hat B_{0,c})^{n-1-j}\bar P_0\hat B_{0,c}^j),\\
        &\Psi_n^3=\Tr(P_0\hat B_{0,c}(\bar P_0\hat B_{0,c})^{n-1}).
        \end{aligned}
\end{equation}
In particular,
    \begin{equation}\label{263}
        \sum_{n=1}^\infty\frac{|\Psi_n^1|+|\Psi_n^2|+|\Psi_n^3|}{n}\leq \sum_{n=1}^{\infty} \frac{2^n \beta^n+n \beta^n+\beta^n}{n}<\infty.
    \end{equation}
\end{prop}
\begin{proof}
    Combining results in Lemma~\ref{kw2p},~\ref{kwiii} and~\ref{kwi} we have
    \begin{equation}
        \begin{aligned}
            &|\Tr(K_w^n)-(\kappa_n(c)L+\Psi_n^1-\Psi_n^2-\Psi_n^3)|\\
            &\leq 3^{n-1}\beta^ne^{-\frac{4L^3}{3n^2}}+2\cdot 3^{n-1}\beta^ne^{-\frac{4L^3}{3n^2}}+n^2\beta^{n}e^{-\frac{4L^3}{3n^2}}
            \leq 3^{n+1}\beta^ne^{-\frac{4L^3}{3n^2}}.
        \end{aligned}
    \end{equation}
\eqref{263} follows from \eqref{214}, \eqref{234} and \eqref{psi2n}.
\end{proof}

\subsubsection{Mixed $u,v$ Terms}
\begin{lem}\label{uvresult}
     Let $n\in\mathbb Z_{\geq 2},L,r\geq 1$, $r^2\leq 2c$ and $\sigma_n\in\{u,v\}^n$ with $u,v\in\sigma_n$, then
    \begin{equation}
        \bigg|\Tr\bigg(\prod_{i=1}^nK_{\sigma_n(i)}\bigg)\bigg|\leq\beta^ne^{-\frac{4L^3}{3}}.
    \end{equation}
\end{lem}
\begin{proof}\
Choose $\sigma_n\in\{u,v\}^n$ with $u,v\in\sigma_n$ . Using the cyclic property, we can assume without loss of generality $\sigma_n(1)=u$. Since $v\in\sigma_n$, it holds $2\leq \ell_v=\max\{i|\sigma_n(i)=v\}\leq n$.
    Then we have
    \begin{equation}
        \begin{aligned}
        \Phi=& P_0B_{0,c}\bigg[\prod_{i=2}^nK_{\sigma_n(i)}\bigg]P_0\\
        =& \underbrace{P_0B_{0,c}\bigg[\prod_{i=2}^{\ell_v-1}K_{\sigma_n(i)}\bigg]e^{L\Delta}P_0}_{=:\varphi_1}\cdot \underbrace{P_0e^{-L\Delta}B_{0,c}P_0}_{=:\varphi_2}\underbrace{(P_0B_{0,c})^{n-\ell_v}P_0}_{=:\varphi_3}.
        \end{aligned}
    \end{equation}
    Applying $\|U_r^{-1}\varphi_1U_r^{-1}\|_{\rm HS}\leq \beta^{\ell_v-1}e^{Lr^2}$ (by Lemma~\ref{uvwkey33}), $\|U_r\varphi_2U_r\|_{\rm HS}\leq\beta e^{-\frac{4L^3}{3}}$ (by~\eqref{1b4}), $\|U_r^{-1}\varphi_3U_r\|_{\rm HS}\leq\beta^{n-\ell_v}$ (by~\eqref{1ab3}) and $r^2\leq 2c$, we have
    \begin{equation}
        |\Tr(\Phi)|\leq\|U_r^{-1}\varphi_1U_r^{-1}\|_{\rm HS} \|U_r\varphi_2U_r\|_{\rm HS}\|U_r^{-1}\varphi_3U_r\|_{\rm HS}\leq\beta^ne^{-\frac{4L^3}{3}}
    \end{equation}
\end{proof}

\subsubsection{Mixed $u,w$ Terms}
First, we define a new operator $K_{\tilde w}=\bar P_0\tilde B_{0,c}$. For a word $\sigma_n\in\{u,w\}^n$, we define
\begin{equation}\label{wtotildew}
    \sigma_n^{w\mapsto\tilde w}(i)=\begin{cases}
        u,\quad&\textrm{if }\sigma_n(i)=u,\\
        \tilde w,&\textrm{otherwise}.
    \end{cases}
\end{equation}

\begin{lem}\label{uwresult}
  Let $n\in\mathbb Z_{\geq 2},L,r\geq 1$, $r^2\leq 2c$ and $\sigma_n\in\{u,w\}^n$ with $u,w\in\sigma_n$, it holds
    \begin{equation}\label{271}
        \bigg|\Tr\bigg(\prod_{i=1}^nK_{\sigma_n^{w\mapsto\tilde w}(i)}\bigg)\bigg|\leq \beta^n
    \end{equation}
    and
    \begin{equation}\label{272}
    \bigg|\Tr\bigg(\prod_{i=1}^nK_{\sigma_n(i)}\bigg)-\Tr\bigg(\prod_{i=1}^nK_{\sigma_n^{w\mapsto\tilde w}(i)}\bigg)\bigg|\leq3^n\beta^ne^{-\frac{4L^3}{3n^2}}.
    \end{equation}
\end{lem}
\begin{proof}
We first show~\eqref{271}, let $\hat\sigma_n\in\{u,\tilde w\}^n$ with $u,\tilde w\in\hat\sigma_n$, using cyclyc property, we assume without loss of generality that $\hat\sigma_n(1)=u$. We define
\begin{equation}
    \Phi_{\hat\sigma_n}=P_0B_{0,c}\bigg[\prod_{i=2}^n K_{\hat\sigma_n(i)}\bigg]P_0.
\end{equation}
Instead of \eqref{271}, we will use induction to show $\max\{|\Tr(\Phi_{\hat\sigma_n})|,\|U_r^{-1}\Phi_{\hat\sigma_n}U_r\|_{\rm HS}\}\leq\beta^n.$ For $n=2$, we have then $\Phi_{\hat\sigma_2}=P_0B_{0,c}\bar P_0\tilde B_{0,c}P_0$, we can then apply the bounds~\eqref{1ab3} and~\eqref{2ab3} to deduce
\begin{equation}
\max\{|\Tr(\Phi_{\hat\sigma_2})|,\|U_r^{-1}\Phi_{\hat\sigma_2}U_r\|_{\rm HS}\}\leq \|U_r^{-1}P_0B_{0,c}\bar P_0U_r\|_{\rm HS}\|U_r^{-1}\bar P_0\tilde B_{0,c}P_0U_r\|_{\rm HS}\leq \beta^2.
\end{equation}
For induction step $n-1\mapsto n$. Let us start with the case $\hat\sigma_n(i)=\tilde w$ for all $i\in\{2,\ldots,n\}$, then $\Phi_{\hat\sigma_n}=P_0B_{0,c}\bar P_0\cdot (\bar P_0\tilde B_{0,c}\bar P_0)^{n-2}\bar P_0\tilde B_{0,c}P_0$, then we can apply the bounds~\eqref{1ab3},~\eqref{2ab3} respectively.
Then we have
\begin{equation}
    |\Tr(\Phi_{\hat\sigma_n})|\leq\|U_r^{-1}P_0B_{0,c}\bar P_0U_r\|_{\rm HS}\|U_r^{-1}\bar P_0\tilde B_{0,c}\bar P_0U_r\|_{\rm op}^{n-2}\|U_r^{-1}\bar P_0\tilde B_{0,c} P_0U_r\|_{\rm HS}\leq\beta^n.
\end{equation}
Next suppose now there exists $i\in\{2,\ldots,n\}$ such that $\sigma_n(i)=u$. Then we have $2\leq \ell_u=\min\{i|\sigma_n(i)=u\}\leq n$ and hence
\begin{equation}
    \Phi_{\hat\sigma_n}=\underbrace{P_0B_{0,c}\bigg[\prod_{i=2}^{\ell_u-1}K_{\sigma_n(i)}\bigg]P_0}_{=:\varphi_1}\cdot\underbrace{ P_0B_{0,c}\bigg[\prod_{i=\ell_u+1}^{n}K_{\sigma_n(i)}\bigg]P_0}_{=:\varphi_2}.
\end{equation}
By induction assumption, we have $\|U_r^{-1}\varphi_1U_r\|_{\rm HS}\leq\beta^{\ell_u-1}$, $\|U_r^{-1}\varphi_2U_r\|_{\rm HS}\leq\beta^{n-\ell_u}$, the claim follows.

It remains to show \eqref{272}, let now $\sigma_n\in\{u,w\}^n$, using cyclic property, we assume $\sigma_n(1)=u$. Then there exists $m\geq 1,\ q_m\geq 1$ such that $\sigma_n$ is given as following:
\begin{equation}\label{sigma}
    \underbrace{u,\ldots,u}_{p_1 \textrm{times}},\underbrace{w,\ldots,w}_{q_1 \textrm{times}},\cdots,\underbrace{u,\ldots,u}_{p_m \textrm{times}},\underbrace{w,\ldots,w}_{q_m \textrm{times}}.
\end{equation}
Denote
\begin{equation}\label{e274}
\begin{aligned}
    &\Phi=P_0B_{0,c}\bigg[\prod_{i=2}^nK_{\sigma_n(i)}\bigg]P_0=\prod_{i=1}^m\left[(P_0B_{0,c})^{p_i}(\bar P_0\tilde B_{0,c}-e^{L\tilde\Delta}P_0e^{-L\Delta}B_{0,c})^{q_i}\right].
\end{aligned}
\end{equation}
Recall that $K_c=e^{L\tilde\Delta}P_0e^{-L\Delta}B_{0,c}$, together with \eqref{e274} and cyclic property, we have then
\begin{equation}\label{31back}
    \begin{aligned}
&\bigg|\Tr(\Phi)-\Tr\bigg(\prod_{i=1}^m(P_0B_{0,c})^{p_i}(\bar P_0\tilde B_{0,c})^{q_i}P_0\bigg)\bigg|\\
        \leq&\sum_{\sigma_{q_1}\in\{a,b,c\}^{q_1}}\cdots\sum_{\sigma_{q_m}\in\{a,b,c\}^{q_m}}\Id_{\{\exists\ i\textrm{ s.t. } c\in\sigma_{q_i}\}}|\Tr\left(\Phi_m\right)|,
    \end{aligned}
\end{equation}
where
\begin{equation}
    \Phi_m=\prod_{i=1}^m\bigg((P_0B_{0,c})^{p_i}\bigg[\prod_{j=1}^{q_i}K_{\sigma_{q_i}(j)}\bigg]P_0\bigg).
\end{equation}
Since there exists $i$ such that $c\in\sigma_{q_i}$, we can rewrite
\begin{equation}
    \Phi_m=(P_0B_{0,c})^{p_1-1}\underbrace{P_0B_{0,c}\bigg[\prod_{i=1}^{\ell_c-1}K_{i}\bigg]e^{L\tilde\Delta}P_0}_{=:\varphi_1}\cdot\underbrace{ P_0e^{-L\Delta}B_{0,c}\bigg[\prod^{n-p_1}_{i=\ell_c+1}K_{i}\bigg]P_0}_{=:\varphi_2}
\end{equation}
with $K_i\in\{K_u,K_a,K_b,K_c\}$ and
\begin{equation}
    \ell_c=q_1+\cdots +p_i+\min\{j|\sigma_{q_i}(j)=c\}.
\end{equation}
Applying the bounds $\|U_r^{-1}P_0B_{0,c}U_r\|_{\rm HS}\leq\beta$ (by~\eqref{1ab3}), $\|U_r^{-1}\varphi_1U_r^{-1}\|_{\rm HS}\leq\beta^{\ell_c}e^{Lr^2}$ (Corollary~\ref{ca15}), $\|U_r\varphi_2U_r\|_{\rm HS}\leq\beta^{n-p_1-\ell_c+1}e^{-\frac{4L^3}{3(n-p_1-\ell_c+1)^2}}e^{-2Lc}$ (Corollary~\ref{ca17}) and $r^2\leq 2c$, we then have
\begin{equation}
    |\Tr(U_r^{-1}\Phi_mU_r)|\leq\|U_r^{-1}P_0B_{0,c}U_r\|_{\rm HS}^{p_1-1} \|U_r^{-1}\varphi_1U_r^{-1}\|_{\rm HS}\|U_r\varphi_2U_r\|_{\rm HS}\leq\beta^{n}e^{-\frac{4L^3}{3n^2}},
\end{equation}
Plugging this back to \eqref{31back} and using the fact there are in total $3^{q_1+\cdots+q_m}-2^{q_1+\cdots+q_m}\leq 3^n$ summations, the proof is completed.
\end{proof}

\subsubsection{Mixed $v,w$ Terms}
Let $\sigma_n\in\{v,w\}^n$ and without loss of generality we can set $\sigma_n(1)=v$. Denote
    $K_\alpha=P_0\hat B_{0,c}$, $K_{\tilde\beta}=\bar P_0\hat B_{0,c}$ and $K_\gamma=\bar P_0B_{0,c}$. For a fixed $\sigma_n\in\{v,w\}^n$, we define its transformed word as
\begin{equation}\label{wtobeta}
    \sigma_n^{w\to\tilde\beta}(i)=\begin{cases}
        u,\quad&\textrm{if }\sigma_n(i)=\sigma_n(i+1)=v,\\
        \tilde\beta,&\textrm{if }\sigma_n(i)=\sigma_n(i+1)=w,\\
        \gamma,&\textrm{if }\sigma_n(i)=w,\sigma_n(i+1)=v,\\
        \alpha,&\textrm{if }\sigma_n(i)=v,\sigma_n(i+1)=w,\\
    \end{cases}
\end{equation}
where we set $\sigma_n(n+1)=\sigma_n(1)$.
\begin{lem}\label{vwresults}
   Let $n\in\mathbb Z_{\geq 2},L,r\geq 1$, $r^2\leq 2c$ and $\sigma_n\in\{v,w\}^n$ with $v,w\in\sigma_n$, it holds
    \begin{equation}\label{287}
        \bigg|\Tr\bigg(\prod_{i=1}^nK_{\sigma_n^{w\mapsto\tilde \beta}(i)}\bigg)\bigg|\leq\beta^n
    \end{equation}
    and
    \begin{equation}
        \bigg|\Tr\bigg(\prod_{i=1}^nK_{\sigma_n(i)}\bigg)-\Tr\bigg(\prod_{i=1}^nK_{\sigma_n^{w\mapsto\tilde \beta}(i)}\bigg)\bigg|\leq3^n\beta^ne^{-\frac{4L^3}{3n^2}}.
    \end{equation}
\end{lem}
\begin{proof}
One can obtain \eqref{287} using the same method for \eqref{271}, so we omit the proof here. Let $\sigma_n\in\{v,w\}^n$ with $v,w\in\sigma_n$. By cyclic property, we can assume $\sigma_n(1)=v$. In particular, there exists $m\geq 1,q_m\geq 1$ such that $\sigma_n$ is given as
\begin{equation}\label{vwsigma}
    \underbrace{v,\ldots,v}_{p_1\ many},\underbrace{w,\ldots,w}_{q_1\ many},\cdots,\underbrace{v,\ldots,v}_{p_m\ many},\underbrace{w,\ldots,w}_{q_m\ many}
\end{equation}
In this case we have
\begin{equation}
    \begin{aligned}
&\Tr\bigg(\prod_{i=1}^m\left[e^{L\Delta}(P_0B_{0,c})^{p_i-1}P_0e^{-L\Delta}B_{0,c}(\tilde B_{0,c}-P_0\tilde B_{0,c}-e^{L\tilde\Delta}P_0e^{-L\Delta}B_{0,c})^{q_i}\right]\bigg)\\      =&\Tr\bigg(\prod_{i=1}^m\left[(P_0B_{0,c})^{p_i-1}P_0e^{-L\Delta}B_{0,c}(e^{L\tilde\Delta}\bar P_0e^{-L\Delta}B_{0,c}-P_0\tilde B_{0,c})^{q_i}e^{L\Delta}P_0\right]\bigg),
    \end{aligned}
\end{equation}
where we use $e^{-L\Delta}B_{0,c}e^{L\Delta}=B_{0,c}$ to deduce $K_v^{p_i}=e^{L\Delta}(P_0B_{0,c})^{p_i-1}P_0e^{-L\Delta}B_{0,c}.$ Applying the identity $e^{-L\Delta}B_{0,c}e^{L\tilde\Delta}=\hat B_{0,c}$ (see~\eqref{a102}), we have then
\begin{equation}\label{33back}
    \begin{aligned}
        &\bigg|\Tr\bigg(\prod_{i=1}^nK_{\sigma_n(i)}\bigg)-\Tr\bigg(\prod_{i=1}^nK_{\sigma_n^{w\mapsto\tilde \beta}(i)}\bigg)\bigg|\\
        \leq&\sum_{
        \substack{ \sigma_{q_1}\in\{a,b,c\}^{q_1}}
       }\cdots \sum_{
        \substack{ \sigma_{q_m}\in\{a,b,c\}^{q_m}
       }
       }\Id_{\{\exists\ j\textrm{ s.t. }b\in\sigma_{q_j}\}}|\Tr(\Phi)|
\end{aligned}
\end{equation}
where
\begin{equation}
    \Phi=\prod_{i=1}^m\bigg((P_0B_{0,c})^{p_i-1}P_0e^{-L\Delta}B_{0,c}\bigg[\prod_{k=1}^{q_i}K_{\sigma_{q_i}(k)}\bigg]e^{L\Delta}P_0\bigg).
\end{equation}
Let $j\in\{1,\ldots,m\}$ such that $b\in\sigma_{q_j}$ and define
\begin{equation}
    \ell_b=q_1+\cdots+p_j+\min\{i|\sigma_{q_j(i)}=b\}.
\end{equation}
Then we have
\begin{equation}
    \Phi=\underbrace{(P_0B_{0,c})^{p_1-1}}_{=:\varphi_1}\underbrace{P_0e^{-L\Delta}B_{0,c}\left[\prod_{i=1}^{\ell_b-1}K_{\sigma_{i}(k)}\right]P_0}_{=:\varphi_2}\cdot \underbrace{P_0\tilde B_{0,c}\left[\prod_{i=\ell_b+1}^{n-p_1}K_i\right]e^{L\Delta}P_0}_{=:\varphi_3}
\end{equation}
with $K_i\in\{K_a,K_b,K_c,K_u\}$. Applying the bounds $\|U_r\varphi_1U_r^{-1}\|_{\rm HS}\leq\beta^{p_1-1}$ (by~\eqref{1ab3}), $\|U_r\varphi_2U_r\|_{\rm HS}\leq \beta^{\ell_b}e^{-\frac{4L^3}{3\ell_b^2}}e^{-2Lc}$ (by Corollary~\ref{ca17}), $\|U_r^{-1}\varphi_3U_r^{-1}\|_{\rm HS}\leq\beta^{n-p_1-\ell_b+1}e^{Lr^2}$ (by Corollary~\ref{ca16}) and $r^2\leq 2c$, we have
\begin{equation}
\begin{aligned}
    &|\Tr(\Phi)|\leq\|U_r\varphi_1U_r^{-1}\|_{\rm HS}\|U_r\varphi_2U_r\|_{\rm HS}\|U_r^{-1}\varphi_3U_r^{-1}\|_{\rm HS}\leq\beta^{n}e^{-\frac{4L^3}{3n^2}}.
\end{aligned}
\end{equation}
Plugging this back to \eqref{33back} and noticing that the number of summations is smaller than $3^n$, we obtain the result.
\end{proof}

\subsubsection{Mixed $u,v,w$ terms}
\begin{lem}\label{uvwresults}
    Let $n\in\mathbb Z_{\geq 3},L,r\geq 1$, $r^2\leq 2c$ and $\sigma_n\in\{u,v,w\}^n$ with $u,v,w\in\sigma_n$. Then
    \begin{equation}
        \bigg|\Tr\bigg(\prod_{i=1}^nK_{\sigma_n}(i)\bigg)\bigg|\leq 5^n\beta^ne^{-\frac{4L^3}{3n^2}}.
    \end{equation}
\end{lem}
\begin{proof}
    Without loos of generality we assume $\sigma_n(1)=u$. Then we have
    \begin{equation}
        \begin{aligned}
&\bigg|\Tr\bigg(\prod_{i=1}^nK_{\sigma_n}(i)\bigg)\bigg|=\bigg|\Tr\bigg(P_0B_{0,c}\prod_{i=2}^nK_{\sigma_n}(i)P_0\bigg)\bigg|\\
\leq&\sum_{\substack{\sigma_n\in \Sigma_n
}}\bigg|\Tr\bigg(P_0B_{0,c}\prod_{i=2}^nK_{\sigma_n}(i)P_0\bigg)\bigg|,
        \end{aligned}
    \end{equation}
    where
    \begin{equation}
        \Sigma_n=\left\{\sigma_n\in\{a,b,c,u,v\}^n|
\sigma_n\not\in\{a,b,c\}^n,
\sigma_n\not\in\{a,b,c,u\}^n,
\sigma_n\not\in\{a,b,c,v\}^n\right\}.
    \end{equation}
For $\sigma_n\in\Sigma_n$, we define
    \begin{equation}
    \begin{aligned}
        &        \Phi=P_0B_{0,c}\prod_{i=2}^nK_{\sigma_n}(i)P_0\\
=&\underbrace{P_0B_{0,c}\bigg[\prod_{i=2}^{\ell_v-1}K_{\sigma_n}(i)\bigg] e^{L\Delta}P_0}_{=:\varphi_1}\cdot\underbrace{ P_0e^{-L\Delta}B_{0,c}\bigg[\prod_{i=\ell_v+1}^{n}K_{\sigma_n}(i)\bigg] P_0}_{=:\varphi_2},
    \end{aligned}
    \end{equation}
where $2\leq\ell_v=\min\{i\mid\sigma_n(i)=v\}\leq n.$ Applying the bounds $\|U_r^{-1}\varphi_1U_r^{-1}\|_{\rm HS}\leq\beta^{\ell_v-1}e^{Lr^2}$ (by Corollary~\ref{uvwkey33}), $\|U_r\varphi_2U_r\|_{\rm HS}\leq\beta^{n-\ell_v+1}e^{-\frac{4L^3}{3(n-\ell_v+1)^2}}e^{-2Lc}$ (by Corollary~\ref{ca17}) and $r^2\leq 2c$, we have
\begin{equation}
    |\Tr(\Phi)|\leq\|U_r^{-1}\varphi_1U_r^{-1}\|_{\rm HS}\|U_r\varphi_2U_r\|_{\rm HS}\leq\beta^{n}e^{-\frac{4L^3}{3n^2}}.
\end{equation}
The result follows by the fact that
    $|\Sigma_n|\leq 5^n$.
\end{proof}

\subsection{Proof of Theorem~\ref{main}}\label{sec24}
Now we are able to prove Theorem~\ref{main}. Before going to the main part, we need to control the upper bound of error terms, that is,
\begin{equation}
    \beta=\max\{2e^{r^3/3-2rc},e^{(r-1/7)^3/3-2(r-1/7)c}\}
\end{equation}
with $r^2\leq 2c, r>1$. For fixed $c$, we set $r=\sqrt{2c}$, then $\beta\leq 2e^{-\frac{4\sqrt{2}}{3}c^{3/2}}.$ In particular, for $c\geq 3/2$, $\beta<1/7$. Recall that
        \begin{equation}\label{mainback}
            \ln(\Pb(\mathcal A_1(s)\leq c,\ s\in[0,L]))=\ln(\det(\Id-K_{L,c}))=-\sum_{n=1}^\infty\frac{1}{n}\Tr(K_{L,c}^n).
        \end{equation}
In Proposition~\ref{klc1} we have obtained
        \begin{equation}
            |\Tr(K_{L,c})-2\Tr(P_0B_{0,c})+2L\Ai'(2c)|\leq e^{-\frac{4L^3}{3}}.
        \end{equation}
For $n\geq 2$, applying Lemma~\ref{forsimple} to \eqref{kln} we get
\begin{equation}\label{eq2.128}
    \bigg|\Tr(K_{L,c}^n)-\sum_{\sigma_n\in\{u,v,w\}^n}\Tr\bigg(\prod_{i=1}^nK_{\sigma_n(i)}\bigg)\bigg|\leq 7^n e^{-\frac{4L^3}{3n^2}}\beta^n.
\end{equation}
Furthermore, by Lemma~\ref{lemEst1}, Proposition~\ref{kwkey}, Lemma~\ref{uvresult}, Lemma~\ref{uwresult}, Lemma~\ref{vwresults} and Lemma~\ref{uvwresults} we get
\begin{equation}
\sum_{\sigma_n\in\{u,v,w\}^n}\Tr\bigg(\prod_{i=1}^nK_{\sigma_n(i)}\bigg) = \kappa_n(c) L + \Psi_n + R_n
\end{equation}
where
    \begin{equation}
        \begin{aligned}
            \Psi_n=&2\Tr((P_0B_{0,c})^n)+\Id_{n\geq 2}\bigg[\Psi_n^1-\Psi_n^2-\Psi_n^3\\
            &+\sum_{\substack{\sigma_n\in\{u,w\}^n\\
        u,w\in\sigma_n
        }}\Tr\bigg(\prod_{i=1}^nK_{\sigma_n^{w\mapsto\tilde w}(i)}\bigg)+\sum_{\substack{\sigma_n\in\{v,w\}^n\\
        v,w\in\sigma_n
        }}\Tr\bigg(\prod_{i=1}^nK_{\sigma_n^{w\mapsto\tilde\beta}(i)}\bigg)\bigg],
        \end{aligned}
    \end{equation}
    $\Psi_n^1$, $\Psi_n^2$ and $\Psi_n^3$ are defined in Proposition~\ref{kwkey}, $\kappa_n(c)=-2n^{-2/3}\Ai'(2n^{2/3}c)$ and
\begin{equation}
|R_n|\leq (5^n + 4\cdot 3^n) \beta^n e^{-\frac{4 L^3}{3n^2}}+\beta^n e^{-\frac{4 L^3}{3}}\leq 4\cdot 5^n \beta^n e^{-\frac{4 L^3}{3 n^2}}.
\end{equation}
Together with \eqref{eq2.128} we have
\begin{equation}
\Tr(K_{L,c}^n) = \kappa_n(c) L+ \Psi_n+ \tilde R_n
\end{equation}
where $|\tilde R_n|\leq (7^n+4\cdot 5^n) \beta^n e^{-\frac{4 L^3}{3 n^2}}\leq7^{n+1} \beta^n e^{-\frac{4 L^3}{3 n^2}}$. By \eqref{242}, \eqref{263}, \eqref{271} and \eqref{287} we have the following bound for the $L$-independent terms
\begin{equation}\label{21110}
\bigg|\sum_{n\geq 1}\frac{\Psi_n}{n}\bigg|\leq 6\sum_{n=1}^\infty\frac{2^n\beta^n}{n}=-6 \ln(1-2\beta)<\infty,
\end{equation}
where we use $\beta<1/7.$ Next, we want to show $\sum_{n\geq 1}|\tilde R_n|\to 0$ as $L\to\infty$. Set $\alpha=7\beta<1$ and notice that
    \begin{equation}
        \begin{aligned}
&\sum_{n=1}^\infty\frac1n\alpha^ne^{-\frac{4L^3}{3n^2}}\leq \sum_{n=1}^{\infty}\alpha^n e^{-\frac{4 L^3}{3 n^2}}.
        \end{aligned}
    \end{equation}
The function $f_{L,\alpha}:\ \N\to\R$ given by
    \begin{equation}
        f_{L,\alpha}(n):=\alpha^n e^{-\frac{4 L^3}{3 n^2}}
    \end{equation}
is increasing on $(1,n_0)$ and decreasing on $[n_0,\infty)$ with
    \begin{equation}
        n_0=2 \cdot 3^{-1 / 3}  \ln(\alpha^{-1})^{-1/3} L.
    \end{equation}
In particular, $f_{L, \alpha}^{\prime}(n)>0$ for all $n<n_0$ and $f_{L, \alpha}^{\prime}(n)<0$ for all $n>n_0$. Together with Riemann approximation, we have then
\begin{equation}\label{2139}
    \begin{aligned}
& \sum_{n=1}^{\infty} \alpha^n e^{-\frac{4 L^3}{3 n^2}}=\sum_{n=1}^{\left\lceil n_0\right\rceil} \alpha^n e^{-\frac{4 L^3}{3 n^2}}+\sum_{n=\left\lceil n_0\right\rceil+1}^{\infty} \alpha^n e^{-\frac{4 L^3}{3 n^2}} \\
\leq & \left\lceil n_0\right\rceil f_{L, \alpha}\left(n_0\right)+\int_{\left\lceil n_0\right\rceil}^{\infty} d n f_{L, \alpha}(n) \leq\left\lceil n_0\right\rceil f_{L, \alpha}\left(n_0\right)+\int_1^{\infty} d n f_{L, \alpha}(n) .
\end{aligned}
\end{equation}

We define $C_\alpha=\left(-\frac{1}{\ln (\alpha)}\right)^{1 / 3}>0$ for $\alpha \in(0,1)$. Then we have
\begin{equation}
    \left\lceil n_0\right\rceil f_{L, \alpha}\left(n_0\right) \leq\left(n_0+1\right) f_{L, \alpha}\left(n_0\right)=\alpha^{3^{2 / 3} L C_\alpha}\left(\frac{2 L C_\alpha}{3^{1 / 3}}+1\right)
\end{equation}
Since $C_\alpha>0$ and $\alpha<1$, for $L$ large, we have then
\begin{equation}
    \left\lceil n_0\right\rceil f_{L, \alpha}\left(\left\lceil n_0\right\rceil\right) \leq \tilde{C}_\alpha e^{-\delta L},
\end{equation}
where $\delta, \tilde{C}_\alpha>0$ independent of $L$. As for the integral term in \eqref{2139} note that
        \begin{equation}
            \int_1^\infty dx\alpha^xe^{-\frac{4L^3}{3x^2}}=\int_1^{L^\gamma}dx\alpha^xe^{-\frac{4L^3}{3x^2}}+\int_{L^\gamma}^\infty dx\alpha^xe^{-\frac{4L^3}{3x^2}}
        \end{equation}
with arbitrary $\gamma\geq1$. Since $\alpha<1$, the first integral is bounded by $(L^\gamma-1)e^{-\frac{4L^3}{3L^{2\gamma}}}$. On the other hand, the second integral is bounded by
    \begin{equation}
        \int_{L^\gamma}^\infty dx \alpha^x=\ln(\alpha^{-1})^{-1}\alpha^{L^\gamma}
    \end{equation}
Choosing now $\gamma=1$, we then see that
\begin{equation}
    0\leq \sum_{n=1}^\infty\frac1n\alpha^ne^{-\frac{4L^3}{3n^2}}\leq Le^{-\frac{4L}{3}}+\ln(\alpha^{-1})^{-1}e^{L\ln\alpha}\leq e^{-\delta L}
\end{equation}
with some $\delta>0$, which then finishes the proof.

\section{Proof of Proposition~\ref{p12}}
In Theorem~\ref{main}, we get the persistence exponent for $c\geq\frac32$, in this section, we try to extend our result to the whole real line via analytic continuation. To this end, we need first show the existence of the following quantities
    \begin{equation}
        \lim_{L\to\infty }\frac{\ln(\Pb\left(A_1(t)\leq c,\ \forall t\in[0,L]\right))}{L}.
    \end{equation}
Recall that a function is called super-additive if
    \begin{equation}
        f(x+y)\geq f(x)+f(y),\quad\forall x,y\in\R.
    \end{equation}
A nice property of supperadditive function is given by Fekete's Lemma.
\begin{thm}[Theorem 16.2.9 of~\cite{HP48}]\label{thmSuperadd}
    Let $f:(0,\infty)\to\R$ be a super-additive function. Then $\lim_{t\to\infty}f(t)/t$ exists.
\end{thm}
To apply Fekete's lemma in our case, we define
\begin{equation}
    f(L)=\ln(\Pb\left(\mathcal A_1(t)\leq c,\ \forall t\in[0,L]\right)).
\end{equation}
We have
\begin{equation}\label{eqsuper}
\begin{aligned}
\Pb\left(\max_{0\leq u\leq L_1+L_2}\mathcal A_1(u)\leq c\right) &= \Pb\left(\max_{0\leq u\leq L_1}\mathcal A_1(u)\leq c,\max_{L_1\leq u\leq L_1+L_2}\mathcal A_1(u)\leq c\right)\\
&\geq \Pb\left(\max_{0\leq u\leq L_1} \mathcal A_1(u)\leq c\right)\Pb\left(\max_{L_1\leq u\leq L_1+L_2}\mathcal A_1(u)\leq c\right)\\
&=\Pb\left(\max_{0\leq u\leq L_1} \mathcal A_1(u)\leq c\right)\Pb\left(\max_{0\leq u\leq L_2}\mathcal A_1(u)\leq c\right)
\end{aligned}
\end{equation}
for all $c\in\R$, $L_1,L_2>0$. For the last step we use translation-invariance of the law of the Airy$_1$ process.

To prove \eqref{eqsuper} we follow the proof of Lemma~3.2 of~\cite{BBF22}: we start with the line-to-point last passage percolation (LPP) model. Consider now the rescaled LPP (see (2.4) in \cite{BBF22} for precise definition)
\begin{equation}
	L_N(u)=\max_{v} \frac{L_{I(u), J(v)}-4 N}{2^{4/3} N^{1/3}}
\end{equation}
with $I(u)=u(2N)^{2/3}(1,-1)$ and $J(v)=(N,N)+v(2N)^{2/3}(1,-1)$.
It was known that
\begin{equation}
	\lim_{N \to \infty}2^{-1/3} L_N(2^{2/3}u)=\mathcal{A}_1\left( u\right),
\end{equation}
(in TASEP finite-dimensional distribution is proven in~\cite{Sas07,BFPS06} and by slow-decorrelation~\cite{Fer08,CFP10b} the result is translated to the LPP setting.) Together with the tightness~\cite{Pim17}, this implies that
\begin{equation}
	\lim_{N \to \infty}\max_{u\in[0,L]} L_N(u)=\max_{u\in[0,L]}\mathcal{A}_1\left( u\right),\quad\forall L\geq 0.
\end{equation}
Now \eqref{eqsuper} follows from FKG inequality (see Lemma 2.1 of \cite{Kesten03}) and the fact that the events $\{\max_{u\in[0,L_1]} L_N(u)\leq c\}$ and $\{\max_{u\in[L_1,L_1+L_2]} L_N(u)\leq c\}$ are both decreasing in the randomness.

To show that the Airy$_1$ process is positively correlated (also called associated in the language of~\cite{Li85b}), a similar argument, but using more involved results as input (the convergence of the KPZ equation to the KPZ fixed point~\cite{QS20}), was presented in~\cite{Pu23}.

Taking the logarithms in \eqref{eqsuper} we get that $f$ is super-additive and by Theorem~\ref{thmSuperadd}, the proof of Proposition~\ref{p12} is completed.

\section{Proof of Proposition~\ref{p13}}\label{pp13}
 In Theorem~\ref{main}, we have already showed that
\begin{equation}
\kappa(c)=-2\sum_{n=1}^\infty n^{-5/3}\Ai'(2n^{2/3}c),\quad\forall c\geq\frac32.
\end{equation}
This function is analytic for all $c>0$ as well. For instance, applying Fubini's theorem and integral representation of Airy function, we obtain
    \begin{equation}\label{eq4.2}
        \begin{aligned}
            &\kappa(c)=-2\sum_{n=1}^\infty n^{-5/3}\frac{1}{2\pi \I}\int_{e^{-\pi \I/3}\infty }^{e^{\pi \I/3}\infty }dw\,w e^{\frac{w^3}{3}-2n^{2/3}wc}\\
            \stackrel{w\mapsto n^{1/3}w}{=}&\frac{-1}{\pi \I}\int_{e^{-\pi \I/3}\infty}^{e^{\pi \I/3}\infty}dw\, w\sum_{n=1}^\infty \frac{(e^{w^3/3-2wc})^n}{n}
            =\frac{-1}{\pi \I}\int_{e^{-\pi \I/3}\infty}^{e^{\pi \I/3}\infty}dw\, w\ln\left(1-e^{w^3/3-2wc}\right).
        \end{aligned}
    \end{equation}

\begin{figure}[h!]
    \centering
    \includegraphics[width=0.6\textwidth]{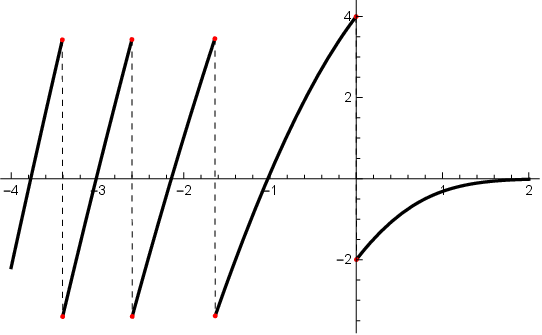}
    \caption{The black solid line is the graph of $(c,\kappa'(c))$. The red point is the jump points in $\mathcal J,$ that is, $c(0),c(1),\ldots$ defined in Lemma~\ref{pb3}. In particular, function $\kappa'(c)$ is analytic on each interval $(c(i+1),c(i)).$ The analytic continuation of $\kappa'(c)|_{(0,\infty)}$ is obtained by gluing $\kappa'(c)|_{(c(i+1),c(i))}$ together in a smooth way.}
    \label{discontinuity}
\end{figure}

To avoid dealing the branch cut of $\ln$ function, we consider $\kappa'(c)$ instead of $\kappa(c)$,
\begin{equation}\label{c11}
    \begin{aligned}
        \kappa'(c)=\frac{2}{\pi \I}\int_{\Gamma}dw\frac{w^2e^{\frac{w^3}{3}-2wc}}{1-e^{\frac{w^3}{3}-2wc}},
    \end{aligned}
\end{equation}
where from now on we \emph{fix} the integration contour as follows
\begin{equation}
\Gamma=\{|r|e^{{\rm sgn}(r)\pi \I/3}\textrm{ s.t. } r\in\R\}
\end{equation}
oriented by increasing imaginary part. Different choices of $\Gamma$ gives rise to different (equivalent) formulas for the analytic continuation. The reason is that when decreasing $c$, there are zeroes of the denominator crossing the contour $\Gamma$.

Define
\begin{equation}
f(w,c)=1-e^{\frac{w^3}{3}-2wc}\quad\textrm{and}\quad g(w,c)=w^2e^{\frac{w^3}{3}-2wc}.
\end{equation}
First we determine the values of $c$ where the denominator of \eqref{c11} vanishes (the poles). It turns out that these values are exactly the discontinuity points appearing in Figure~\ref{discontinuity} of function $\kappa'(c)$. Define the set
    \begin{equation}
        \mathcal J=\{c\in\R_-| f(w,c)=0\textrm{ for some }w\in\Gamma\}.
    \end{equation}
\begin{lem}\label{pb3}
We have
\begin{equation}
\mathcal J=\left\{-(2n\pi/3)^{2/3}| n\in\Z_{\geq0}\right\}
\end{equation}
Moreover, for $c(n)=-(2n\pi/3)^{2/3}$, $f(c(n),w(n))=f(c(n),\bar w(n))=0$ for $w(n)=3^{1/2} (2\pi n/3)^{1/3} e^{\pi \I/3}$.
\end{lem}
\begin{figure}[h!]
    \centering
    \includegraphics[width=0.3\textwidth]{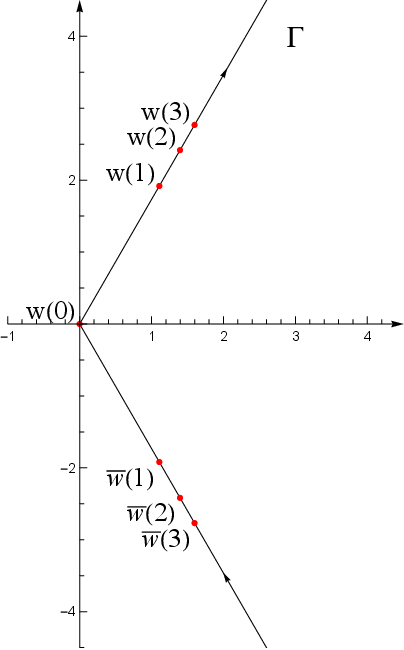}
    \caption{The black solid line is the integral contour $\Gamma$, the red points are $w(n)$ and $\bar w(n)$, $n\geq 0$. }
    \label{w}
\end{figure}
\begin{proof}
By symmetry with respect to the real axis, it is clear that there are complex-conjugated zeroes of the denominator. So parameterize $\Gamma$ on the upper half plane by $w=re^{\pi i /3}$ with $r\geq 0$. We have $f(w,c)=0$ if and only if both the real and the imaginary parts are zero. This happens if, for some $n\in\Z$, $w^3/3-2wc=2\pi \I n$, that is,
\begin{equation}\label{c15}
       \frac{r^3}{3}+cr=0\quad \textrm{and}\quad  \sqrt{3}cr=2\pi n.
\end{equation}
We can restrict to $n\geq 0$ since the other gives the complex conjugate solutions. The solution of \eqref{c15} are precisely the pairs given by $(c(n),w(n))$ of the lemma.
\end{proof}

For $c\in\R$ and $n\in\Z_{\geq 0}$, we denote by $w_{n,c}\in \C$ with $\Re(w)\geq 0$ such that $w_{n,c}^3/3-2w_{n,c}c=2\pi \I n$. Let $W_0=\{w_{0,c}| 0\leq c\leq 3/2\}$ and $W_n=\{w_{n,c}| c(n)\leq c\leq 0\}$ for $n\geq 1$. We also define $\overline{W}_{n}$ as the conjugate set of $W_n.$ Furthermore, we denote $L_\Gamma$ (resp. $R_\Gamma$) as the set of points that are to the left (resp. right) of contour $\Gamma$, that is, $L_\Gamma=\{z\in\C| \rm{arg}(z)\in (\pi/3,5\pi/3)\}$.
\begin{figure}[h!]
    \centering
    \includegraphics[width=0.3\textwidth]{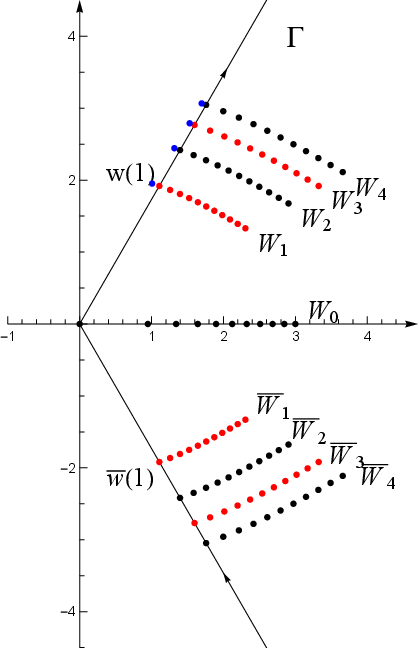}
    \caption{An Illustration for set $W_{n}$ and $\overline W_{n}.$ The black solid line is the original integral contour $\Gamma$. The black points for $W_0$ is $w_{0,c}$ for $c=3/2-3m/20$ with $m\in\{0,1,\ldots,10\}$. The red and black points are the set $W_n,\overline{W}_n$ with $n\in \{1,2,3,4\}$ are $w_{n,c}$ for $c=c(n)-m c(n)/10$ with $m\in\{1,2,\ldots,10\}$. The four blue points from bottom left to top right is $w_{n,c(n)-1/10}$ with $n=1,2,3,4$. In particular, note that $w_{n,c(n)-1/10}\in L_\Gamma$ and $d(W_m,W_n)=d(\overline{W}_m,\overline{W}_n)>0$. }\label{discont}
\end{figure}

\begin{lem}\label{distance}$ $\\
(a) For $n\in\N$,
$w_{n,c}\in L_{\Gamma}$ for any $c<c(n)$ and $w_{n,c}\in R_{\Gamma}$ for any $0\geq c>c(n).$

\noindent (b) For any $m,n\in\N$ with $m\neq n$, it holds $W_n\cap W_m=\varnothing.$ See Figure~\ref{discont} for an illustration.
\end{lem}
\begin{proof}
Let us prove (a). Parameterize $w=re^{\I\phi}$ with $(r,\phi)\in\R_+\times (0,\pi/2)$. The condition $\frac{w^3}{3}-2wc=2\pi \I n$ for some $c<0$ and $n\in\Z_{\geq 0}$ is then equivalent to
    \begin{equation}
       c=h(r,\phi)=\tfrac{1}{6} r^2 \cos (2 \phi )-\frac{n\pi \sin (\phi )}{r}+\tfrac{1}{6} \I r^2 \sin (2 \phi )-\frac{\I n \pi  \cos (\phi )}{r}.
    \end{equation}
Since $c\in\R$, we need to have $\Im(h(r,\phi))=0$. This and the condition $r>0$ leads to the following relation between $r$ and $\phi$:
    \begin{equation}\label{b19}
        r_{n,\phi}=\left(6 n\pi {\cos (\phi )} {\sin(2 \phi )^{-1}}\right)^{1/3},
    \end{equation}
now we define
    \begin{equation}
        k(\phi)=h(r_{n,\phi},\phi)=3^{-1/3}n^{2/3}\pi ^{2/3}\sin(\phi)^{-4/3}(\cot(\phi)\cot(2\phi)-1) .
    \end{equation}
Its derivative is given by
\begin{equation}\label{412}
        k'(\phi)=3^{-4/3}n^{2/3}\pi ^{2/3} \cos(\phi)\sin(\phi)^{-5/3}(4\cos(2\phi)-5)<0,
    \end{equation}
where we use the fact $\cos(\phi),\sin(\phi)>0$ for $\phi\in(0,\pi/2)$ and $\cos(2x)\leq 1$ for all $x.$ In particular, this implies that $k(\phi)$ is monotone decreasing. On the other hand, we have $k(\pi/3)=c(n),$ this implies that for any $c<c(n)$, ${\rm arg}(w_{n,c})>\frac\pi 3$ for any $w_{n,c}\in W_{n}$, which shows the first claim.

Next we show (b). Choose now $m,n\in\N$ with $m\neq n$, it is enough to show that the trajectory of $W_n$ and $W_m$ will not intersect with each other. Note that by definition, $w_{n,0}$ is the solution of $w^3/3=2\pi \I n$ for $n\in\N$, this implies that ${\rm arg}(w_{n,0})=\pi/6$, together with \eqref{412}, we have
 \begin{equation}
        W_n=\left\{r_{n,\phi}e^{i\phi}| \phi\in(\pi/6,\pi/3)\right\},\quad W_m= \left\{r_{m,\phi}e^{i\phi}| \phi\in(\pi/6,\pi/3)\right\},
    \end{equation}
where $r_{n,\phi}$ is defined as in \eqref{b19}.

For a fixed $\phi\in[\pi/6,\pi/3]$, it is clear that $r_{m,\phi}\not=r_{n,\phi}$ when $m\neq n,$ which then implies $W_n\cap W_m=\varnothing$ for $m,n\in\mathbb Z_{\geq 1}$ with $m\neq n$. Note that $W_0=\{\sqrt{6c}| 0\leq c\leq 3/2\}$ and hence $W_0\cap W_n=\varnothing$ for any $n\geq 1$, since $r_{n,\phi}>0$ and $\pi/6\leq\phi\leq \pi/3.$ This completes then the proof.
\end{proof}
With this in hand, we are able to do the analytic continuation of $\kappa'(c)$ from $c>0$ to the whole real line. Denote now $\tilde\kappa'(c)$ as the function obtained by extending $\kappa'(c)$ analytically from $(0,\infty)$ to all $\R$.
\begin{figure}[h!]
    \centering
    \includegraphics[width=0.3\textwidth]{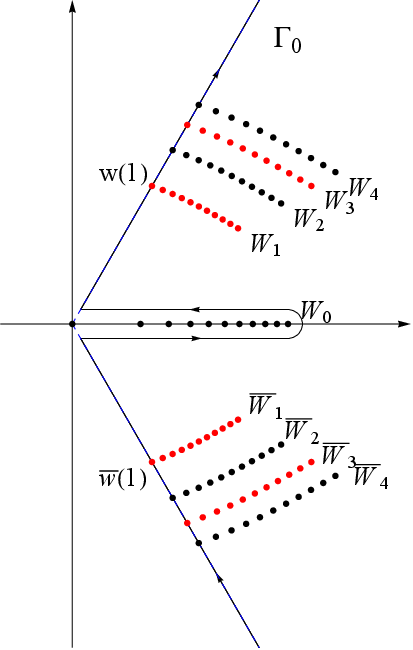}
    \caption{The blue dashed line is the original contour $\Gamma$ and the black line is the deformed contour $\Gamma_0$. In particular, the region between $\Gamma$ and $\Gamma_0$ contain only $W_0$ but not $W_n,\ \overline{W}_n$ for any $n\geq 1.$}\label{contour}
\end{figure}
We first consider the extension from $(0,\infty)$ to $(c(1),\infty)$. Now we deform the contour $\Gamma$ to the contour $\Gamma_0$ (see also Figure~\ref{contour} for an illustration) such that the region between $\Gamma$ and $\Gamma_0$ contains only $W_0$ but not $W_n,\overline{W}_n$ for any $n\geq 1$, this is possible by Lemma~\ref{distance}. Denote the integrand in \eqref{c11} as
\begin{equation}
    Q(w,c)=\frac{g(w,c)}{f(w,c)}.
\end{equation}
By Cauchy residue theorem, we know that
    \begin{equation}
    \begin{aligned}
       &   \kappa'(c)=\frac{1}{2\pi \I}\int_{\Gamma} dw Q(w,c)\\
       =&\frac{1}{2\pi \I}\int_{\Gamma_0}dw Q(w,c)-{\rm Res}\left(Q(w,c)| w=\sqrt{6c}\right),\quad\forall c\in(0,1).
    \end{aligned}
    \end{equation}
For $c>0$, $\sqrt{6c}$ is a simple pole for $Q(w,c),$ we then have
\begin{equation}
    {\rm Res}\left(Q(w,c)| w=\sqrt{6c}\right)=\lim_{w\to\sqrt{6c}}\frac{(w-\sqrt{6c})g(w,c)}{f(w,c)}=6.
\end{equation}
By the definition of $\Gamma_0$, there exists no $w\in\Gamma_0$ such that $f(w,c)=0$ for some $c\in(c(1),0)$, moreover, $Q(w,c)$ is bounded on $(w,c)\in \Gamma_0\times (c(1),0)$, hence (see for instance Lemma~B.2 of~\cite{BCFV14}), the function
    \begin{equation}
        c\mapsto \frac{1}{2\pi \I} \int_{\Gamma_0}dwQ(w,c)
    \end{equation}
is analytic on $(c(1),1)$. On the other hand, by the choice of $\Gamma_0,$ the region between $\Gamma$ and $\Gamma_0$ does not contain any $w_{n,c}$ for $n\in\N,\ c\in(c(1),0)$, we have
\begin{equation}
    \kappa'(c)=\frac{1}{2\pi \I}\int_{\Gamma}dw Q(w,c)=\frac{1}{2\pi \I}\int_{\Gamma_0}dwQ(w,c),\quad\forall c\in(c(1),0).
\end{equation}

\begin{figure}[h!]
    \centering
    \includegraphics[width=0.3\textwidth]{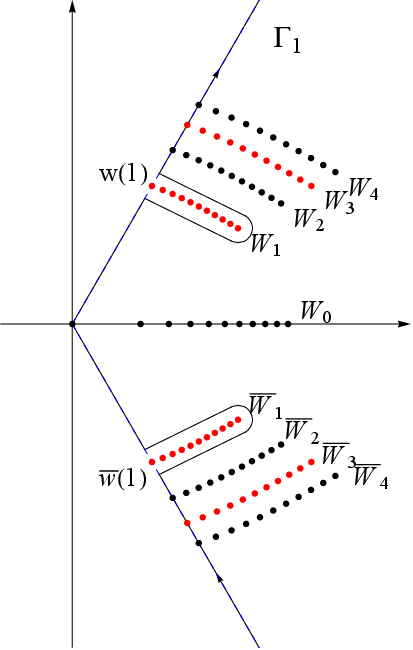}
    \caption{Illustration of $\Gamma_1.$ The blue dashed line is the original integral contour $\Gamma$, while the black solid line is the new integral contour $\Gamma_1$. In particular, the region between $\Gamma$ and $\Gamma_1$ should only contain $W_{1}$ and $\overline{W}_{1}$ but not $W_{n}\cup\overline{W}_{n}$ for $n\geq 2.$}
    \label{gamma1}
\end{figure}
The value of $\tilde\kappa'(c)$ at $c=0$ is then given by $\lim_{c\downarrow 0}\kappa'(c)$. In conclusion, the analytic extension $\tilde\kappa'(c)$ on $(c(1),\infty)$ is given by
\begin{equation}
    \tilde \kappa'(c)=\begin{cases}
        \kappa'(c),&\textrm{for }c>0,\\
        \lim_{c\downarrow 0}\kappa'(c),&\textrm{for } c=0,\\
        \kappa'(c)-6,&\textrm{for } c\in(c(1),0).
    \end{cases}
\end{equation}
Using the same method, we can also extend the result to $(c(2),\infty)$, namely, we choose the contour $\Gamma_1$ such that the region between $\Gamma_1$ and $\Gamma$ contains only $W_1$ and $\overline{W}_1$ but not $W_n,\overline{W}_n$ for any $n\neq 1$ (see Figure~\ref{gamma1}). Similar as above, the function
\begin{equation}
    c\mapsto \frac{1}{2\pi \I}\int_{\Gamma_1}dw Q(w,c)
\end{equation}
is analytic on $(c(2),0)$. And for $c\in (c(1),0)$, we have
\begin{equation}
\begin{aligned}
   &    \tilde\kappa'(c)=\int_{\Gamma}dw Q(w,c)-6\\
   =&\int_{\Gamma_1}dw Q(w,c)-6-{\rm Res}\left(Q(w,c)| w=w_{1,c}\right)-{\rm Res}\left(Q(w,c)| w=\bar w_{1,c}\right).
\end{aligned}
\end{equation}
Note that $w_{1,c}$ and $\bar w_{1,c}$ are poles of $c$ of order 1, hence, we have
\begin{equation}
    \begin{aligned}
        &\lim_{c\downarrow c(1)}\left({\rm Res}\left(Q(w,c)| w=w_{1,c}\right)+{\rm Res}\left(Q(w,c)| w=\bar w_{1,c}\right)\right)\\
        =&\lim_{c\downarrow c(1)}\left(\frac{[(w-w_{1,c})g(w,c)]'|_{w=w_{1,c}}}{f'(w,c)|_{w=w_{1,c}}}+\frac{[(w-w_{1,c})g(w,c)]|_{w=\bar w_{1,c}}}{f'(w,c)|_{w=\bar w_{1,c}}}\right)\\
        =&{\rm Res}\left(Q(w,c(1))| w=w(1)\right)+{\rm Res}\left(Q(w,c(1))| w=\bar w(1)\right)=\frac{48}{7}.
    \end{aligned}
\end{equation}

\begin{figure}[h!]
    \centering
    \includegraphics[width=0.6\textwidth]{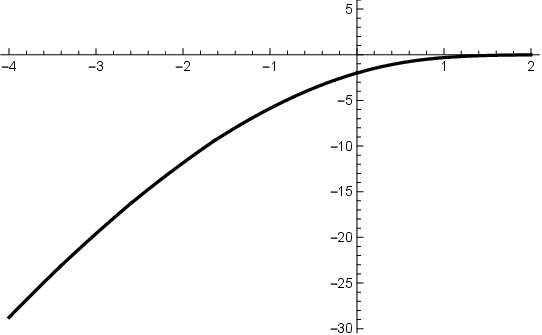}
    \caption{The black line is the graph $(c,\tilde \kappa'(c))$, where $\tilde\kappa'(c)$ is defined in \eqref{kprime}. Comparing to the graph in Figure~\ref{discontinuity}, we notice that we really glue the analytical part together.}
    \label{ana}
\end{figure}

Hence, we can then extend $\kappa'(c)$ to $(c(2),c(1))$ as the following function:
\begin{equation}
    \tilde\kappa'(c)=\begin{cases}
        \kappa'(c),&\textrm{for }c>0,\\
        \lim_{c\downarrow 0}\kappa'(c)-6,&\textrm{for }c=0,\\
        \kappa'(c)-6,&\textrm{for }c\in(c(1),0),\\
        \lim_{c\downarrow c(1)}\kappa'(c)-6-48/7,&\textrm{for }c=c(1),\\
        \kappa'(c)-6-48/7,&\textrm{for }c\in (c(2),c(1)).
    \end{cases}
\end{equation}

Using the similar method, we can also extend the above function to the interval $(c(3),c(2))$ and eventually to the whole real line. It turns out that for $n\geq1$, we have
\begin{equation}
\begin{aligned}
    &\lim_{c\downarrow c(n)}\left({\rm Res}\left(Q(w,c)| w=w_{1,c}\right)+{\rm Res}\left(Q(w,c)| w=\bar w_{1,c}\right)\right)\\
    =&{\rm Res}\left(Q(w,c(n))| w=w(n)\right)+{\rm Res}\left(Q(w,c)| w=\bar w(n)\right)=\frac{48}{7},\quad\forall n\geq1.
\end{aligned}
\end{equation}Hence, we obtain the following analytical continuation of $\kappa'(c)$ (see also Figure~\ref{ana})
\begin{equation}\label{kprime}
    \tilde\kappa'(c)=\begin{cases}
        \kappa'(c)-6\cdot \Id_{c<0}-\frac{48}{7}\sum_{n=1}^\infty\Id_{c<c(n)},&\textrm{for }c\not\in\mathcal J,\\
\lim_{\epsilon\downarrow 0}\tilde\kappa'(c+\epsilon),&\textrm{for }c\in\mathcal J.
    \end{cases}
\end{equation}

As a consequence, we then obtain the full solution of the exponent:
\begin{lem}
The analytic continuation of $\kappa|_{(0,\infty)}$ is given by
    \begin{equation}\label{tildekappa}
        \tilde\kappa(c)=\begin{cases}
            \kappa(c),&\textrm{for }c\geq 0,\\
            \kappa(0)-\int_{c}^0dx\kappa'(x)-6c-\frac{48}{7}\sum_{n\geq 1}(c-c(n))\Id_{c<c(n)},&\textrm{for }c<0,
        \end{cases}
    \end{equation}
where $\kappa'(c)$ is defined in \eqref{c11} and $c(n)$ is defined as in Lemma~\ref{pb3}.
\end{lem}

\appendix
\section{Preliminary Upper Bounds}\label{sectAppB}
In this section, we deduce some preliminary upper bounds for later use. To this end, we first deduce some identities regarding to Airy function. Recall the definitions
\begin{equation}
\begin{aligned}
   &B_{0,c}(x,y)=\Ai(x+y+2c),\\
   &\tilde B_{0,c}(x,y)=\Ai(y-x+2c),\\
   &\hat B_{0,c}(x,y)=\Ai(x-y+2c),
\end{aligned}
\end{equation}
the heat kernel $e^{L\Delta}(x,y)=\frac{1}{\sqrt{4\pi L}}e^{-\frac{(x-y)^2}{4L}}$, a variant of it, that is, $e^{L\tilde\Delta}(x,y)=e^{L\Delta}(-x,y)$. For $L>0$, $e^{-L\Delta}B_{0,c}$ is still well-defined with
\begin{equation}\label{e-lap}
\begin{aligned}
e^{-L\Delta}B_{0,c}(x,y)&=e^{-2L^3/3-L(x+y+2c)}\Ai(L^2+x+y+2c) .
\end{aligned}
\end{equation}
We have the following identities.
\begin{lem}\label{identity}
Let $n\in\mathbb Z_{\geq 1},L>0$ and $c\in\R$, then
        \begin{align}
         \hat B_{0,c}^n(x,y)&=n^{-1/3}\Ai(n^{-1/3}(x-y+2nc))\label{06a2}\\
            \tilde B_{0,c}^n(x,y)&=n^{-1/3}\Ai(n^{-1/3}(y-x+2nc)).\label{06a1}\\
            \hat B_{0,c}^{n-1}B_{0,c}(x,y)&=B_{0,c}\tilde B^{n-1}_{0,c}(x,y)=n^{-1/3}\Ai(n^{-1/3}(x+y+2nc))\label{09a1}\\
         e^{-L\Delta}B_{0,c}\tilde B_{0,c}^{n-1}e^{L\tilde\Delta}&=\hat B_{0,c}^n\label{a103},\\
         e^{-L\Delta}B_{0,c}\tilde B_{0,c}&=\hat B_{0,c}e^{-L\Delta}B_{0,c},\label{4a1},\\
         e^{L\tilde\Delta}e^{-L\Delta}B_{0,c}&=\tilde B_{0,c}\label{a112},\\
          e^{L\Delta}B_{0,c}(x,y)&=e^{2L^3/3+L(x+y+2c)}\Ai(L^2+x+y+2c),\label{a105},\\
          e^{L\Delta}\hat B_{0,c}(x,y)&=B_{0,c}e^{L\tilde\Delta}(x,y)=e^{\tfrac{2L^3}{3}+L(x-y+2 c)} \Ai\left(L^2+x-y+2 c\right)\label{a102},\\
         \tilde B_{0,c}^ne^{L\Delta}(x,y)&=e^{L\tilde\Delta} B_{0,c}^n(x,y)=\tfrac{e^{2Lc+\frac{2L^3}{3n^2}}}{n^{1/3}}e^{\frac{L(y-
     x)}{n}}\Ai\left(\tfrac{L^2}{n^{4/3}}+\tfrac{y-x+2nc}{n^{1/3}}\right),\label{a108}\\
         e^{L\tilde\Delta}\hat B_{0,c}^n(x,y)&=\tilde B_{0,c}^ne^{L\tilde\Delta}(x,y)=\tfrac{e^{2Lc+\frac{2 L^3}{3 n^2}}}{n^{1/3}}e^{-\frac{L(x+y)}{n}}\Ai\left(\tfrac{L^2}{n^{4/3}}-\tfrac{x+y+2nc}{n^{1/3}}\right),\label{a101}\\
         \hat B_{0,c}^{n-1}e^{-L\Delta}B_0(x,y)&=\tfrac{e^{-2Lc-\frac{2L^3}{3n^2}}}{n^{1/3}}e^{-\frac{L(x+y)}{n}}\Ai\left(\tfrac{L^2}{n^{4/3}}+\tfrac{x+y+2nc}{n^{1/3}}\right)\label{a107},\\
      B_{0,c}\tilde B^{n-1}_{0,c}e^{L\tilde\Delta}(x,y)&=\tfrac{e^{2Lc+\frac{2L^3}{3n^2}}}{n^{1/3}}e^{\frac{L (x-y)}{n}}\Ai\left(\tfrac{L^2}{n^{4/3}}+\tfrac{x-y+2nc}{n^{1/3}} \right).\label{a110}
     \end{align}
\end{lem}
\begin{proof}
Let us show in detail how to derive~\eqref{06a2} and ~\eqref{a101} only, since the others follows using similar computations. Recall that
    \begin{equation}
        \Ai(x)=\frac{1}{2\pi\I}\int_{\I\R+\e}dw e^{\frac{w^3}{3}-wx},\quad \e>0.
    \end{equation}
Then we have
\begin{equation}
\begin{aligned}
\hat B_{0,c}^n(x,y) =&\int_{\R^{n-1}} dz_1\cdots dz_{n-1} \int_{\I\R+\e_1}dw_1\cdots\int_{\I\R+\e_n} dw_n \frac{1}{(2\pi\I)^n}\bigg(\prod_{k=1}^n e^{-w_k^3/3-2c w_k}\bigg)\\
&\times  e^{-w_1(x-z_1)}\bigg(\prod_{\ell=2}^{n-1} e^{-w_\ell (z_{\ell-1}-z_\ell)}\bigg) e^{-z_{n-1}(w_n-w_{n-1})}.
\end{aligned}
\end{equation}
We can take the integral over $z_{n-1}$ separately for $z_{n-1}\in\R_+$ and $z_{n-1}\in\R_-$. In the first case we need to assume $\e_{n}>\e_{n-1}$, while in the second case $\e_n<\e_{n-1}$. Then the integral over $z_{n-1}\in\R_+$ gives a factor $\frac1{w_n-w_{n-1}}$ while the integral over $z_{n-1}\in\R_-$ gives a factor $-\frac1{w_n-w_{n-1}}$. For fixed $w_{n-1}$, putting the two integrals together we get that the integral over $w_n$ is a simple anticlockwise oriented path enclosing $w_{n-1}$, which has a simple pole at $w_{n-1}$. Doing the same for the integrals over $z_{n-2}$ until $z_1$ we obtain
\begin{equation}
\hat B_{0,c}^n(x,y) = \frac{1}{2\pi\I}\int_{\I\R+\e_1} e^{n w_1^3/3- w_1(x-y+2 n c)} =n^{-1/3}\Ai(n^{-1/3}(x-y+2nc)).
\end{equation}
    Calculating the Gaussian integral, we then obtain
    \begin{equation}\label{equal}
        \begin{aligned}
            e^{L\tilde \Delta}\hat B_{0,c}^n(x,y)&=\frac{1}{2\pi \I}\int_{\I\R+\epsilon}dw\int_\R dz\frac{1}{\sqrt{4\pi L}} e^{-\frac{(x+z)^2}{4L}}e^{\frac{w^3}{3}-wn^{-1/3}(z-y+2nc)}\\
            &=\frac{1}{2\pi \I}\int_{\I\R+\epsilon}dwe^{\frac{w^3}{3}+wn^{-1/3}(x+y)+n^{-2/3}Lw^2-2n^{2/3}cw},\quad\forall\epsilon>0.
        \end{aligned}
    \end{equation}
Clearly we have $\tilde B_{0,c}^n(x,y)=\hat B_{0,c}^n(y,x)=n^{-1/3}\Ai(n^{-1/3}(y-x+2nc))$. Using this we get $\tilde B_{0,c}^ne^{L\tilde\Delta}(x,y)=\textrm{r.h.s.\ of }\eqref{equal}$. By the change of variable $w\mapsto w-n^{-2/3}L$ we get the claimed expression \eqref{a101}.
\end{proof}
For $r>0$, let $U_r: L^{2}(\R)\to L^2(\R)$ be as $U_rf(x)=e^{rx}f(x)$ and
\begin{equation}
\beta= \max\{2e^{r^3/3-2rc},e^{(r-1/7)^3/3-2(r-1/7)c}\}.
\end{equation}
\begin{lem}\label{bound}
Let $n\in\mathbb Z_{\geq 1}$ and $1\leq r^2\leq 2c$ with $r>0$, then
        \begin{align}
             \max\{\left\|U_r^{-1} P_0 B_{0, c} U_r\right\|_{\rm HS},\left\|U_r P_0 B_{0, c} P_0U_r^{-1} \right\|_{\rm HS}\}&\leq \beta\label{1ab3},\\
             \|U_rP_0\hat B_{0,c}^nB_{0,c}P_0U_r^{-1}\|_{\rm HS}=\|U_r^{-1}P_0B_{0,c}\tilde B_{0,c}^nP_0U_r\|_{\rm HS}&\leq\beta^{n+1}\label{1cb3},\\
             \max\{\|U_r^{-1}\tilde B_{0,c}^nU_r\|_{\rm op},\|U_r^{-1}\bar P_0\tilde B_{0,c}^nP_0U_r\|_{\rm HS},\|U_r^{-1} P_0\tilde B_{0,c}^n\bar P_0U_r\|_{\rm HS}\}&\leq\beta^n,\label{2ab3}\\
             \max\{\|U_r\hat B_{0,c}^nU_r^{-1}\|_{\rm op},\|U_r P_0\hat B_{0,c}^n\bar P_0U_r^{-1}\|_{\rm HS},\|U_r\bar P_0\hat B_{0,c}^nP_0U_r^{-1}\|_{\rm HS}\}&\leq\beta^n,\label{3ab3}\\
            \max\{\|U_r^{-1}P_0e^{L\tilde\Delta}P_0U_r^{-1}\|_{\rm HS},\|U_r^{-1}P_0e^{L\Delta}P_0U_r^{-1}\|_{\rm HS}\}&\leq\tfrac{1}{\sqrt{L}}\label{4b3}.
        \end{align}
\end{lem}
\begin{proof}
Let $f\in L^2(\R)$. For~\eqref{1ab3}, by symmetry of $B_{0,c}$, we have
\begin{equation}
    \left\|U_r P_0 B_{0, c} P_0U_r^{-1} \right\|_{\rm HS}=\left\|U_r^{-1} P_0 B_{0, c} P_0U_r \right\|_{\rm HS}\leq \left\|U_r^{-1} P_0 B_{0, c} U_r\right\|_{\rm HS}.
\end{equation}
It is enough to show $\left\|U_r^{-1} P_0 B_{0, c} U_r\right\|_{\rm HS}\leq\beta$. By definition of Hilbert-Schmidt norm, we have
    \begin{equation}\label{na26}
        \begin{aligned}
        &\|U_r^{-1} P_0 B_{0, c} U_r\|_{\rm HS}^2=\int_0^\infty dx\int_{-\infty}^\infty dy e^{-2r(x-y)}\Ai(x+y+2c)^2=\frac{e^{\frac{2 r^3}{3}-4 c r}}{8 \sqrt{2 \pi } r^{3/2}},
    \end{aligned}
    \end{equation}
    where in the last step we use (Lemma 2.6 of~\cite{Ok02})
    \begin{equation}\label{ol26e}
        \int_\R dye^{Ly}\Ai(y)^2=\frac{e^{L^3/12}}{\sqrt{4L\pi }},\quad\forall L>0.
    \end{equation}
    \eqref{1ab3} follows then from $r\geq 1.$ The first equality of~\eqref{1cb3} follows from~\eqref{09a1} and symmetry. Similar as~\eqref{na26}, we have
    \begin{equation}
        |U_r\hat B_{0,c}^nB_{0,c}P_0U_r^{-1}\|_{\rm HS}^2=\frac{e^{(n+1)(2r^3/3-4cr)}}{8\sqrt{2\pi (n+1)}r^{3/2}}\leq \beta^{2(n+1)}.
    \end{equation}
    \eqref{1cb3} follows from $\|U_rP_0\hat B_{0,c}^nB_{0,c}P_0U_r^{-1}\|_{\rm HS}\leq |U_r\hat B_{0,c}^nB_{0,c}P_0U_r^{-1}\|_{\rm HS}.$ For~\eqref{2ab3}, we first show $\|U_r^{-1}\tilde B_{0,c}^nU_r\|_{\rm op}\leq\beta$. Define $h(x)=e^{-rx}\Ai(-x+2c)$, then
    \begin{equation}
        \begin{aligned}
            &\|U_r^{-1}\tilde B_{0,c}U_r f\|_{L^2(\R)}=\|h*f\|_{L^2(\R)}\leq\|f\|_{L^2(\R)}\|h\|_{L^1(\R)},
        \end{aligned}
    \end{equation}
      where we apply Young's inequality for convolution in the last step. It is enough to bound $\|h\|_{L^1(\R)}.$ Note that
      \begin{equation}\label{na30}
        \begin{aligned}
            \|h\|_{L^1(\R)}=e^{-2rc}\int_{\R}dx |e^{rx}\Ai(x)|\leq e^{-2rc}\left(\int_{\R} dx e^{rx}\Ai(x)+2\int_{-\infty}^0 dxe^{rx}|\Ai(x)|\right).\\
        \end{aligned}
    \end{equation}
    By $\max|\Ai(x)|\leq 3/5$, we have
    \begin{equation}
        2\int_{-\infty}^0 dxe^{rx}|\Ai(x)|\leq\frac{6}{5r}\leq e^{r^3/3},\quad\forall r\geq 1.
    \end{equation}
    Together with the identity (see (9.10.13) in \cite{NIST:DLMF})
\begin{equation}
   \int_\R dx e^{rx}\Ai(x)=e^{r^3/3},\quad\forall r>0.
\end{equation}
we have $\|h\|_{L^1(\R)}\leq 2e^{r^3/3-2rc}\leq\beta.$ For the rest two quantities in~\eqref{2ab3}, it is enough to show $\|U_r^{-1}\bar P_0\tilde B_{0,c}^nP_0U_r\|_{\rm HS}\leq\beta$. Using~\eqref{06a1} and the definition of Hilbert-Schmidt norm, we have
        \begin{equation}\label{wehaverhd}
        \begin{aligned}
            &\|U_r^{-1}\bar P_0\tilde B_{0,c}^nP_0U_r\|_{\rm HS}^2=\int_{-\infty}^0dxn^{-2/3}\int_0^\infty dy e^{2r(y-x)}\Ai\left(\tfrac{y-x}{n^{1/3}}+2n^{2/3}c\right)^2\\
            =&\int_0^\infty du e^{2rn^{1/3}u}\Ai(u+2n^{2/3}c)^2u,
        \end{aligned}
        \end{equation}
        where we made the change of variable $u=(y-x)/n^{1/3}$ and $v=(y+x)/n^{1/3}$ and integrated over $v$. Using the identity ((2.4) of~\cite{Rei95})
        \begin{equation}\label{airysquare}
        \Ai(y)^2=\frac{1}{4 \pi^{3 / 2} \I} \int_{\I\R+\e} dw w^{-1 / 2}e^{\frac{1}{12} w^3-wy},\quad \e>0,y\in\R,
    \end{equation}
    the last integral in~\eqref{wehaverhd} is equal to
        \begin{equation}
            \frac{1}{4\pi^{3/2}\I}\int_{\I\R+\e} dw \frac{e^{\frac{w^3}{12}-2wn^{2/3}c}}{\sqrt{w}(w-2n^{1/3}r)^2}
        \end{equation}
    for any $\e>2n^{1/3}r$. Choosing $\e=2n^{1/3}r+1/(r^{1/2} n^{1/6})$ we get that the absolute value of the last integral is, for $1\leq r^2\leq 2c$, bounded by
    \begin{equation}
        \frac{e^{2n r^3/3-4n c}}{\sqrt{2}\pi r^{1/2}}\leq \beta^{2n}.
    \end{equation}
    Applying~\eqref{06a2} and same method as the one for~\eqref{2ab3}, we can obtain~\eqref{3ab3}. For~\eqref{4b3}, by definition, we have
    \begin{equation}
    \begin{aligned}
        &            \|U_r^{-1}P_0e^{L\tilde\Delta}P_0U_r^{-1}\|_{\rm HS}^2=\frac{1}{4 \pi L}\int_{\R_+^2} dxdy e^{-2r x-2ry}e^{-\frac{(x+y)^2}{4L}}\leq\frac{1}{16\pi L r ^2}\leq\frac{1}{L},
    \end{aligned}
    \end{equation}
    since $r\geq 1.$ Similarly, we can also show the case for $e^{L\Delta}.$
\end{proof}

Next, we deduce upper bounds for operators involving heat kernel.
\begin{lem}\label{cobound}
  Let $n\in\mathbb Z_{\geq 1}, L,r\geq 1$ and $r^2\leq 2c$, then
  \begin{align}
        \|U_rP_0\hat B_{0,c}^{n-1}e^{-L\Delta}B_0P_0U_{r}\|_{\rm HS}&\leq \beta^{n}e^{-\frac{4L^3}{3n^2}}e^{-2Lc},\label{1b4},\\
        \|U_r^{-1}P_0e^{L\Delta}B_{0,c}P_0U_r^{-1}\|_{\rm HS}&\leq \beta e^{Lr^2}\label{2b4},\\
       \|U_r^{-1} P_0 B_{0, c} e^{L \tilde{\Delta}}P_0 U_r^{-1}\|_{\rm HS}&\leq\beta e^{Lr^2}\label{3b4},\\
       \|U_r^{-1}P_0e^{L\tilde\Delta}\hat B_{0,c}^nP_0U_r^{-1}\|_{\rm HS}&=\|U_r^{-1}P_0\tilde B_{0,c}^ne^{L\tilde\Delta}P_0U_r^{-1}\|_{\rm HS}\leq\beta^ne^{Lr^2}\label{4bb4},\\
       \|U_r^{-1}P_0\tilde B_{0,c}^ne^{L\Delta}P_0U_r^{-1}\|_{\rm HS}^2&\leq \beta^ne^{Lr^2}\label{5b4},\\
       \|U_r^{-1}P_0B_{0,c}\tilde B_{0,c}^{n-1}e^{L\tilde\Delta}P_0U_r^{-1}\|_{\rm HS}&\leq \beta^{n}e^{Lr^2}\label{6b4}.
  \end{align}
\end{lem}
\begin{proof}
For~\eqref{1b4}: applying~\eqref{a107} and definition of Hilbert-Schmidt norm, we have
\begin{equation}\label{eqB16}
    \begin{aligned}
            &\|U_rP_0 \hat B_{0,c}^{n-1}e^{-L\Delta}B_0P_0U_{r}\|_{\rm HS}^2\\
            =&\frac{e^{-4Lc}e^{-\frac{4L^3}{3n^2}}}{n^{2/3}}\int_{\R_+^2} dx dy e^{2rx+2ry-\frac{2L(x+y)}{n}}\Ai\left(\tfrac{L^2}{n^{4/3}}+\tfrac{x+y}{n^{1/3}}+2n^{2/3}c\right)^2\\
           =&e^{-4Lc-\frac{4L^3}{3n^2}}\int_0^\infty du u e^{-u\big(\frac{2L}{n^{2/3}}-2n^{1/3}r\big)}\Ai\left(\tfrac{L^2}{n^{4/3}}+u+2n^{2/3}c\right)^2\\
            \stackrel{\eqref{airysquare}}{=}&\frac{e^{-4Lc}e^{-\frac{4L^3}{3n^2}}}{4\pi^{3/2}\I}\int_{\I\R+\alpha} dw \frac{e^{\frac{w^3}{12}-w\big(\frac{L^2}{n^{4/3}}+2n^{2/3}c\big)}}{\left(w-\frac{2(-L+nr)}{n^{2/3}}\right)^2 w^{-1/2}}
        \end{aligned}
\end{equation}
 with $\alpha=\frac{2(-L+nr)}{n^{2/3}}+\e$ for arbitrary $\e>0$. With the choice $\alpha=\frac{2(L+nr)}{n^{2/3}}$ we get
    \begin{equation}
|\eqref{eqB16}|\leq \frac{n^{5/3}}{32\sqrt{2}L^2\pi \sqrt{L+nr}}e^{2L r^2-8cL}e^{-\frac{8L^3}{3n^2}}(e^{2r^3/3-4c r})^{n},
    \end{equation}
    which, for $1\leq r^2\leq 2c$, will be dominated by the claimed bound (we used that $(e^{2r^3/3-4c r})^{n}\leq (\beta/2)^{2n}$ and $n^{5/3}\leq 4^{n}$ for $n\geq 1$). Applying~\eqref{a105} and similar method as above, we will get~\eqref{2b4}. For~\eqref{3b4}, applying~\eqref{a102}, definition of Hilbert-Schmidt norm,~\eqref{ol26e} and $L,r\geq 1$,  we have
    \begin{equation}\label{a27}
        \begin{aligned}
          &\|U_r^{-1} P_0 B_{0, c} e^{L \tilde{\Delta}} P_0U_r^{-1}\|_{\rm HS}^2 \leq \|U_r^{-1} P_0 B_{0, c} e^{L \tilde{\Delta}} U_r^{-1}\|_{\rm HS}^2 =\frac{e^{\frac{2r^3}{3}-4rc+2Lr^2}}{8\sqrt{2\pi}r\sqrt{L+r}}\leq \beta^2e^{2Lr^2}.
        \end{aligned}
    \end{equation}
    For~\eqref{4bb4}, the first equality follows from~\eqref{a101}. For the inequality, since $r\geq 1$, we have
        \begin{equation}\label{b35}
            \|U_r^{-1}P_0\tilde B_{0,c}^ne^{L\tilde\Delta}P_0U_r^{-1}\|_{\rm HS}\leq  \|U_r^{-1}P_0\tilde B_{0,c}^ne^{L\tilde\Delta}P_0U_{r-1/7}^{-1}\|_{\rm HS}
                \leq \|U_r^{-1}P_0\tilde B_{0,c}^ne^{L\tilde\Delta}U_{r-1/7}^{-1}\|_{\rm HS}.
        \end{equation}
    Similarly as~\eqref{a27}, applying~\eqref{ol26e}, we have
\begin{equation}
    \|U_r^{-1} P_0 \tilde{B}_{0, c}^n e^{L \tilde{\Delta}} U_s^{-1}\|_{\rm HS}^2=\frac{e^{(2 s^3 / 3-4 c s) n+2 L s^2}}{4 \sqrt{2 \pi}(r-s) \sqrt{L+n s}},\quad\forall r>s>0.
\end{equation}
Applying this with $s=r-1/7$, $7\leq{4\sqrt{2\pi}}$ and plugging back to~\eqref{b35}, we obtain~\eqref{4bb4}. For~\eqref{5b4}, applying~\eqref{a108}, $r\geq 1$ and definition of Hilbert-Schmidt norm, we have
   \begin{equation}
       \begin{aligned}
           &\|U_r^{-1}P_0\tilde B_{0,c}^ne^{L\Delta}P_0U_r^{-1}\|_{\rm HS}^2\\
           &=  n^{-2/3}e^{4Lc+\frac{4L^3}{3n^2}}\int_{\R_+^2}dxdy e^{\frac{2L(y-x)}{n}-2r(x+y)}\Ai\left(\tfrac{L^2}{n^{4/3}}+\tfrac{y-x}{n^{1/3}}+2n^{2/3}c\right)^2\\
           &=  \frac{e^{4Lc+\frac{4L^3}{3n^2}}}{4n^{1/3}r}\left(\int_{-\infty}^0du e^{\frac{2Lu}{n^{2/3}}+2n^{1/3}ru}\Ai(L^2n^{-4/3}+u+2n^{2/3}c)\right.\\
           & \quad\qquad\qquad+\left.\int_{0}^{\infty}du e^{\frac{2Lu}{n^{2/3}}-2n^{1/3}ru}\Ai(L^2n^{-4/3}+u+2n^{2/3}c)\right)\\
           &\leq \frac{e^{4Lc+\frac{4L^3}{3n^2}}}{4n^{1/3}r}\int_\R du e^{\frac{2Lu}{n^{2/3}}+2n^{1/3}ru}\Ai(L^2n^{-4/3}+u+2n^{2/3}c)\\&
           \stackrel{\eqref{ol26e}}{=}\frac{e^{\frac{2nr^3}{3}-4cnr+2Lr^2}}{8\sqrt{2\pi}r\sqrt{L+nr}}\leq \beta^{2n}e^{2Lr^2}
       \end{aligned}
   \end{equation}
For~\eqref{6b4}, applying~\eqref{a110}, definition of Hilbert-Schmidt norm,~\eqref{ol26e} and $r\geq 1$, we have
   6. By~\ref{a110} in Lemma~\ref{identity} and
   \begin{equation}
       \|U_r^{-1}P_0B_{0,c}\tilde B_{0,c}^ne^{L\tilde\Delta}P_0U_r^{-1}\|_{\rm HS}^2\leq \|U_r^{-1}P_0B_{0,c}\tilde B_{0,c}^ne^{L\tilde\Delta}U_r^{-1}\|_{\rm HS}^2\leq\beta^{2(n+1)}e^{2Lr^2}.
   \end{equation}
\end{proof}

\section{Upper bounds}\label{upperbounds}

In this section, we provide some useful upper bounds. Let's first recall that
\begin{equation}
\begin{aligned}
        &K_a=\tilde B_{0,c},\quad K_b=P_0\tilde B_{0,c},\quad K_c=e^{L\tilde\Delta}P_0e^{-L\Delta}B_{0,c},\quad K_d=P_0e^{L\tilde\Delta}P_0e^{-L\Delta}B_{0,c},\\
        &K_e=P_0e^{L\Delta}P_0e^{-L\Delta}B_{0,c},\quad K_u=P_0B_{0,c},\quad K_v=e^{L\Delta}P_0e^{-L\Delta}B_{0,c}.
    \end{aligned}
\end{equation}
From here, we will use the convention
    \begin{equation}
        \prod_{i=1}^0 A_i=1,\quad\text{for any operators $A_1,\ldots,A_n.$}
    \end{equation}
Also, we will always assume $L,r\geq 1$ and $2c\geq r^2$ so that we can apply the bounds of Section~\ref{sectAppB}.

\begin{lem}\label{la121}
    Let $n\in\mathbb Z_{\geq 0}$, $\sigma_n\in\{a,b\}^n$, then $\|\hat \Phi_{\sigma_n}\|_{\rm op}\leq \beta^{n+1}$, where
    \begin{equation}
        \hat\Phi_{\sigma_n}=U_r^{-1}P_0B_{0,c}\bigg[\prod_{i=1}^nK_{\sigma_n(i)}\bigg]P_0U_r.
    \end{equation}
\end{lem}
\begin{proof}
We show this via induction on $n$. The case $n=0$ follows from~\eqref{1ab3}. Consider now the induction step $n-1\mapsto n.$\\
\textbf{Case I: $\sigma_n\equiv a$.} Then $\hat\Phi_{\sigma_n}=U_r^{-1}P_0B_{0,c}\tilde B_{0,c}^nP_0U_r$, the result follows from~\eqref{1cb3}.\\
\textbf{Case II: $\sigma_n\in\{a,b\}^n$ with $b\in\sigma_n$.} Define $\ell_b=\max\{i|\sigma_n(i)=b\}\geq 1$, then
\begin{equation}
    \hat \Phi_{\sigma_n}=\underbrace{U_r^{-1}P_0B_{0,c}\bigg[\prod_{i=1}^{\ell_b-1}K_{\sigma_n(i)}\bigg]P_0U_r}_{=:\hat\varphi_1}\underbrace{U_r^{-1} P_0\tilde B_{0,c}^{n-\ell_b+1}P_0U_r}_{=:\hat\varphi_2}.
\end{equation}
Applying $\|\hat\varphi_1\|_{\rm op}\leq\beta^{\ell_b}$ (induction assumption) and $\|\hat\varphi_2\|_{\rm op}\leq\beta^{n-\ell_b+1}$ (by~\eqref{2ab3}), we have $\|\hat \Phi_{\sigma_n}\|_{\rm op}\leq \|\hat\varphi_1\|_{\rm op}\|\hat\varphi_2\|_{\rm op}\leq\beta^{n+1}.$
\end{proof}

\begin{lem}\label{l25}
    Let $n\in\mathbb Z_{\geq 0}$, $\sigma_n\in \{a,b\}^n$, then $\|\Phi_{\sigma_n}\|_{\rm HS}\leq\beta^{n+1}$, where
    \begin{equation}
        \Phi_{\sigma_n}=U_r^{-1}\bar P_0\tilde B_{0,c}\bigg[\prod_{i=1}^nK_{\sigma_n(i)}\bigg]P_0U_r.
    \end{equation}
\end{lem}
\begin{proof}
For $n=0$, $\Phi_{\sigma_0}=U_r^{-1}\bar P_0\tilde B_{0,c}P_0U_r$, the result follows from \eqref{2ab3}. For $n=1$, there are only two possible $\Phi_{\sigma_1}$, namely $\sigma_1(1)\in\{a,b\}$. For the first case, applying \eqref{2ab3}, we have $ \|\Phi_{\sigma_1}\|_{HS}=\|U_r^{-1}\bar P_0\tilde B_{0,c}^2U_r\|_{HS}\leq\beta^2.$ Suppose $\sigma_1(1)=b$. Applying \eqref{2ab3}, we get
\begin{equation}
    \|\Phi_{\sigma_1}\|_{\rm HS}\leq\|U_r^{-1}\bar P_0\tilde B_{0,c}P_0U_r\|_{\rm HS}\|U_r^{-1}P_0\tilde B_{0,c}P_0U_r\|_{\rm op}\leq\beta^2.
\end{equation}
For $n\geq 2$ we prove it by induction. For the induction step form $n-1$ to $n$ we need to consider the following cases:\\
\textbf{Case I: $\sigma_n(i)=a$ for all $i=1,2,\ldots,n$.} We have $\Phi_{\sigma_n}=U_r^{-1}\bar P_0\tilde B_{0,c}^{n+1}P_0U_r$, the results follows from~\eqref{2ab3}.\\
\textbf{Case II: $b\in\sigma_n$.} In this case, denote by $1\leq \ell_b=\max\{i|\sigma_{n}(i)=b\}\leq n$. Then
    \begin{equation}
        \Phi_{\sigma_n}=\underbrace{U_r^{-1}\bar P_0\tilde B_{0,c}\bigg[\prod_{i=1}^{\ell_b-1}K_{\sigma_n(i)}\bigg]P_0U_r}_{=:\varphi_1}\cdot\underbrace{U_r^{-1} P_0\tilde B_{0,c}^{n-\ell_b+1}P_0U_r}_{=:\varphi_2}.
    \end{equation}
Applying now $\|\varphi_1\|_{\rm HS}\leq\beta^{\ell_b}$ (induction assumption) and $\|\varphi_2\|_{\rm op}\leq\beta^{n-\ell_b+1}$ (by~\eqref{2ab3}), we have $\|\Phi_{\sigma_n}\|_{\rm HS}\leq\|\varphi_1\|_{\rm HS}\|\varphi_2\|_{\rm op}\leq \beta^{n+1}.$
\end{proof}

\begin{lem}\label{wkey}
Let $n\in\mathbb Z_{\geq 0}$, $\sigma_n\in\{a,b,c\}^n$, then $\|\hat \Phi_{\sigma_n}\|_{\rm HS}\leq \beta^{n+1}e^{Lr^2}$ where
\begin{equation}
    \hat \Phi_{\sigma_n}=U_r^{-1} P_0\tilde B_{0,c}\bigg[\prod_{i=1}^{n}K_{\sigma(i)}\bigg]e^{L\tilde\Delta}P_0 U_r^{-1}.
\end{equation}
\end{lem}
\begin{proof}
The case $n=0$ follows from~\eqref{4bb4}. We show this via induction on $n\geq 1$. For $n=1$, we have the following cases. If $\sigma_1(1)=a$, we have $\hat\Phi_{\sigma_1}=U_r^{-1}P_0\tilde B_{0,c}^2e^{L\tilde\Delta}P_0U_r^{-1}$, the result follows from \eqref{4bb4}. If $\sigma_1(1)=b$, then we have
     $\hat\Phi_{\sigma_1}=U_r^{-1}P_0\tilde B_{0,c}P_0\tilde B_{0,c}e^{L\tilde\Delta}P_0U_r^{-1}$. Applying the bounds $\|U_r^{-1}P_0\tilde B_{0,c}P_0U_r\|_{\rm op}\leq\beta$ by \eqref{2ab3}, $\|U_r^{-1}P_0\tilde B_{0,c}e^{L\tilde\Delta}P_0U_r^{-1}\|_{\rm HS}\leq\beta e^{Lr^2}$ (by \eqref{4bb4}) and Theorem~\ref{simon}, we have
    \begin{equation}
       \|\hat\Phi_{\sigma_1}\|_{\rm HS}\leq \|U_r^{-1}P_0\tilde B_{0,c}P_0U_r\|_{\rm op}\|U_r^{-1}P_0\tilde B_{0,c}e^{L\tilde\Delta}P_0U_r^{-1}\|_{\rm HS}\leq\beta^2e^{Lr^2},
    \end{equation}
If $\sigma_1(1)=c$, using $e^{-L\Delta}B_{0,c}e^{L\tilde\Delta}=\hat B_{0,c}$ (by~\eqref{a102}), we have
\begin{equation}
    \hat \Phi_{\sigma_1}=U_r^{-1}P_0\tilde B_{0,c}e^{L\tilde\Delta}P_0e^{-L\Delta}B_{0,c}e^{L\tilde\Delta}P_0U_r^{-1}=U_r^{-1}P_0\tilde B_{0,c}e^{L\tilde\Delta}P_0\hat B_{0,c}P_0U_r^{-1}.
\end{equation}
Applying $\|U_r^{-1}P_0\tilde B_{0,c}e^{L\tilde\Delta}P_0U_r^{-1}\|_{\rm HS}\leq\beta e^{Lr^2}$ (by \eqref{4bb4}) and $\| U_rP_0\hat B_{0,c}P_0U_r^{-1}\|_{\rm op}\leq\beta$ (by \eqref{3ab3}), we have
\begin{equation}
    \|\hat \Phi_{\sigma_1}\|_{\rm HS}\leq \|U_r^{-1}P_0\tilde B_{0,c}e^{L\tilde\Delta}P_0U_r^{-1}\|_{\rm HS} \| U_rP_0\hat B_{0,c}P_0U_r^{-1}\|_{\rm op}\leq\beta^2e^{Lr^2}.
\end{equation}
Now we consider the induction step $n-1\mapsto n:$\\
\textbf{Case I: $\sigma_n\equiv a$}, then $\hat \Phi_{\sigma_n}=U_r^{-1}P_0\tilde B_{0,c}^{n+1}e^{L\tilde\Delta}P_0U_r^{-1}$, the result follows from \eqref{4bb4}.\\
\textbf{Case II: $\sigma_n\in\{a,b\}^n\setminus\{a\}^n$.} We then have
\begin{equation}
    \hat \Phi_{\sigma_n}=\underbrace{U_r^{-1}P_0\tilde B_{0,c}^{f_b}P_0 U_r}_{=:\hat\varphi_1}\underbrace{U_r^{-1}P_0\tilde B_{0,c}\bigg[\prod_{i=f_b+1}^{n}K_{\sigma_n(i)}\bigg]e^{L\tilde\Delta}P_0U_r^{-1}}_{=:\hat\varphi_2},
\end{equation}
where $f_b=\min\{i|\sigma_n(i)=b\}\geq1$. Applying $\|\hat\varphi_1\|_{\rm op}\leq\beta^{f_b}$ (by~\eqref{2ab3}) and $\|\hat\varphi_2\|_{\rm HS}\leq \beta^{n-f_b+1}e^{Lr^2}$ (induction assumption), we have
\begin{equation}
    \|\hat \Phi_{\sigma_n}\|_{\rm HS}\leq\|\hat\varphi_1\|_{\rm op}\|\hat\varphi_2\|_{\rm HS}\leq \beta^{n+1}e^{Lr^2}.
\end{equation}
\textbf{Case III: $\sigma_n\in\{a,b,c\}^n\setminus\{a,b\}^n$.} Then we have $\ell_c=\max\{i|\sigma_n(i)=c\}\geq 1$
and hence
\begin{equation}
    \hat \Phi_{\sigma_n}=\underbrace{U_r^{-1}P_0\tilde B_{0,c}\bigg[\prod_{i=1}^{\ell_c-1}K_{\sigma_n(i)}\bigg]e^{L\tilde\Delta}P_0U_r^{-1}}_{=:\hat\varphi_1}\cdot\underbrace{ U_r P_0e^{-L\Delta}B_{0,c}\bigg[\prod_{i=\ell_c+1}^{n}K_{\sigma_n(i)}\bigg]e^{L\tilde\Delta}P_0U_r^{-1}}_{=:\hat\varphi_2}.
\end{equation}
By induction assumption, $\|\hat\varphi_1\|_{\rm HS}\leq\beta^{\ell_c}e^{Lr^2}$. As for $\hat\varphi_2$, if $\sigma_n(i)=a$ for all $i\in\{\ell_c+1,\ldots,n\}$, then by~\eqref{a103}, we have
\begin{equation}
    \begin{aligned}
        \hat \varphi_2&=U_r P_0e^{-L\Delta}B_{0,c}\tilde B_{0,c}^{n-\ell_c}e^{L\tilde\Delta}P_0U_r^{-1}=U_r P_0\hat B_{0,c}^{n-\ell_c+1}P_0U_r^{-1}.
    \end{aligned}
\end{equation}
Applying $\|\hat\varphi_2\|_{\rm op}\leq\beta^{n-\ell_c+1}$ (by~\eqref{3ab3}), we have
    $\|\hat \Phi_{\sigma_n}\|_{\rm HS}\leq\|\hat\varphi_1\|_{\rm HS}\|\hat\varphi_2\|_{\rm op}\leq \beta^{n+1}e^{Lr^2}$. If there exists $j\geq\ell_c+1$ such that $\sigma_n(j)=b$, we set $f_b=\min\{j\geq\ell_c+1|\sigma_n(j)=b\}$. Then we have
\begin{equation}
    \hat \varphi_2=\underbrace{U_rP_0\hat B_{0,c}^{f_b-\ell_c-1}e^{-L\Delta}B_{0,c}P_0U_r}_{=:\hat\varphi_2^1}\cdot \underbrace{U_r^{-1}P_0\tilde B_{0,c}\bigg[\prod_{i=f_b+1}^{n}K_{\sigma_n(i)}\bigg]e^{L\tilde\Delta}P_0U_r^{-1}}_{=:\hat\varphi_2^2},
\end{equation}
where we use the definition of $f_b,\ell_c$ and the identity $e^{-L\Delta}B_{0,c}\tilde B_{0,c}^{f_b-\ell_c-1}=\hat B_{0,c}^{f_b-\ell_c-1}e^{-L\Delta}B_{0,c}$ (by~\eqref{4a1}). Applying now $\|\hat\varphi_2^2\|_{\rm HS}\leq\beta^{n-f_b+1}e^{Lr^2}$ (induction assumption), $\|\hat\varphi_2^1\|_{\rm HS}\leq \beta^{f_b-\ell_c}e^{-\frac{4L^3}{3(f_b-\ell_c)^2}}e^{-2Lc}$ (by~\eqref{1b4}) and $r^2\leq 2c$, we have
\begin{equation}
    \begin{aligned}
        &\|\hat \Phi_{\sigma_n}\|_{\rm HS}\leq\|\hat\varphi_1\|_{\rm HS}\|\hat\varphi_2^1\|_{\rm HS}\|\hat\varphi_2^2\|_{\rm HS}\leq\beta^{n+1}e^{-\frac{4L^3}{3n^2}}.
    \end{aligned}
\end{equation}
\end{proof}

\begin{lem}\label{uwkey1}
For $\sigma_n\in\{a,b,c,d,e\}^n$ with $n\in\Z_{\geq 0}$, it holds $\|\hat \Phi_{\sigma_n}\|_{\rm HS}\leq\beta^{n+1}e^{-\frac{4L^3}{3(n+1)^2}}e^{-2Lc}$, where
\begin{equation}
\hat \Phi_{\sigma_n}=U_rP_0e^{-L\Delta}B_{0,c}\bigg[\prod_{i=1}^nK_{\sigma_n(i)}\bigg]P_0 U_r.
\end{equation}
\end{lem}
\begin{proof}

We show this via induction on $n$. The case $n=0$ follows from~\eqref{1b4}. For $n=1$, we have the following cases:
    \begin{enumerate}
    \item $\sigma_1(1)=a$: then $\hat \Phi_{\sigma_n}=U_rP_0\hat B_{0,c}e^{-L\Delta}B_{0,c}P_0U_r$ by~\eqref{4a1}. The results follows from~\eqref{1b4}.
        \item $\sigma_1(1)=b$: then $\hat \Phi_{\sigma_n}=U_rP_0e^{-L\Delta}B_{0,c}P_0\tilde B_{0,c}P_0U_r$. Applying \eqref{1b4} and \eqref{2ab3}, we have
 \begin{equation}
     \|\hat \Phi_{\sigma_n}\|_{\rm HS}\leq\|U_r P_0e^{-L\Delta}B_{0,c} P_0U_r\|_{\rm HS}\|U_r^{-1}P_0\tilde B_{0,c}P_0U_r\|_{\rm op}\leq\beta^{2}e^{-\frac{4L^3}{3}}e^{-2Lc},
 \end{equation}
  \item  $\sigma_1(1)=c$: then $ \hat \Phi_{\sigma_n}=U_rP_0\hat B_{0,c}P_0e^{-L\Delta}B_{0,c}P_0U_r$  (by~\eqref{a102}). Applying $\|U_rP_0\hat B_{0,c}P_0U_r^{-1}\|_{\rm op}\leq\beta$ (by~\eqref{3ab3}) and \eqref{1b4}, we have
 \begin{equation}
     \|\hat \Phi_{\sigma_n}\|_{\rm HS}\leq\|U_rP_0\hat B_{0,c}P_0U_r^{-1}\|_{\rm op}\|U_r P_0e^{-L\Delta}B_{0,c} P_0U_r\|_{\rm HS}\leq\beta^{2}e^{-\frac{4L^3}{3}}e^{-2Lc}.
 \end{equation}

  \item  $\sigma_1(1)=d$: then
  \begin{equation}
      \hat \Phi_{\sigma_n}=U_r P_0e^{-L\Delta}B_{0,c}P_0 U_r \cdot U_r^{-1}P_0e^{L\tilde\Delta}P_0U_r^{-1}\cdot U_r P_0e^{-L\Delta}B_{0,c}P_0U_r.
  \end{equation}
  Applying $\|U_r^{-1}P_0e^{L\tilde\Delta}P_0U_r^{-1}\|_{\rm HS}\leq\frac{1}{\sqrt{L}}$ (by~\eqref{4b3}), \eqref{1b4} and $L\geq 1$, we have
  $\|\hat \Phi_{\sigma_n}\|_{\rm HS}\leq\beta^{2}e^{-\frac{8L^3}{3}}e^{-4Lc}$.
  \item Similarly, we can also show the case for $\sigma_1(1)=e.$.
    \end{enumerate}
Now let's consider the induction step: $n-1\mapsto n$.
\\
    \textbf{Case I: $\sigma_{n}\equiv a$}, then $\hat \Phi_{\sigma_{n}}=U_rP_0\hat B_{0,c}^{n}e^{-L\Delta}B_{0,c} P_0U_r$ (by~\eqref{4a1}). The result follows from \eqref{1b4}.
\\
\textbf{Case II: $\sigma_{n}\in\{a,c\}^{n}\setminus\{a\}^{n}$}. Define $f_c=\min\{j| \sigma_{n+1}(j)=c\}\geq 1$, then
\begin{equation}
\hat \Phi_{\sigma_n} =\underbrace{U_rP_0e^{-L\Delta}B_{0,c}\tilde B_{0,c}^{f_c-1}e^{L\tilde\Delta}P_0U_r^{-1}}_{=:\hat\varphi_1} \underbrace{U_rP_0e^{-L\Delta}B_{0,c}\bigg[\prod_{j=f_c+1}^{n}K_{\sigma_{n}(i)}\bigg]P_0U_r}_{=:\hat\varphi_2}.
\end{equation}
Using \eqref{4a1} and~\eqref{a102}, we have $\hat\varphi_1=U_rP_0\hat B_{0,c}^{f_c}P_0U_r^{-1}.$ Applying $\|\hat\varphi_1\|_{\rm op}\leq\beta^{f_c}$ by~\eqref{3ab3} and $\|\hat\varphi_2\|_{\rm HS}\leq\beta^{n-f_c+1}e^{-\frac{4L^3}{3(n+1-f_c)^2}}e^{-2Lc}$ (induction hypothesis), we have
$\|\hat \Phi_{\sigma_{n}}\|\leq\|\hat\varphi_1\|_{\rm op}\|\hat\varphi_2\|_{\rm HS}\leq\beta^{n+1}e^{-\frac{4L^3}{3(n+1)^2}}e^{-2Lc}$.
\\
    \textbf{Case III: $\sigma_{n}\in\{a,b,c\}^{n}\setminus\{a,c\}^{n}$.} Define $\ell_b=\max\{j|\sigma_{n}(j)=b\}\geq 1$, then
    \begin{equation}
        \hat\Phi_{\sigma_n}=\underbrace{U_rP_0e^{-L\Delta}B_{0,c}\bigg[\prod_{i=1}^{\ell_b-1}K_{\sigma_{n}(i)}\bigg]P_0U_r}_{\hat\varphi_1}\underbrace{U_r^{-1}P_0\tilde B_{0,c}\bigg[\prod_{j=\ell_b+1}^{n}K_{\sigma_{n}(i)}\bigg]P_0U_r}_{=:\hat\varphi_2}.
    \end{equation}
    By induction we have
$\|\hat\varphi_1\|_{\rm HS}\leq \beta^{\ell_b}e^{-\frac{4L^3}{3\ell_b^2}}e^{-2Lc}$. For $\hat\varphi_2$, if $\sigma_{n}(i)=a$ for all $i>\ell_b$, then we have
        $\hat\varphi_2=U_r^{-1}P_0\tilde B_{0,c}^{n+1-\ell_b}P_0U_r$.
   Applying $\|\hat\varphi_2\|_{\rm op}\leq\beta^{n+1-\ell_b}$ by \eqref{2ab3}, we have
    \begin{equation}
        \|\hat \Phi_{\sigma_n}\|_{\rm HS}\leq\|\hat\varphi_1\|_{\rm HS}\|\hat\varphi_2\|_{\rm op}\leq \beta^{n+1}e^{-\frac{4L^3}{3(n+1)^2}}e^{-2Lc}.
    \end{equation}
    If it exists $j\in\{\ell_b+1,\ldots,n\}$ with $\sigma_{n}(j)=c$, define $f_c=\min\{j>\ell_b| \sigma_{n}(j)=c\}$, then
 \begin{equation}
     \begin{aligned}
            &\hat\varphi_2=\underbrace{U_r^{-1}P_0\tilde B_{0,c}^{f_c-\ell_b}e^{L\tilde\Delta}P_0U_r^{-1}}_{=\hat\varphi_2^1} \underbrace{U_rP_0e^{-L\Delta}B_{0,c}\bigg[\prod_{j=f_c+1}^{n}K_{\sigma_{n}(i)}\bigg]P_0U_r}_{=:\hat\varphi_2^2}.
        \end{aligned}
 \end{equation}
 Applying $\|\hat\varphi_2^2\|_{\rm HS}\leq\beta^{n+1-f_c}e^{-\frac{4L^3}{3(n+1-f_c)^2}}e^{-2Lc}$ (induction assumption), $\|\hat\varphi_2^1\|_{\rm HS}\leq\beta^{f_c-\ell_b}e^{Lr^2}$ (by~\eqref{4bb4}) and $r^2\leq 2c$, we obtain
    \begin{equation}
    \begin{aligned}
       & \|\hat \Phi_{\sigma_n}\|_{\rm HS}\leq\|\hat\varphi_1\|_{\rm HS}\|\hat\varphi_2^1\|_{\rm HS}\|\hat\varphi_2^2\|_{\rm HS}\leq \beta^{n+1}e^{-\frac{4L^3}{3(n+1)^2}}e^{-2Lc},
    \end{aligned}
    \end{equation}
\textbf{Case IV: $\sigma_{n}\in\{a,b,c,d\}^{n}\setminus\{a,b,c\}^{n}$.} Set $f_d=\min\{i|\sigma_{n}(i)=d\}\geq 1$. Then
\begin{equation}
\begin{aligned}
    \hat \Phi_{\sigma_n}&=\underbrace{U_r P_0e^{-L\Delta}B_{0,c}\bigg[\prod_{i=1}^{f_d-1}K_{\sigma_{n}(i)}\bigg]P_0U_r}_{=:\hat\varphi_1} \underbrace{U_r^{-1} P_0e^{L\tilde \Delta}P_0U_r^{-1}}_{=:\hat\varphi_2}\\
    &\times \underbrace{U_r P_0e^{-L\Delta}B_{0,c}\bigg[\prod_{i=f_d+1}^{n}K_{\sigma_{n}(i)}\bigg]P_0 U_r}_{=:\hat\varphi_3}.
\end{aligned}
\end{equation}
Applying $\|\hat\varphi_1\|_{\rm HS}\leq\beta^{f_d}e^{-\frac{4L^3}{3f_d^2}}e^{-2Lc}$, $\|\hat\varphi_3\|_{\rm HS}\leq\beta^{n+1-f_d}e^{-\frac{4L^3}{3(n-f_d+1)^2}}e^{-2Lc}$ (induction assumption) and $\|\hat\varphi_2\|_{\rm HS}\leq \tfrac{1}{\sqrt{L}}$ (by~\eqref{4b3}), we have
\begin{equation}
\|\hat \Phi_{\sigma_n}^{u,w}\|_{\rm HS}\leq\|\hat\varphi_1\|_{\rm HS}\|\hat\varphi_2\|_{\rm HS}\|\hat\varphi_3\|_{\rm HS}\leq\beta^{n+1}e^{-\frac{4L^3}{3(n+1)^2}}e^{-2Lc}.
\end{equation}
\textbf{Case V: $\sigma_{n}\in\{a,b,c,d,e\}^{n}\setminus\{a,b,c,d\}^{n}$}. Same as \textbf{Case IV}.
\end{proof}
\begin{lem}\label{uvwkey11}
    For $n\in\mathbb Z_{\geq 0}$ and $\sigma_n\in\{a,c\}^n$, it holds $\|\hat \Phi_{\sigma_n}\|_{\rm op}\leq \beta^{n+1}$, where
    \begin{equation}
        \hat\Phi_{\sigma_n}=U_r P_0e^{-L\Delta}B_{0,c}\bigg[\prod_{i=1}^nK_{\sigma_n(i)}\bigg]e^{L\Delta}P_0U_r^{-1}
    \end{equation}
\end{lem}
\begin{proof}The case $n=0$ follows from~\eqref{1ab3} and $\|\cdot\|_{\rm op}\leq\|\cdot\|_{\rm HS}$. We show this via induction on $n$. For $n\geq1$, we have a few cases. If $\sigma_1=a$, then $\hat \Phi_{\sigma_n}=U_r P_0\hat B_{0,c}B_{0,c}P_0U_r^{-1}$ by~\eqref{4a1}, the result follows from~\eqref{1cb3}. If $\sigma_1=c$, then $\hat \Phi_{\sigma_1}=U_rP_0\hat B_{0,c}P_0 B_{0,c}P_0U_r^{-1}$ by~\eqref{a102}. Applying $\|U_rP_0\hat B_{0,c}P_0 U_r^{-1}\|_{\rm op} \leq\beta$ (by~\eqref{3ab3}) and $ \| U_rP_0 B_{0,c}P_0 U_r^{-1}\|_{\rm op}\leq \| U_rP_0 B_{0,c}P_0 U_r^{-1}\|_{\rm HS}\leq \beta$ (by~\eqref{1ab3}), we then have
    \begin{equation}
    \|\hat \Phi_{\sigma_1}\|_{\rm op}\leq \|U_rP_0\hat B_{0,c}P_0 U_r^{-1}\|_{\rm op} \| U_rP_0 B_{0,c}P_0 U_r^{-1}\|_{\rm op}\leq \beta^2
    \end{equation}
    For induction step $n-1\mapsto n$, we consider the following cases.\\
    \textbf{Case I: $\sigma_{n}\equiv a$}: then
        $\hat \Phi_{\sigma_n}=U_r P_0\hat B_{0,c}^nB_{0,c}P_0U_r^{-1}$ by~\eqref{4a1}, the result is true by~\eqref{1cb3}. \\
\textbf{Case II: $\sigma_n\in\{a,c\}^n\setminus\{a\}^n$}: Defining $f_c=\min\{j|\sigma_n(j)=c\}\geq 1$ we have
\begin{equation}
    \hat \Phi_{\sigma_n}=\underbrace{U_r P_0e^{-L\Delta}B_{0,c}\tilde B_{0,c}^{f_c-1}e^{L\tilde\Delta}P_0U_r^{-1}}_{=:\hat\varphi_1} \underbrace{U_r P_0e^{-L\Delta}B_{0,c}\bigg[\prod_{i=f_c+1}^{n}K_{\sigma_n(i)}\bigg]e^{L\Delta}P_0U_r^{-1}}_{=:\hat\varphi_2}.
\end{equation}
By~\eqref{a103}, we have
    $\hat\varphi_1=U_r P_0\hat B_{0,c}^{f_c}P_0U_r^{-1}$. Applying $\|\hat\varphi_1\|_{\rm op}\leq \beta^{f_c}$ (by~\eqref{3ab3}) and $\|\hat\varphi_2\|_{\rm op}\leq \beta^{n-f_c+1}$ (induction assumption), we obtain $\|\hat \Phi_{\sigma_n}\|_{\rm op}\leq\|\hat\varphi_1\|_{\rm op}\|\hat\varphi_2\|_{\rm op}\leq\beta^{n+1}$.\\
\end{proof}

\begin{lem}\label{vwkey1}
   Let $n\in\mathbb Z_{\geq 0}$, $\sigma_n\in\{a,b,c\}^n$, then $\|\hat \Phi_{\sigma_n}\|_{\rm HS}\leq\beta^{n+1}e^{Lr^2}$, where
    \begin{equation}
        \hat \Phi_{\sigma_n}=U_r^{-1}P_0\tilde B_{0,c}\bigg[\prod_{j=1}^{n}K_{\sigma_n(j)}\bigg]e^{L\Delta}P_0U_r^{-1}
    \end{equation}
\end{lem}
\begin{proof}
    We show this via induction on $n$. The case $n=0$ follows from~\eqref{5b4}. The case $n=1$ can be handled similarly as before, so we omit the proof here. Now we consider the induction step $n-1\mapsto n$.\\
    \textbf{Case I: $\sigma_{n}\equiv a$}: then $\hat \Phi_{\sigma_n}=U_r^{-1}P_0\tilde B_{0,c}^{n+1}e^{L\Delta}P_0U_r^{-1}$, the result follows from~\eqref{5b4}\\
    \textbf{Case II: $\sigma_{n}\in\{a,c\}^n\setminus\{a\}^n$}. We set $\ell_c=\min\{j|\sigma_{n}(j)=c\}\geq 1$. Then
    \begin{equation}
           \hat \Phi_{\sigma_n}=\underbrace{U_r^{-1}P_0\tilde B_{0,c}^{\ell_c}e^{L\tilde\Delta}P_0U_r^{-1}}_{=:\hat\varphi_1}\underbrace{ U_r P_0e^{-L\Delta}B_{0,c}\bigg[\prod_{j=\ell_c+1}^{n}K_{\sigma_{n}(j)}\bigg]e^{L\Delta}P_0U_r^{-1}}_{=:\hat\varphi_2}.
    \end{equation}
Applying $\|\hat\varphi_1\|_{\rm HS}\leq \beta^{\ell_c}e^{Lr^2}$ (by~\eqref{4bb4}) and $\|\hat\varphi_2\|_{\rm op}\leq\beta^{n-\ell_c+1}$ (by Lemma~\ref{uvwkey11}), we obtain $\|\hat\Phi_{\sigma_n}\|_{\rm HS}\leq \|\hat\varphi_1\|_{\rm HS}\|\hat\varphi_2\|_{\rm op}\leq\beta^{n+1}e^{Lr^2}$.\\
    \textbf{Case III: $\sigma_{n}\in\{a,b,c\}^{n}\setminus\{a,c\}^{n}$}. Define $f_b=\min\{j|\sigma_{n}(j)=b\}\geq 1$. Then
    \begin{equation}
        \hat\Phi_{\sigma_n}=\underbrace{U_r^{-1}P_0\tilde B_{0,c}\bigg[\prod_{j=1}^{f_b-1}K_{\sigma_{n}(j)}\bigg]P_0U_r}_{=:\hat\varphi_1} \underbrace{U_r^{-1}P_0\tilde B_{0,c}\bigg[\prod_{j=f_b+1}^{n}K_{\sigma_{n}(j)}\bigg]e^{L\Delta}P_0U_r^{-1}}_{=:\hat\varphi_2}.
    \end{equation}
    By induction assumption, $\|\hat \varphi_2\|_{\rm HS}\leq\beta^{n+1-f_b}e^{Lr^2}$. Next we deal with $\hat\varphi_1$. If $\sigma_{n}(j)=a$ for all $j\in\{1,\ldots,f_{b}-1\}$, then $\hat\varphi_1=U_r^{-1}P_0\tilde B_{0,c}^{f_b}P_0U_r$ and $\|\hat\varphi_1\|_{\rm op}\leq\beta^{f_b}$ (by~\eqref{2ab3}). Then $\|\hat\Phi_{\sigma_n}\|_{\rm HS}\leq\|\hat\varphi_1\|_{\rm op}\|\hat\varphi_2\|_{\rm HS}\leq \beta^{n+1}e^{Lr^2}$. If there exists $j\in\{1,\ldots,f_b-1\}$ such that $\sigma_{n}(j)=c$, then we define $f_c=\min\{j\leq f_{b}-1| \sigma_{n}(j)=c\}$ and decompose
\begin{equation}
        \hat \varphi_1=\underbrace{U_r^{-1}P_0\tilde B_{0,c}^{f_c}e^{L\tilde\Delta}P_0U_r^{-1}}_{=:\hat\varphi_1^1}\underbrace{ U_r P_0e^{-L\Delta}B_{0,c}\bigg[\prod^{f_b-1}_{f_c+1}K_{\sigma_n(j)}\bigg]P_0U_r}_{\hat\varphi_1^2}.
\end{equation}
Applying now $\|\hat\varphi_1^1\|_{\rm HS}\leq \beta^{f_c}e^{Lr^2}$ (by~\eqref{4bb4}), $\|\hat\varphi_1^2\|_{\rm HS}\leq\beta^{f_b-f_c}e^{-2cL}$ (by Lemma~\ref{uwkey1}) and $r^2\leq 2c$, we get $\|\hat \Phi_{\sigma_{n}}\|_{\rm HS}\leq \|\hat\varphi_1^1\|_{\rm HS}\|\hat\varphi_1^2\|_{\rm HS}\|\hat\varphi_2\|_{\rm HS}\leq \beta^{n+1}e^{Lr^2}$.
\end{proof}

\begin{lem}\label{uvwkey1}
    For $n\in\mathbb Z_{\geq 0}$ and $\sigma_n\in\{a,b,c,d,e,v\}^n$, the operator
    \begin{equation}
        \hat\Phi_{\sigma_n}=U_rP_0e^{-L\Delta}B_{0,c}\left[\prod_{i=1}^nK_{\sigma_n(i)}\right]e^{L\Delta}P_0U_r^{-1}
    \end{equation}
    satisfies the following bounds:
    \begin{enumerate}
        \item if $\sigma_n\in\{a,c,v\}^n$, then
            $\|\hat \Phi_{\sigma_n}\|_{\rm op}\leq \beta^{n+1}$,
        \item\label{a42} if $\sigma_n\in\{a,b,c,d,e,v\}\setminus\{a,c,v\}^n$, then
            $\|\hat \Phi_{\sigma_n}\|_{\rm op}\leq\|\hat \Phi_{\sigma_n}\|_{\rm HS}\leq \beta^{n+1}e^{-\frac{4L^3}{3n^2}}$.
    \end{enumerate}
\end{lem}
\begin{proof}
The case $\sigma_n\in\{a,c\}^n$ is solved in Lemma~\ref{uvwkey11}. We show the rest via induction on $n$ and omit the details for $n=1$. For induction step $n-1\mapsto n$. \\
\textbf{Case I: $\sigma_n\in\{a,b,c\}^n\setminus\{a,c\}^n$.} Define $\ell_b=\max\{j|\sigma_n(j)=b\}\geq 1$,
    then we have
    \begin{equation}
        \hat \Phi_{\sigma_n}=\underbrace{U_r P_0e^{-L\Delta}B_{0,c}\bigg[\prod_{i=1}^{\ell_b-1}K_{\sigma_n(i)}\bigg]P_0U_r}_{=:\hat\varphi_1}\underbrace{U_r^{-1} P_0\tilde B_{0,c}\bigg[\prod_{i=\ell_b+1}^{n}K_{\sigma_n(i)}\bigg]e^{L\Delta}P_0U_r^{-1}}_{=:\hat\varphi_2}.
    \end{equation}
Since $\sigma_n\in\{a,b,c\}^n$, we have $\|\hat\varphi_1\|_{\rm HS}\leq \beta^{\ell_b}e^{-\frac{4L^3}{3\ell_b^2}}e^{-2Lc}$ (by Lemma~\ref{uwkey1}) and $\|\hat\varphi_2\|_{\rm HS}\leq\beta^{n-\ell_b+1}e^{Lr^2}$ (by Lemma~\ref{vwkey1}), which implies $\|\hat \Phi_{\sigma_n}\|_{\rm HS}\leq\|\hat \varphi_1\|_{\rm HS}\|\hat\varphi_2\|_{\rm HS}\leq \beta^{n+1}e^{-\frac{4L^3}{3\ell_b^2}}e^{Lr^2-2Lc}\leq\beta^{n+1}e^{-\frac{4L^3}{3n^2}}$, where we use the assumption $r^2\leq2c$.\\
\textbf{Case II: $\sigma_n\in\{a,b,c,d\}^n\setminus\{a,b,c\}^n$.} Setting $\ell_d=\min\{i|\sigma_n(i)=d\}\geq 1$, we have
    \begin{equation}
    \begin{aligned}
        \hat\Phi_{\sigma_n}&=\underbrace{U_r P_0e^{-L\Delta}B_{0,c}\bigg[\prod_{i=1}^{\ell_d-1}K_{\sigma_n(i)}\bigg]P_0U_r}_{=:\hat\varphi_1} \underbrace{U_r^{-1}P_0e^{L\Delta}P_0U_r^{-1}}_{=:\hat\varphi_2} \\
        &\times\underbrace{ U_r P_0e^{-L\Delta}B_{0,c}\bigg[\prod_{i=\ell_d+1}^{n}K_{\sigma_n(i)}\bigg]e^{L\Delta}P_0U_r^{-1}}_{=:\hat\varphi_3}.
    \end{aligned}
    \end{equation}
    Applying the bounds $\|\hat\varphi_1\|_{\rm HS}\leq\beta^{\ell_d} e^{-\frac{4L^3}{3\ell_d^2}}e^{-2Lc}$ (by Lemma~\ref{uwkey1}), $\|\hat\varphi_2\|_{\rm HS}\leq \tfrac{1}{\sqrt{L}}$ (by~\eqref{4b3}) and  $\|\hat\varphi_3\|_{\rm op}\leq \beta^{n-\ell_d+1}$ (induction assumption), we have $\|\hat\Phi_{\sigma_n}\|_{\rm HS}\leq \|\hat\varphi_1\|_{\rm HS}\|\hat\varphi_2\|_{\rm HS}\|\hat\varphi_3\|_{\rm op}\leq\beta^{n+1}e^{-\frac{4L^3}{3n^2}}.$\\
    \textbf{Case III: $\sigma_n\in\{a,b,c,d,e\}^n\setminus\{a,b,c,d\}^n$.} It is the same decomposition as in the previous case except that in $\hat\varphi_2$ there is $e^{L\tilde\Delta}$ instead of $e^{L\Delta}$ and we use~\eqref{4b3}.\\
     \textbf{Case IV: $\sigma_n\in\{a,b,c,d,e,v\}^n\setminus\{a,b,c,d,e\}^n$}. Define $\ell_v=\max\{i\mid\sigma_n(i)=v\}\leq n$, then
    \begin{equation}
       \hat \Phi_{\sigma_n}=\underbrace{U_r P_0e^{-L\Delta}B_{0,c}\bigg[\prod_{i=1}^{\ell_v-1}K_{\sigma_n(i)}\bigg]e^{L\Delta}P_0U_r^{-1}}_{=:\hat\varphi_1} \underbrace{U_r P_0e^{-L\Delta}B_{0,c}\bigg[\prod_{i=\ell_v+1}^{n}K_{\sigma_n(i)}\bigg]e^{L\Delta}P_0U_r^{-1}}_{=:\hat\varphi_2}.
    \end{equation}
    The results follows from induction assumption.
\end{proof}

\begin{cor}\label{ca10}
   Let $n\in\mathbb Z_{\geq 0}$, $\sigma_n\in\{a,b,c,d,e,v\}^n$, then $\|\hat \Phi_{\sigma_n}\|_{\rm HS}\leq\beta^{n+1}e^{-\frac{4L^3}{3(n+1)^2}}e^{-2Lc}$, where
    \begin{equation}
        \hat \Phi_{\sigma_n}=U_rP_0e^{-L\Delta}B_{0,c}\bigg[\prod_{i=1}^nK_{\sigma_n(i)}\bigg]P_0U_r.
    \end{equation}
\end{cor}
\begin{proof}
    If $\sigma_n\in\{a,b,c,d,e\}^n$, the result follows from Lemma~\ref{uwkey1}. Consider $\sigma_n\in\{a,b,c,d,e,v\}^n$ with $v\in\sigma_n$,  we have
    \begin{equation}
        \hat \Phi_{\sigma_n}=\underbrace{U_rP_0e^{-L\Delta}B_{0,c}\bigg[\prod_{i=1}^{\ell_v-1}K_{\sigma_n(i)}\bigg]e^{L\Delta}P_0U_r^{-1}}_{=:\hat\varphi_1}\underbrace{U_r P_0e^{-L\Delta}B_{0,c}\bigg[\prod_{i=\ell_v+1}^nK_{\sigma_n(i)}\bigg]P_0U_r}_{=:\hat\varphi_2},
    \end{equation}
    where $\ell_v=\max\{i|\sigma_n(i)=v\}\geq 1$. Applying $\|\hat\varphi_1\|_{\rm op}\leq \beta^{\ell_v}$ (by Lemma~\ref{uvwkey1}) and $\|\hat\varphi_2\|_{\rm HS}\leq \beta^{n-\ell_v+1}e^{-\frac{4L^3}{(n-\ell_v+1)^2}}e^{-2Lc}$ (by definition of $\ell_v$, we can use Lemma~\ref{uwkey1}), we then have $\|\hat\Phi_{\sigma_n}\|_{\rm HS}\leq \|\hat\varphi_1\|_{\rm op}\|\hat\varphi_2\|_{\rm HS}\leq\beta^{n+1}e^{-\frac{4L^3}{3(n+1)^2}}e^{-2Lc}$.
\end{proof}

\begin{lem}\label{uvwkey2}
    For $n\in\mathbb Z_{\geq 0}$ and $\sigma_n\in\{a,b,c\}^n$, we have $\|\hat \Phi_{\sigma_n}\|_{\rm HS}\leq \beta^{n+1}e^{Lr^2}$, where
    \begin{equation}
        \hat \Phi_{\sigma_n}=U_r^{-1}P_0B_{0,c}\bigg[\prod_{i=1}^{n}K_{\sigma_n(i)}\bigg]e^{L\tilde\Delta}P_0U_r^{-1}.
    \end{equation}
\end{lem}
\begin{proof}
We show this via induction on $n$ the case $n=0$ follows from~\eqref{3b4}. We omit details for $n=1$. For induction step $n-1\mapsto n$, consider the following cases.\\
\textbf{Case I: $\sigma_n\equiv a$}. Then  $\hat\Phi_{\sigma_n}=U_r^{-1}P_0B_{0,c}\tilde B_{0,c}^{n}e^{L\tilde\Delta}P_0U_r^{-1}$, the result is true by~\eqref{6b4}.\\
\textbf{Case II: $\sigma_n\in\{a,b\}^n\setminus\{a\}^n$}. Define $\ell_b=\max\{j|\sigma_n(j)=b\}\geq 1$, then
    \begin{equation}
        \hat \Phi_{\sigma_n}=\underbrace{U_r^{-1}P_0B_{0,c}\bigg[\prod_{i=1}^{\ell_b-1}K_{\sigma_n(i)}\bigg]P_0U_r}_{=:\hat\varphi_1}\underbrace{U_r^{-1} P_0\tilde B_{0,c}^{n-\ell_b+1}e^{L\tilde\Delta}P_0U_r^{-1}}_{=:\hat\varphi_2}.
    \end{equation}
Since $\sigma_n\in\{a,b\}^n$, we can apply Lemma~\ref{la121} to get $\|\hat\varphi_1\|_{\rm op}\leq \beta^{\ell_b}$, together with $\|\hat\varphi_2\|_{\rm HS}\leq\beta^{n-\ell_b+1}e^{Lr^2}$ (by~\eqref{4bb4}), we have $\|\hat \Phi_{\sigma_n}\|_{\rm HS}\leq \|\hat\varphi_1\|_{\rm op}\|\hat\varphi_2\|_{\rm HS}\leq\beta^{n+1}e^{Lr^2}.$\\
\textbf{Case III: $\sigma_n\in\{a,b,c\}^n\setminus\{a,b\}^n$}. Define
        $\ell_c=\max\{j|\sigma_n(j)=c\}\geq 1$, then
\begin{equation}
\hat\Phi_{\sigma_n}=\underbrace{U_r^{-1}P_0B_{0,c}\bigg[\prod_{i=1}^{\ell_c-1}K_{\sigma_n(i)}\bigg]e^{L\tilde\Delta}P_0U_r^{-1}}_{=:\hat\varphi_1}\underbrace{ U_rP_0e^{-L\Delta}B_{0,c}\bigg[\prod^{n}_{i=\ell_c+1}K_{\sigma_n(i)}\bigg]e^{L\tilde\Delta}P_0U_r^{-1}}_{=:\hat\varphi_2}.
\end{equation}
 By induction assumption $\|\hat\varphi_1\|_{\rm HS}\leq\beta^{\ell_c}e^{Lr^2}$. Now we need to deal with $\hat\varphi_2$. If $\sigma_n(i)=a$ for all $i\in\{\ell_c+1,\ldots,n\}$, applying~\eqref{a103}, we have
        $\hat\varphi_2=U_rP_0\hat B_{0,c}^{n-\ell_c+1}P_0U_r^{-1}$ and $\|\hat\varphi_2\|_{\rm op}\leq \beta^{n-\ell_c+1}$ (by~\eqref{3ab3}), so that
    $\|\hat\Phi_{\sigma_n}\|_{\rm HS}\leq\|\hat\varphi_1\|_{\rm HS}\|\hat\varphi_2\|_{\rm op}\leq\beta^{n+1}e^{Lr^2}$.
    If it exists $j>\ell_c+1$ with $\sigma_n(j)=b$, setting $f_b=\min\{j\geq\ell_c+1|\sigma_n(j)=b\}$, we get
    \begin{equation}
        \hat\varphi_2= \underbrace{U_r P_0e^{-L\Delta}B_{0,c}\bigg[\prod^{f_b-1}_{i=\ell_c+1}K_{\sigma_n(i)}\bigg]P_0U_r}_{=:\hat\varphi_2^1}\underbrace{U_r^{-1} P_0\tilde B_{0,c}\bigg[\prod_{i=f_b+1}^{n}K_{\sigma_n(i)}\bigg]e^{L\tilde\Delta}P_0U_r^{-1}}_{=:\hat\varphi_2^2}.
    \end{equation}
     Applying $\|\hat\varphi_2^1\|_{\rm HS}\leq \beta^{f_b-\ell_c}e^{-\frac{4L^3}{3(f_b-\ell_c)^2}}e^{-2Lc}$ (by Lemma~\ref{ca10}), $\|\hat\varphi_2^2\|_{\rm HS}\leq \beta^{n-f_b+1}e^{Lr^2}$ (induction assumption) and $r^2\leq 2c$, we have $\|\hat\Phi_{\sigma_n}\|_{\rm HS}\leq\|\hat\varphi_1\|_{\rm HS}\|\hat\varphi_2^1\|_{\rm HS}\|\hat\varphi_2^2\|_{\rm HS}\leq \beta^{n+1}e^{Lr^2}$.
\end{proof}

\begin{lem}\label{la12}
    For any $n\in\mathbb Z_{\geq 0}$ and $\sigma_n\in\{a,b,c,d,e\}^n$, the operator
    \begin{equation}
    \hat\Phi_{\sigma_n}=U_r^{-1}P_0B_{0,c}\bigg[\prod_{i=1}^nK_{\sigma_n(i)}\bigg]P_0U_r
    \end{equation}
satisfies the following bounds:
    \begin{enumerate}
        \item if $\sigma_n\in\{a,b\}^n$, then $\|\hat \Phi_{\sigma_n}\|_{\rm op}\leq \beta^{n+1}$,
        \item if $\sigma_n\in\{a,b,c,d,e\}^n\setminus\{a,b\}^n$, then $\|\hat \Phi_{\sigma_n}\|_{\rm op}\leq \|\hat \Phi_{\sigma_n}\|_{\rm HS}\leq \beta^{n+1}e^{-\frac{4L^3}{3n^2}}$.
    \end{enumerate}
\end{lem}
\begin{proof}
The case $\sigma_n\in\{a,b\}^n$ is solved in Lemma~\ref{la121}. We show the rest cases via induction on $n$ and  omit the details for $n=1$. Consider induction step $n-1\mapsto n$.\\
\textbf{Case I: $\sigma_n\in\{a,b,c\}^n\setminus\{a,b\}^n$.} Define $\ell_c=\min\{i|\sigma_n(i)=c\}\geq 1$, then
    \begin{equation}
    \hat \Phi_{\sigma_n}=\underbrace{U_r^{-1}P_0B_{0,c}\bigg[\prod_{i=1}^{\ell_c-1}K_{\sigma_n(i)}\bigg]e^{L\tilde\Delta}P_0U_r^{-1}}_{=:\hat\varphi_1} \underbrace{U_rP_0e^{-L\Delta}B_{0,c}\bigg[\prod_{i=\ell_c+1}^{n}K_{\sigma_n(i)}\bigg]P_0U_r}_{=:\hat\varphi_2}.
\end{equation}
Applying $\|\hat\varphi_1\|_{\rm HS}\leq \beta^{\ell_c}e^{Lr^2}$ (by Lemma~\ref{uvwkey2}), $\|\hat\varphi_2\|_{\rm HS}\leq\beta^{n-\ell_c+1}e^{-\frac{4L^3}{3(n-\ell_c)^2}}e^{-2Lc}$ (by Lemma~\ref{uwkey1}) and $r^2\leq 2c$, we get $\|\hat\Phi_{\sigma_n}\|_{\rm HS}\leq\|\hat\varphi_1\|_{\rm HS}\|\hat \varphi_2\|_{\rm HS}\leq e^{-\frac{4L^3}{3n^2}}\beta^{n+1}$.\\
\textbf{Case II: $\sigma_n\in\{a,b,c,d\}^n\setminus\{a,b,c\}^n$.} Set $\ell_d=\max\{i|\sigma_n(i)=d\}\geq 1$, then
\begin{equation}
\begin{aligned}
    \hat \Phi_{\sigma_n}&=\underbrace{U_r^{-1}P_0B_{0,c}\bigg[\prod_{i=1}^{\ell_d-1}K_{\sigma_n(i)}\bigg]P_0U_r}_{=:\hat\varphi_1} \underbrace{U_r^{-1}P_0e^{L\tilde\Delta}P_0U_r^{-1}}_{=:\hat\varphi_2}\\
    &\times\underbrace{U_rP_0e^{-L\Delta}B_{0,c}\bigg[\prod_{i=\ell_d+1}^{n}K_{\sigma_n(i)}\bigg]P_0U_r}_{=:\hat\varphi_3}.
\end{aligned}
\end{equation}
Applying $\|\hat\varphi_1\|_{\rm op}\leq \beta^{\ell_d}$ (induction assumption), $\|\hat\varphi_2\|_{\rm HS}\leq \frac{1}{\sqrt{L}}$ (by~\eqref{4b3}) and $\|\hat\varphi_3\|_{\rm HS}\leq \beta^{n-\ell_d+1}e^{-\frac{4L^3}{(n-\ell_d+1)^2}}e^{-2Lc}$ (by Lemma~\ref{uwkey1}), we have $\|\hat\Phi_{\sigma_n}\|_{\rm HS}\leq \|\hat\varphi_1\|_{\rm op}\|\hat\varphi_2\|_{\rm HS}\|\hat\varphi_3\|_{\rm HS}\leq\beta^{n+1}e^{-\frac{4L^3}{3n^2}}$.\\
\textbf{Case III: $\sigma_n\in\{a,b,c,d,e\}^n\setminus\{a,b,c,d\}^n$.} This is the same as the previous case, except that $e^{L\tilde\Delta}$ is replaced by $e^{L\Delta}$.
\end{proof}

\begin{lem}\label{uvwkey3}
    Let $n\in\mathbb Z_{\geq 0}$, $\sigma_n\in\{a,b,c,d,e,v\}^n$, then $\|\hat\Phi_{\sigma_n}\|_{\rm HS}\leq\beta^{n+1}e^{Lr^2}$, where
    \begin{equation}
        \hat\Phi_{\sigma_n}=U_r^{-1}P_0B_{0,c}\bigg[\prod_{i=1}^{n}K_{\sigma_n(i)}\bigg]e^{L\Delta}P_0U_r^{-1}.
    \end{equation}
\end{lem}
\begin{proof}
We show this via induction on $n$. The case $n=0$ follows from~\eqref{2b4}. We omit the details for $n=1$. Consider the induction step $n-1\mapsto n$.
\\
\textbf{Case I: $\sigma_n\in\{a,b\}^n$}. If $\sigma_n\equiv a$, then
$\hat \Phi_{\sigma_n}=U_r^{-1}P_0B_{0,c}\tilde B_{0,c}^ne^{L\Delta} P_0U_r^{-1}$, the result follows from~\eqref{6b4}. If $b\in\sigma_n$, define $\ell_b=\max\{i|\sigma_n(i)=b\}\geq 1$, then
\begin{equation}
 \hat\Phi_{\sigma_n}=\underbrace{U_r^{-1}P_0B_{0,c}\bigg[\prod_{i=1}^{\ell_b-1}K_{\sigma_n(i)}\bigg]P_0U_r}_{=:\hat\varphi_1}\underbrace{U_r^{-1} P_0\tilde B_{0,c}^{n-\ell_b+1}e^{L\Delta} P_0U_r^{-1}}_{=:\hat\varphi_2}.
\end{equation}
Applying $\|\hat\varphi_1\|_{\rm op}\leq\beta^{\ell_b}$ (by Lemma~\ref{la12}) and $\|\hat\varphi_2\|_{\rm HS}\leq \beta^{n-\ell_b+1}e^{Lr^2}$ (by~\eqref{5b4}), we have $\|\hat\Phi_{\sigma_n}\|_{\rm HS}\leq \|\hat\varphi_1\|_{\rm op}\|\hat\varphi_2\|_{\rm HS}\leq\beta^{n+1}e^{Lr^2}.$\\
\textbf{Case II: $\sigma_n\in\{a,b,c\}^n\setminus\{a,b\}^n$.} Define $\ell_c=\max\{i|\sigma_n(i)=c\}\geq 1$ and
\begin{equation}
 \hat\Phi_{\sigma_n}=\underbrace{U_r^{-1}P_0B_{0,c}\bigg[\prod_{i=1}^{\ell_c-1}K_{\sigma_n(i)}\bigg]e^{L\tilde\Delta}P_0U_r^{-1}}_{=:\hat\varphi_1}\underbrace{ U_r P_0e^{-L\Delta}B_{0,c}\bigg[\prod_{i=\ell_c+1}^{n}K_{\sigma_n(i)}\bigg]e^{L\Delta}P_0U_r^{-1}}_{=:\hat\varphi_2}.
\end{equation}
Applying $\|\hat\varphi_1\|_{\rm HS}\leq \beta^{\ell_c}e^{Lr^2}$ (by Lemma~\ref{uvwkey2}) and $\|\hat\varphi_2\|_{\rm op}\leq\beta^{n-\ell_c+1}$ (by Lemma~\ref{uvwkey1}), we have $\|\hat \Phi_{\sigma_n}\|_{\rm HS}\leq \|\hat\varphi_1\|_{\rm HS}\|\hat\varphi_2\|_{\rm op}\leq \beta^{n+1}e^{Lr^2}$.\\
\textbf{Case III: $\sigma_n\in\{a,b,c,d\}^n\setminus\{a,b,c\}^n$.} Define $\ell_d=\min\{i|\sigma_n(i)=d\}\geq 1$, then
\begin{equation}
\begin{aligned}
    \hat\Phi_{\sigma_n}&=\underbrace{U_r^{-1}P_0B_{0,c}\bigg[\prod_{i=1}^{\ell_d-1}K_{\sigma_n(i)}\bigg]P_0U_r}_{=:\hat\varphi_1}\underbrace{U_r^{-1}P_0e^{L\tilde\Delta}P_0U_r^{-1}}_{=:\hat\varphi_2}
    \\&\times\underbrace{U_r P_0e^{-L\Delta}B_{0,c}\bigg[\prod_{i=\ell_d+1}^{n}K_{\sigma_n(i)}\bigg]e^{L\Delta}P_0U_r^{-1}}_{=:\hat\varphi_3}.
\end{aligned}
\end{equation}
Applying $\|\hat\varphi_1\|_{\rm op}\leq\beta^{\ell_d}$ (by Lemma~\ref{la12}), $\|\hat\varphi_2\|_{\rm HS}\leq\frac{1}{\sqrt{L}}$ (by~\eqref{4b3}) and $\|\hat\varphi_3\|_{\rm op}\leq \beta^{n-\ell_d+1}$ (by Lemma~\ref{uvwkey1}), we have $\|\hat \Phi_{\sigma_n}\|_{\rm HS}\leq\|\hat\varphi_1\|_{\rm op}\|\hat\varphi_2\|_{\rm HS}\|\hat\varphi_3\|_{\rm op}\leq\beta^{n+1}\leq\beta^{n+1}e^{Lr^2}$.
\\
\textbf{Case IV: $\sigma_n\in\{a,b,c,d,e\}^n\setminus\{a,b,c,d\}^n$.} It is the same as the previous case.
\\
\textbf{Case V: $\sigma_n\in\{a,b,c,d,e,v\}^n\setminus\{a,b,c,d,e\}^n$.} Define $\ell_v=\min\{i|\sigma_n(i)=v\}\geq 1$, then
\begin{equation}
    \hat\Phi_{\sigma_n}=\underbrace{U_r^{-1}P_0B_{0,c}\bigg[\prod_{i=1}^{\ell_v-1}K_{\sigma_n(i)}\bigg]e^{L\Delta}P_0U_r^{-1}}_{=:\hat\varphi_1} \underbrace{U_r P_0e^{-L\Delta}B_{0,c}\bigg[\prod_{i=\ell_v+1}^{n}K_{\sigma_n(i)}\bigg]e^{L\Delta}P_0U_r^{-1}}_{=:\hat\varphi_2}.
\end{equation}
Applying $\|\hat\varphi_1\|_{\rm HS}\leq\beta^{\ell_v}e^{Lr^2}$ (induction assumption) and $\|\hat\varphi_2\|_{\rm op}\leq \beta^{n-\ell_v+1} $ (by Lemma~\ref{uvwkey1}), we have $\|\hat \Phi_{\sigma_n}\|_{\rm HS}\leq \|\hat\varphi_1\|_{\rm HS}\|\hat\varphi_2\|_{\rm op}\leq\beta^{n+1}e^{Lr^2}$.
\end{proof}

\begin{cor}\label{ca14}
    For any $n\in\mathbb Z_{\geq 0}$ and $\sigma_n\in\{a,b,c,d,e,v,u\}^n$, the operator
    \begin{equation}
        \hat\Phi_{\sigma_n}=U_r^{-1}P_0B_{0,c}\bigg[\prod_{i=1}^nK_{\sigma_n(i)}\bigg]P_0U_r
    \end{equation}
satisfies the following bounds:
    \begin{enumerate}
        \item if $\sigma_n\in\{a,b,u\}^n$, then $\|\hat \Phi_{\sigma_n}\|_{\rm op}\leq \beta^{n+1}$,
        \item if $\sigma_n\in\{a,b,c,d,e,v,u\}^n\setminus\{a,b,u\}^n$, then $\|\hat \Phi_{\sigma_n}\|_{\rm op}\leq\|\hat \Phi_{\sigma_n}\|_{\rm HS}\leq \beta^{n+1}e^{-\frac{4L^3}{3n^2}}$.
    \end{enumerate}
\end{cor}
\begin{proof}
    The case for $\sigma_n\in\{a,b,c,d,e\}^n$ is proved in Lemma~\ref{la12}. We show rest cases via induction on $n$ and omit details for $n=1$. Consider induction step $n-1\mapsto n$.
\\
    \textbf{Case I: $\sigma_n\in\{a,b,c,d,e,v\}^n\setminus\{a,b,c,d,e\}^n$.} Set $\ell_v=\max\{i|\sigma_n(i)=v\}\geq 1$, then
    \begin{equation}
        \hat\Phi_{\sigma_n}=\underbrace{U_r^{-1}P_0B_{0,c}\bigg[\prod_{i=1}^{\ell_v-1}K_{\sigma_n(i)}\bigg]e^{L\Delta}P_0U_r^{-1}}_{=:\hat\varphi_1} \underbrace{U_rP_0e^{-L\Delta}B_{0,c}\bigg[\prod_{i=\ell_v+1}^{n}K_{\sigma_n(i)}\bigg]P_0U_r}_{=:\hat\varphi_2}.
    \end{equation}
    Applying $\|\hat\varphi_1\|_{\rm HS}\leq \beta^{\ell_v}e^{Lr^2}$ (by Lemma~\ref{uvwkey3}), $\|\hat\varphi_2\|_{\rm HS}\leq \beta^{n-\ell_v+1}e^{-\frac{4L^3}{3(n-\ell_v+1)^2}}e^{-2Lc}$ (by Corollary~\ref{ca10}) and $r^2\leq 2c$, we have $\|\hat \Phi_{\sigma_n}\|_{\rm HS}\leq \|\hat\varphi_1\|_{\rm HS}\|\hat\varphi_2\|_{\rm HS}\leq\beta^{n+1}e^{-\frac{4L^3}{3n^2}}$.
\\
    \textbf{Case II: $\sigma_n\in\{a,b,c,d,e,v,u\}^n\setminus\{a,b,c,d,e,v\}^n$.} Define $\ell_u=\min\{i|\sigma_n(i)=u\}\geq 1$, then
    \begin{equation}
        \hat\Phi_{\sigma_n}=\underbrace{U_r^{-1}P_0B_{0,c}\bigg[\prod_{i=1}^{\ell_u-1}K_{\sigma_n(i)}\bigg]P_0U_r^{-1}}_{=:\hat\varphi_1} \underbrace{U_rP_0B_{0,c}\bigg[\prod_{i=\ell_u+1}^{n}K_{\sigma_n(i)}\bigg]P_0U_r}_{=:\hat\varphi_2}.
    \end{equation}
    The result follows by induction assumption.
\end{proof}

\begin{cor}\label{ca17}
    For any $n\in\mathbb Z_{\geq 0}$ and $\sigma_n\in\{a,b,c,d,e,u,v\}^n$, the operator
    \begin{equation}
        \hat \Phi_{\sigma_n}=U_rP_0e^{-L\Delta}B_{0,c}\bigg[\prod_{i=1}^nK_{\sigma_n(i)}\bigg]P_0U_r
    \end{equation}
satisfies $\|\hat \Phi_{\sigma_n}\|_{\rm HS}\leq\beta^{n+1}e^{-\frac{4L^3}{3(n+1)^2}}e^{-2Lc}$.
\end{cor}
\begin{proof}
    The case for $\sigma_n\in\{a,b,c,d,e,v\}^n$ is already proved in Corollary~\ref{ca10}. We only need to consider the case $u\in\sigma_n$. Define     $\ell_u=\min\{i|\sigma_n(i)=u\}\geq 1$, then
    \begin{equation}
        \hat\Phi_{\sigma_n}=\underbrace{U_rP_0e^{-L\Delta}B_{0,c}\bigg[\prod_{i=1}^{\ell_u-1}K_{\sigma_n(i)}\bigg]P_0U_r}_{=:\hat\varphi_1}\underbrace{ U_r^{-1}P_0B_{0,c}\bigg[\prod_{i=\ell_u+1}^nK_{\sigma_n(i)}\bigg]P_0U_r}_{=:\hat\varphi_2}.
    \end{equation}
Applying $\|\hat\varphi_1\|_{\rm HS}\leq \beta^{\ell_u}e^{-\frac{4L^3}{3\ell_u^2}}e^{-2Lc}$ (by Corollary~\ref{ca10}) and $\|\hat\varphi_2\|_{\rm op}\leq\beta^{n-\ell_u+1} $ (by Lemma~\ref{ca14}), we have $\|\hat \Phi_{\sigma_n}\|_{\rm HS}\leq\|\hat\varphi_1\|_{\rm HS}\|\hat\varphi_2\|_{\rm op}\leq \beta^{n+1}e^{-\frac{4L^3}{3n^2}}e^{-2Lc}$.
    \end{proof}
\begin{cor}\label{uvwkey33}
    Let $n\in\mathbb Z_{\geq 0}$, $\sigma_n\in\{a,b,c,d,e,v,u\}^n$, then $\|\hat\Phi_{\sigma_n}\|_{\rm HS}\leq\beta^{n+1}e^{Lr^2}$, where
    \begin{equation}
        \hat\Phi_{\sigma_n}=U_r^{-1}P_0B_{0,c}\bigg[\prod_{i=1}^{n}K_{\sigma_n(i)}\bigg]e^{L\Delta}P_0U_r^{-1}
    \end{equation}
\end{cor}
\begin{proof}
    The case for $\sigma_n\in\{a,b,c,d,e,v\}^n$ is already proved in Lemma~\ref{uvwkey3}. Now consider the case $u\in\sigma_n$, then we have $1\leq\ell_u=\max\{i\mid\sigma_n(i)=u\}\leq n$ and
        \begin{equation}
            \hat\Phi_{\sigma_n}=\underbrace{U_r^{-1}P_0B_{0,c}\bigg[\prod_{i=1}^{\ell_u-1}K_{\sigma_n(i)}\bigg]P_0U_r}_{=:\hat\varphi_1}\cdot \underbrace{U_r^{-1}P_0B_{0,c}\bigg[\prod_{i=\ell_u+1}^{n}K_{\sigma_n(i)}\bigg]e^{L\Delta}P_0U_r^{-1}}_{=:\hat\varphi_2}.
        \end{equation}
    By Corollary~\ref{ca14}, we have $\|\hat\varphi_1\|_{\rm op}\leq \beta^{\ell_u}$. By definition of $\ell_u$, we have $\sigma_n(i)\neq u$ for all $i\geq\ell_u+1$, hence, we can apply Lemma~\ref{uvwkey3} to deduce $\|\hat\varphi_2\|_{\rm HS}\leq\beta^{n-\ell_u+1}e^{Lr^2}.$ Combining together, we then have $
        \|\hat\Phi_{\sigma_n}\|_{\rm HS}\leq\|\hat\varphi_1\|_{\rm op}\|\hat\varphi_2\|_{\rm HS}\leq\beta^{n+1}e^{Lr^2}.$
\end{proof}
\begin{cor}\label{ca15}
     Let $n\in\mathbb Z_{\geq 0}$, $\sigma_n\in\{a,b,c,u\}^n$, then $\|\hat \Phi_{\sigma_n}\|_{\rm HS}\leq \beta^{1+n}e^{Lr^2}$, where
     \begin{equation}
         \hat \Phi_{\sigma_n}=U_r^{-1}P_0B_{0,c}\bigg[\prod_{i=1}^{n}K_{\sigma_n(i)}\bigg]e^{L\tilde\Delta}P_0U_r^{-1}.
     \end{equation}
\end{cor}
\begin{proof}
    The case for $\sigma_n\in\{a,b,c\}^n$ is already proved in Lemma~\ref{uvwkey2}. Let now $\sigma_n\in\{a,b,c,u\}^n$ with $u\in\sigma_n$, then
    \begin{equation}
     \hat\Phi_{\sigma_n}=\underbrace{U_r^{-1}P_0B_{0,c}\bigg[\prod_{i=1}^{\ell_u-1}K_{\sigma_n(i)}\bigg]P_0U_r}_{=:\hat\varphi_1}\underbrace{ U_r^{-1} P_0B_{0,c}\bigg[\prod_{i=\ell_u+1}^{n}K_{\sigma_n(i)}\bigg]e^{L\tilde\Delta}P_0U_r^{-1}}_{=:\hat\varphi_2},
    \end{equation}
    where $\ell_u=\min\{i|\sigma_n(i)=u\}\geq 1$. Since $\sigma_n(i)\neq u$ for $i<\ell_u$, we have $\|\hat\varphi_2\|_{\rm HS}\leq\beta^{n-\ell_u+1}e^{Lr^2}$ (by Lemma~\ref{uvwkey2}), together with $\|\hat\varphi_1\|_{\rm op}\leq \beta^{\ell_u}$ (by Corollary~\ref{ca14}), it holds
    $\|\hat\Phi_{\sigma_n}\|_{\rm HS}\leq \|\hat\varphi_1\|_{\rm op}\|\hat\varphi_2\|_{\rm HS}\leq \beta^{n+1}e^{Lr^2}$.
\end{proof}

\begin{lem}\label{wkey1111}
Let $n\in\mathbb Z_{\geq 0}$,  $\sigma_n\in\{a,b,c,v\}^n$, then $\|\hat \Phi_{\sigma_n}\|_{\rm op}\leq \beta^{n+1}$, where
\begin{equation}
    \hat \Phi_{\sigma_n}=U_r^{-1} P_0\tilde B_{0,c}\bigg[\prod_{i=1}^{n}K_{\sigma(i)}\bigg]P_0 U_r.
\end{equation}
\end{lem}
\begin{proof}
The case $n=0$ follows from~\eqref{2ab3}. We show this via induction on $n\geq 1$ and omit the proof for $n=1$. Consider induction step $n-1\mapsto n.$\\
    \textbf{Case I: $\sigma_n\equiv a$.} Then $\hat\Phi_{\sigma_n}=U_r^{-1} P_0\tilde B_{0,c}^{n+1}P_0 U_r$, the result follows from~\eqref{2ab3}.\\
    \textbf{Case II: $\sigma_n\in\{a,b\}^n\setminus\{a\}^n$.} Define $1\leq \ell_b=\min\{i\mid\sigma_n(i)=b\}\leq n$, then
        \begin{equation}
            \hat \Phi_{\sigma_n}=\underbrace{U_r^{-1} P_0\tilde B_{0,c}^{\ell_b}P_0 U_r}_{=:\hat\varphi_1}\cdot \underbrace{U_r^{-1}P_0\tilde B_{0,c}\bigg[\prod_{i=\ell_b+1}^{n}K_{\sigma(i)}\bigg]P_0 U_r}_{=:\hat\varphi_2}.
        \end{equation}
    Applying now $\|\hat\varphi_2\|_{\rm op}\leq\beta^{\ell_b}$ (by~\eqref{2ab3}) and $\|\hat\varphi_2\|_{\rm op}\leq\beta^{n-\ell_b}$ (induction assumption), we have $\|\hat \Phi_{\sigma_n}\|_{\rm op}\leq\|\hat\varphi_1\|_{\rm op}\|\hat\varphi_2\|_{\rm op}\leq\beta^{n+1}.$\\
    \textbf{Case III: $\sigma_n\in\{a,b,c\}^n\setminus\{a,b\}^n$.} Define $\ell_c=\min\{i\mid\sigma_n(i)=c\}\leq n$, then
    \begin{equation}
            \hat \Phi_{\sigma_n}=\underbrace{U_r^{-1} P_0\tilde B_{0,c}\bigg[\prod_{i=1}^{\ell_c-1}K_{\sigma(i)}\bigg]e^{L\tilde\Delta}P_0U_r^{-1}}_{=:\hat\varphi_1}\cdot \underbrace{U_rP_0e^{-L\Delta}B_{0,c}\bigg[\prod_{i=\ell_c+1}^{n}K_{\sigma(i)}\bigg]P_0 U_r}_{=:\hat\varphi_2},
        \end{equation}
        Applying $\|\hat\varphi_1\|_{\rm HS}\leq\beta^{\ell_c}e^{Lr^2}$ (by Lemma~\ref{wkey}), $\|\hat\varphi_2\|_{\rm HS}\leq \beta^{n-\ell_c+1}e^{-\frac{4L^3}{3(n-\ell_c+1)^2}}e^{-2Lc}$ (by Lemma~\ref{ca17}) and $r^2\leq 2c$, we have
        \begin{equation}
            \| \hat \Phi_{\sigma_n}\|_{\rm op}\leq \| \hat \Phi_{\sigma_n}\|_{\rm HS}\leq\|\hat\varphi_1\|_{\rm HS}\|\hat\varphi_2\|_{\rm HS}\leq\beta^{n+1}.
        \end{equation}
    \textbf{Case IV: $\sigma_n\in\{a,b,c,v\}^n\setminus\{a,b,c\}^n.$} Define $\ell_v=\min\{i\mid\sigma_n(i)=v\}\leq n$, then
    \begin{equation}
        \hat \Phi_{\sigma_n}=\underbrace{U_r^{-1} P_0\tilde B_{0,c}\bigg[\prod_{i=1}^{\ell_v-1}K_{\sigma_n(i)}\bigg]e^{L\Delta}P_0U_r^{-1}}_{=:\hat\varphi_1}\cdot \underbrace{U_rP_0e^{-L\Delta}B_{0,c}\bigg[\prod_{i=\ell_v+1}^{n}K_{\sigma_n(i)}\bigg]P_0 U_r}_{=:\hat\varphi_2}.
    \end{equation}
    By definition of $\ell_v$, we have $\sigma_n(i)\in\{a,b,c\}$ for any $1\leq i\leq\ell_v-1$. Hence, we can apply Lemma~\ref{vwkey1} to deduce $\|\hat\varphi_1\|_{\rm HS}\leq\beta^{\ell_v}e^{Lr^2}$. Together with $\|\hat\varphi_2\|_{\rm HS}\leq\beta^{n-\ell_v+1}e^{-2Lc}.$ (by Corollary~\ref{ca17}) and $r^2\leq 2c$, we have
    \begin{equation}
        \|\hat \Phi_{\sigma_n}\|_{\rm op}\leq \|\hat \Phi_{\sigma_n}\|_{\rm HS}\leq \|\hat\varphi_1\|_{\rm HS}\|\hat\varphi_2\|_{\rm HS}\leq\beta^{n+1}.
    \end{equation}
\end{proof}

\begin{cor}\label{ca16}
     Let $n\in\mathbb Z_{\geq 0},\sigma_n\in\{a,b,c,u,v\}^n$, then $\|\hat \Phi_{\sigma_n}\|_{\rm HS}\leq\beta^{n+1}e^{Lr^2}$, where
     \begin{equation}
         \hat \Phi_{\sigma_n}=U_r^{-1}P_0\tilde B_{0,c}\bigg[\prod_{j=1}^{n}K_{\sigma_n(j)}\bigg]e^{L\Delta}P_0U_r^{-1}.
     \end{equation}
\end{cor}
\begin{proof}
    The case for $\sigma_n\in\{a,b,c\}^n$ is proved in Lemma~\ref{vwkey1}. We show the rest cases via induction and omit the details for $n=1$. Consider induction step $n-1\mapsto n$. \\
    \textbf{Case I: $\sigma_n\in\{a,b,c,v\}^n\setminus\{a,b,c\}^n.$} Define $\ell_v=\min\{i|\sigma_n(i)=v\}\geq 1$, then
    \begin{equation}
        \hat\Phi_{\sigma_n}=\underbrace{U_r^{-1}
        P_0\tilde B_{0,c}\bigg[\prod_{j=1}^{\ell_v-1}K_{\sigma_n(j)}\bigg]e^{L\Delta}P_0U_r^{-1}}_{=:\hat\varphi_1}  \underbrace{U_rP_0e^{-L\Delta}B_{0,c}\bigg[\prod_{j=\ell_v+1}^{n}K_{\sigma_n(j)}\bigg]e^{L\Delta}P_0U_r^{-1}}_{=:\hat\varphi_2}.
    \end{equation}
    Applying $\|\hat\varphi_1\|_{\rm HS}\leq \beta^{\ell_v}e^{Lr^2} $ (induction assumption) and $\|\hat\varphi_2\|_{\rm op}\leq\beta^{n-\ell_v+1}$ (by Lemma~\ref{uvwkey1}), we have $\|\hat \Phi_{\sigma_n}\|_{\rm HS}\leq\|\hat\varphi_1\|_{\rm HS}\|\hat\varphi_2\|_{\rm op}\leq\beta^{n+1}e^{Lr^2}$.\\
    \textbf{Case II: $\sigma_n\in\{a,b,c,v,u\}^n\setminus\{a,b,c,v\}^n.$} Define $\ell_u=\min\{i\mid\sigma_n(i)=u\}$, then
    \begin{equation}
         \hat \Phi_{\sigma_n}=\underbrace{U_r^{-1}P_0\tilde B_{0,c}\bigg[\prod_{j=1}^{\ell_u-1}K_{\sigma_n(j)}\bigg]P_0U_r}_{=:\hat\varphi_1}\cdot \underbrace{U_r^{-1}P_0B_{0,c}\bigg[\prod_{j=\ell_u+1}^{n}K_{\sigma_n(j)}\bigg]e^{L\Delta}P_0U_r^{-1}}_{=:\hat\varphi_2}.
    \end{equation}
    By definition of $\ell_u$, $\sigma_n(j)\in\{a,b,c,v\}$ for any $j\leq\ell_u-1$, we can then apply Lemma~\ref{wkey1111} to deduce $\|\hat\varphi_1\|_{\rm op}\leq\beta^{\ell_u}.$ Together with $\|\hat\varphi_2\|_{\rm HS}\leq\beta^{n-\ell_u+1}e^{Lr^2}$ (by Corollary~\ref{uvwkey33}), we have $\| \hat \Phi_{\sigma_n}\|_{\rm HS}\leq \|\hat\varphi_1\|_{\rm op}\|\hat\varphi_2\|_{\rm HS}\leq\beta^{n+1}e^{Lr^2}.$

\end{proof}

%\bibliographystyle{patplain}
%\bibliography{Biblio}

\end{document}